%% file: witpaper-master.tex
\documentclass[11pt,oneside]{amsart}  
\input{witpaper-macro.tex}

\author[Beaudry]{Agn\`es Beaudry}
\address{Department of Mathematics, University of Colorado, Boulder, CO 80309, USA}
\email{agnes.beaudry@colorado.edu}
\author[Lewis]{Chloe Lewis}
\address{Department of Mathematics, University of Wisconsin-Eau Claire, Eau Claire, WI 54701, USA}
\email{lewischl@uwec.edu}
\author[May]{Clover May}
\address{Department of Mathematical Sciences\\NTNU\\ 7491 Trondheim, Norway}
\email{clover.may@ntnu.no}
\author[Pauli]{Sabrina Pauli}
\address{Fachbereich Mathematik\\TU Darmstadt\\64289 Darmstadt, Germany}
\email{pauli@mathematik.tu-darmstadt.de}
\author[Tatum]{Elizabeth Tatum}
\address{Mathematisches Institut\\Universität Bonn\\53115 Bonn, Germany}
\email{tatum@math.uni-bonn.de}

\title[Guide to Parametrized Cohomology]{A Guide to Equivariant Parametrized Cohomology}

\begin{document}

\begin{abstract}
This article investigates equivariant parametrized cellular cohomology, a cohomology theory introduced by Costenoble--Waner for spaces with an action by a compact Lie group $G$.  The theory extends the $\RO(G)$-graded cohomology of a $G$-space $B$ to a cohomology graded by $\RO(\Pi B)$, the representations of the equivariant fundamental groupoid of $B$. This paper is meant to serve as a guide to this theory and contains some new computations.

We explain the key ingredients for defining parametrized cellular cohomology when $G$ is a finite group, with particular attention to the case of the cyclic group $G=C_2$. We compute some examples and observe that $\RO(\Pi B)$ is not always free.  When $G$ is the trivial group, we explain how to identify equivariant parametrized cellular cohomology with cellular cohomology in local coefficients. 
Finally, we illustrate the theory with some new computations of  parametrized cellular cohomology for several spaces with $G = C_2$ and $G=C_4$.
\end{abstract}

\maketitle

\tableofcontents

\input{witpaper-intro}

\input{witpaper-fun}

\input{witpaper-para}

\input{witpaper-coh}

\input{witpaper-examples}


\providecommand{\bysame}{\leavevmode\hbox to3em{\hrulefill}\thinspace}
\providecommand{\MR}{\relax\ifhmode\unskip\space\fi MR }
\providecommand{\MRhref}[2]{%
  \href{http://www.ams.org/mathscinet-getitem?mr=#1}{#2}
}
\providecommand{\href}[2]{#2}

\end{document}

%% file: witpaper-macro.tex
\usepackage{amsmath}
\usepackage{amsfonts}
\usepackage{amssymb}
\usepackage{amsthm}
\usepackage{mathrsfs}
\usepackage{mathtools}
\usepackage{dsfont}
\usepackage{tikz-cd}
\usepackage[all]{xy}
\usepackage{pdfpages}
\usepackage{enumerate}
\usepackage{float}
\usepackage{todonotes}
\usepackage{graphicx}
\usepackage{xcolor}
\usepackage[colorlinks,linkcolor=blue,urlcolor=blue,citecolor=blue,destlabel=true]{hyperref}
\definecolor{darkspringgreen}{rgb}{0.09, 0.6, 0.1}

\usepackage[capitalise]{cleveref}
\numberwithin{equation}{section}

\newtheorem{thm}{Theorem}[section]
\newtheorem*{thm*}{Theorem}
\newtheorem{theorem}[thm]{Theorem}
\newtheorem{cor}[thm]{Corollary}

\newtheorem{lem}[thm]{Lemma}

\theoremstyle{definition}
\newtheorem{defn}[thm]{Definition}
\newtheorem*{defn*}{Definition}
\newtheorem{definition}[thm]{Definition}

\newtheorem{notation}[thm]{Notation}

\newtheorem{example}[thm]{Example}

\newtheorem{ex}[thm]{Example}

\newtheorem{rem}[thm]{Remark}
\newtheorem{remark}[thm]{Remark}
\newtheorem{warn}[thm]{Warning}

\newtheorem{convention}[thm]{Convention}

\newcommand{\Ab}{\mathrm{Ab}}

\newcommand{\id}{\mathrm{id}}

\DeclareMathOperator{\Rep}{Rep}

\DeclareMathOperator*{\colim}{colim}

\DeclareMathOperator{\Hom}{Hom}

\newcommand{\op}{\mathrm{op}}

\newcommand{\mF}{\underline{\mathbb{F}}_2}

\newcommand{\res}{\mathrm{res}}

\newcommand{\Z}{\mathbb{Z}}
\newcommand{\R}{\mathbb{R}}

\newcommand{\F}{\mathbb{F}}
\newcommand{\pt}{\mathrm{pt}}

\newcommand{\cO}{\mathcal{O}}
\newcommand{\Top}{\mathrm{Top}}

\newcommand{\upi}{\underline{\pi}}

\newcommand{\RO}{RO}

\newcommand{\Nat}{\mathrm{Nat}}

\newcommand{\cV}{\mathcal{V}}
\newcommand{\vV}{v\mathcal{V}}
\newcommand{\Th}{\mathrm{Th}}
\newcommand{\mono}[1]{m_{#1}}

\newcommand{\wH}{\widetilde{H}}

\newcommand{\CWg}{\text{CW}(\gamma)}

\newcommand{\lax}{\mathrm{lax}}

\newcommand{\PTop}[2]{\mathrm{Top}_{/#2}^{#1}}
\newcommand{\PTopb}[2]{\mathrm{Top}_{#2}^{#1}}

\newcommand{\hPTopb}[2]{\mathrm{hTop}_{#2}^{#1}}

\newcommand{\hPTop}[2]{\mathrm{hTop}_{/#2}^{#1}}

\newcommand{\pPTopb}[2]{\pi\mathrm{Top}_{#2}^{#1}}

\newcommand{\pPTop}[2]{\pi\mathrm{Top}_{/#2}^{#1}}

\newcommand{\pPTopblax}[2]{\pi\mathrm{Top}_{#2,\mathrm{lax}}^{#1}}

\newcommand{\PSH}[2]{\mathcal{SH}_{#2}^{#1}}
\newcommand{\ldeg}[2]{\varepsilon_{#2}}

\newcommand{\tr}{\mathrm{tr}}

\newcommand{\uM}{\underline{M}}
\newcommand{\uN}{\underline{N}}
\newcommand{\uC}{\underline{C}}
\newcommand{\uZ}{\underline{\mathbb{Z}}}

\newcommand{\cC}{\mathcal{C}}

\newcommand{\Gpds}{\mathrm{Gpds}}

\newcommand{\taut}{\mathcal{L}}
\newcommand{\pars}[1]{\left( #1 \right)}

\newcommand{\RPtwist}{\mathbb{P}(\mathbb{R}^{3,1})}

\newcommand{\Fin}{\mathrm{Fin}}
\newcommand{\Span}{\mathrm{Span}}

\newcommand{\FinG}{\mathrm{Fin}^G}

\newcommand{\sorb}[1]{\mathcal{B}^{#1}}
\newcommand{\Burn}[1]{\mathrm{Burn}^{#1}}

\newcommand{\BBurn}[2]{\mathrm{Burn}^{#1}_{#2}}
\newcommand{\Bsorb}[2]{\mathcal{B}^{#1}_{#2}}
\newcommand{\BFin}[2]{\mathrm{Fin}^{#1}_{#2}}

\newcommand{\period}    {{\makebox[0pt][l]{\hspace{2pt} .}}}

%% file: witpaper-intro.tex

\section{Introduction}

Computations in $\RO(G)$-graded homotopy and (co)homology have been fundamental to modern equivariant homotopy theory, despite their complexity. One reason is that $\RO(G)$-graded invariants capture quite a large amount of equivariant information.  Yet in general, ordinary $\RO(G)$-graded cohomology theories lack equivariant analogues of characteristic classes, such as Thom classes for $G$-vector bundles and fundamental classes for $G$-manifolds.  To remedy this failure,  Costenoble--Waner \cite{CW_book} introduced an extension they called $\RO(\Pi B)$-graded equivariant cohomology, building on \cite{CW_Duality},\cite{CW_Thom}, and \cite{CMW}.  This cohomology theory is graded on representations of the equivariant fundamental groupoid $\Pi B$ of a $G$-space $B$, and is built to witness characteristic classes.

Classically, the Thom isomorphism states: given a rank $n$ vector bundle $\xi \colon E \to B$, the cohomology of the base space $B$ with coefficients in $\F_2$ is isomorphic to a shift of the reduced cohomology of the Thom space of $E$.  The isomorphism is obtained by cupping with the Thom class in degree $n$.  To incorporate integral coefficients and still allow for unoriented bundles, one requires cohomology with local coefficients.  Equivariantly, there is a Thom isomorphism for ordinary $\RO(G)$-graded cohomology only when the fibers of the $G$-bundle are isomorphic to a constant representation $V$, also called a $V$-bundle. In this case, the Thom class lives in dimension $V \in \RO(G)$. However, there are many important equivariant bundles where the fibers are isomorphic to different representations; take for instance the tautological bundles on projective spaces for $G=C_2$, the cyclic group of order two. Even here we do not have an $\RO(G)$-graded Thom isomorphism. A major hurdle in attempting to write down such an isomorphism is that there is no clear choice of dimension for the Thom class.  See for example \cite{CW_Thom}, \cite{Hazel_fund}, and \cite{BhattZou} for related discussions. The dimension of the fundamental class poses a similar problem for Poincaré duality.

Extending to the $\RO(\Pi B)$-grading of Costenoble--Waner resolves this problem.  Each bundle gives rise to a dimension in $\RO(\Pi B)$.  Moreover, $\RO(\Pi B)$ allows for different representations over different components of the fixed-set.  In \cite{CW_book}, Costenoble and Waner showed that for $G$ a compact Lie group, their $\RO(\Pi B)$-graded cohomology theory exhibits Poincaré duality for any smooth $G$-manifold and a Thom isomorphism for any $G$-vector bundle. However, the theory described in \cite{CW_book} is very technical and few computations are known.  Fairly recently, there have been some computations for $G=C_2$, of the classifying space $B_{C_2}U(1)$ \cite{CostenobleB} and of finite complex projective spaces \cite{CHT_C2}.  There are, to our knowledge, no known computations using an explicit cellular cochain complex and almost no known computations for larger groups.

In this paper, our goal is to give an accessible and reasonably self-contained presentation of the cellular cohomology theory from \cite{CW_book} in order to facilitate computations.  To simplify the story greatly, we let $G$ be a finite group, and we include a number of examples for $G=C_2$.  Unlike the previous work in \cite{CMW}, \cite{CW_Duality}, and \cite{CW_Thom}, the cohomology theory described in \cite{CW_book} is defined for $G$-spaces over a base $B$. This version of an $\RO(\Pi B)$-graded cohomology theory is represented by equivariant parametrized spectra, thus we refer to it as \textit{equivariant parametrized cohomology}. We do not always use the same notation\footnote{
For example, throughout \cite{CW_book} the reader will find references to dimension functions $\delta$. These are only needed for non-finite compact Lie groups, to handle the fact that  orbits are not self-dual. Since we are only considering finite groups, the dimension functions $\delta$ are always zero, and so we leave $\delta$ out of our notation.} or terminology as \cite{CW_book}, since our aim is to explain the theory to someone familiar with $\RO(G)$-graded cohomology.  For a recent overview of $\RO(G)$-graded equivariant homotopy theory, see \cite{HillHandbook}.

In what follows, we describe the key pieces needed to define equivariant parametrized cohomology.  We define $\RO(\Pi B)$ and its representations, look at the CW-structures twisted by representations of $\RO(\Pi B)$, and explain the construction of the parametrized cellular cochain complex. 
Along the way, we compute pieces of the theory for several new examples.  In particular, we compute $\RO(\Pi B)$ for a $C_2$-twisted real projective space and demonstrate that, unlike $\RO(G)$ which is always free abelian, the extended grading may have torsion.

\begin{lem}[\emph{c.f.} \cref{lem:ROpiRP2}]
Let $\R^{3,1}$ be the direct sum of a $2$-dimensional trivial representation and the $1$-dimensional sign representation. 
Consider $B = \RPtwist$, the projective space of  $\R^{3,1}$ (often called ``$\R P^2$-twist''). 
    There is an isomorphism 
    \[RO\pars{\Pi \RPtwist}\cong \Z^3\times \Z/2.\]
    \end{lem}

Given $\gamma \in \RO(\Pi B)$, a representation of the equivariant fundamental groupoid, and a $\CWg$-structure on $B$, the equivariant parametrized cellular cohomology of $B$ takes coefficients in a parametrized Mackey functor.  We define parametrized Mackey functors and explicitly describe the parametrized cellular cochain complex $C^{\gamma+*}(B;\uM)$ and its coboundary $d$ that gives the cellular cohomology groups
\[ H^{\gamma+*}(B;\uM) = H^*(C^{\gamma+*}(B;\uM),d).\] 
These groups extend to an $\RO(\Pi B)$-graded cohomology theory, and when $\gamma$ comes from a $V$-bundle on $B$, then $H^{\gamma}(B;\uM)\cong H_G^{V}(B;\uM)$ agrees with ordinary $\RO(G)$-graded Bredon cohomology.  By the Yoneda lemma, the cochain complex in degree $\gamma + n$ is the direct sum over the centers of the $n$-cells of the Mackey functor applied to these centers.  The difficult part for computations is identifying the coboundary in the cochain complex.

The first obvious question is what this theory captures in the case of the trivial group. In \cite{CW_Duality}, Costenoble--Waner present a model for cellular $RO(\Pi B)$-graded cohomology that uses an equivariant universal cover. They remark in \cite[Rmk~4.7.3]{CW_Duality} that their theory generalizes cohomology with local coefficients for the trivial group and this is clear from the universal cover approach. However, the approach in \cite{CW_book} is different and the comparison with local coefficients is not made in the newer text. So using the explicit description of the cellular cochain complex, we carefully make this connection in \cref{sec:trivialgroupcohomology}.

\begin{thm}[\emph{c.f.} \cref{thm:localispofinal}]
Let $G=e$ be the trivial group and $\uN$ be a constant parmetrized Mackey functor.
 The equivariant parametrized cellular cohomology of a path-connected CW-complex $B$ in degree $\gamma$ is isomorphic to the cellular cohomology of $B$ with coefficients twisted by $\gamma$ in degree  $|\gamma|$
\[H^{\gamma}(B;\underline{N}) \cong H^{|\gamma|}(B;N_\gamma)\]
where  $|\gamma|$ is the underlying dimension of $\gamma$.
\end{thm}

Returning to the equivariant case, we use the explicit description of the equivariant parametrized cellular cochains to compute some parametrized cohomology groups for the following $G$-spaces:
\begin{enumerate}[(1)]
\item the $C_2$-disk $D(\R^{1,1})$, where $\R^{1,1}$ is the $1$-dimensional sign representation,
\item the $C_2$-circle $S^{1,1}$ obtained as the one-point compactification of $\R^{1,1}$, 
\item the $C_2$-projective space $\RPtwist$ of 
the representation $\R^{3,1}$, 
\item the $C_4$-projective space of the 3-dimensional  representation which is a direct sum of a trivial representation and the representation $\R^2$ with rotation by 90 degrees.
\end{enumerate}
We do not carry out any complete $\RO(\Pi B)$-graded computations here, but we compute cohomology in some  degrees beyond the $RO(G)$-grading. Our goal is to demonstrate the theory and how to apply it.

\subsection{Literature}
In this paper we mainly follow Costenoble--Waner  \cite{CW_book}, though we also look to their previous work in \cite{CW_Duality, CW_Thom}, work of Costenoble--May--Waner \cite{CMW}, and the work on parametrized homotopy theory of May--Sigurdsson \cite{MaySig}. We also point the reader to more recent treatments of parametrized homotopy theory by Malkiewich \cite{malkiewich2023parametrizedlowtech, Malkiewich_Convenient}.

Other people have also thought about equivariant cohomology with local coefficients, for example, both Mukherjee--Mukherjee in \cite{Mukherjee} and Moerdijk--Svensson \cite{Moerdijk}. 
More recently, there is work of Mukherjee--Sen \cite{MukherjeeSen} and Basu--Sen \cite{BasuSen} extending \cite{Mukherjee}. 
Perhaps the most similar treatment to the approach of Costenoble--Waner can be found in Shulman's thesis \cite[\S 4]{Shulman}, which appears to be closely related to the approach in the original paper \cite{CW_Duality} in the context of $RO(G)$-graded cohomology.

The idea of recording cohomology with local coefficients in an extended grading of regular cohomology also appeared in \v{C}adek in \cite{Cadek}. 
In this work, without phrasing it as such, \v{C}adek computes the nonequivariant $\RO(\Pi \R P^\infty)$-graded cohomology of $\R P^\infty$. Noting that 
\[\RO(\Pi \R P^\infty)\cong \Z \times \Z/2,\] 
the parametrized cohomology is 
\[ H^{*,*}(\R P^\infty ;\Z) \cong \Z[a]/2a\]
for a class $a \in H^{1,1}(\R P^\infty ;\Z)$ where here $(*,*) \in  \Z \times \Z/2$. In fact, \v{C}adek does the computation for $BO(n)$ for all $n$, and
Lastovecki treats products of $BO(n)$'s in \cite{Lastovecki}.

We note that none of these articles cite any of Costenoble--Waner's work, and we did not find a comparison of the various approaches.  
Our goal is \emph{not} to give a comprehensive treatment comparing the different approaches in the literature, but rather to give an expository treatment of certain aspects of \cite{CW_book} in order to facilitate computations.

\subsection{Organization}
In \cref{sec:fun} we explain the construction of $\RO(\Pi B)$, and compute various examples. In \cref{sec:para} we review background on parametrized homotopy theory that gets used in the construction of cellular parametrized cohomology. \cref{sec:coh} then turns to cohomology. We first review the definition of $\CWg$-structures and the associated cellular chain complexes.  We then give an explicit description of the cellular cochain complex, and carefully explain the connection between this construction and cohomology with local coefficients in the case of the trivial group.
In \cref{sec:ex} we turn to the computation of equivariant parametrized cohomology groups for some examples with $G=C_2$ and $G=C_4$.

\subsection{Acknowledgements}
We would like to thank the following people for useful conversations: Prasit Bhattacharya, Anna Marie Bohmann, Thomas Brazelton, Steven Costenoble, Michael Hill, Eric Hogle, Inbar Klang, Cary Malkiewich, Peter May, Vesna Stojanoska, Stefan Waner, and Foling Zou. We thank the anonymous referee for their helpful comments and suggestions. 
We also thank the Women in Topology (WIT) network for making this project happen, and the Hausdorff Research Institute for Mathematics for their hospitality during the WIT IV conference.

This material is based upon work supported by the National Science Foundation under Grant No. DMS 2143811 and DMS 2135960.
Additionally, Clover May is supported by grant number TMS2020TMT02 from the Trond Mohn Foundation. Sabrina Pauli acknowledges support by Deutsche Forschungsgemeinschaft (DFG, German Research Foundation) through the Collaborative Research Centre TRR 326 \textit{Geometry and Arithmetic of Uniformized Structures}, project number 444845124.

%% file: witpaper-fun.tex

\section{The Fundamental groupoid and its representations }\label{sec:fun}
Parametrized equivariant cohomology is graded on $RO(\Pi B)$. 
 The goal of this section is to explain the group $\RO(\Pi B)$ and compute it in a few examples. 
 
\subsection{The equivariant fundamental groupoid}
In this section, we review the definition of the equivariant fundamental groupoid $\Pi B$. This construction generalizes the classical fundamental groupoid in the sense that the definitions coincide when the group $G$ is trivial.

We first give the definition of the equivariant fundamental groupoid as a category fibered in groupoids over the orbit category $\cO_G$, as described in \cite{CW_book, CMW, CW_Duality}, and originally due to tom Dieck \cite{td}. 

\begin{definition}
    The \emph{orbit category} $\cO_G$ is the category with objects $G/H$ for subgroups $H\subset G$, and morphisms maps of $G$-sets.
\end{definition}

\begin{example}\label{ExOrbitCats}
        For $G=C_2=\langle \tau \rangle$, there are two objects $C_2/e$ and $C_2/C_2$ in $\cO_{C_2}$. The non-identity morphisms are depicted in the following diagram.,
\[\xymatrix{
 C_2/C_2  \\
 C_2/e  \ar[u]_-\rho \ar@(dr,dl)[]^{\tau}
}\]
where the map $\tau$ indicates multiplication by the element $\tau$. 
The morphisms satisfy the relation $\rho \circ \tau = \rho$.  We often denote the identity morphism at $C_2/e$ by $e$ rather than $\id$.
\end{example}

We now turn to the definition of the equivariant fundamental groupoid. For homotopies $\omega_1$ and $\omega_2$, we write $\omega_2 \star \omega_1$ for the concatenation of the homotopy $\omega_1$ followed by the homotopy $\omega_2$, as in composition of functions. We adopt the same convention in the special case when the homotopies are paths. 

The definition of a category fibered in groupoids is due to Grothendieck \cite[Expos\'e 6, pp.165-166]{Grothendieck_SGA1}. For an exposition in this context, see \cite[Definition 5.1]{CMW}.

\begin{definition}\label{def:unstraight fungroupoid}
Let $B$ be a $G$-space. 
    The \emph{equivariant fundamental groupoid of $B$} is the category $\Pi B$ fibered in groupoids over $\cO_G$
\[\phi\colon \Pi B\rightarrow \cO_G\] 
where
    \begin{enumerate}[(a)]
        \item objects of $\Pi B$ are $G$-maps $x\colon G/H\rightarrow B$,
        \item morphisms of $\Pi B$ from $x\colon G/H\rightarrow B$ to $y\colon G/K\rightarrow B$ are pairs $(\alpha, \omega)$ where $\alpha\colon G/H\rightarrow G/K$ is a $G$-map (i.e.\ a morphism in $\cO_G$) and 
        \[\omega\colon G/H\times I\rightarrow B\] 
        is a $G$-homotopy from $x$ to $y\circ \alpha$. Furthermore, if $\omega$ is homotopic to $\omega'$ relative to the endpoints $G/H \times\{0,1\} \to B$, then we identify the morphisms $(\alpha,\omega)=(\alpha, \omega')$. We will often depict the data of a morphism by 
        \[\xymatrix{ G/H \ar[r]^-\alpha \ar[d]_-x \ar@{=>}[dr]^-{\omega} & G/K \ar[d]^-y \\ 
        B \ar@{=}[r] & B
        \period } \]
        \item Composition is given by
        \[ (\alpha,\omega)\circ (\alpha',\omega') = (\alpha\alpha', \omega(\alpha' \times \id_I)  \star \omega' ).\] 
        \item The functor $\phi \colon \Pi B \to \cO_G$ sends $x\colon G/H\rightarrow B$ to $G/H$ and a morphism $(\alpha,\omega)$ to the $G$-map $\alpha$.
    \end{enumerate}
\end{definition}
Note that the fiber of $\phi$ over $(G/H,\id_{G/H})$ can be identified with the nonequivariant fundamental groupoid of the $H$-fixed points $\Pi B^H$.

\begin{rem}
In much of the literature, including \cite{CW_book} and \cite{CMW}, the equivariant fundamental groupoid of a $G$-space $B$ is denoted by $\Pi_GB$. We omit the $G$ in the notation.  Instead, we are careful to use the restriction functor $i^*_H$ when we want to restrict to a subgroup $H$.
\end{rem}

\begin{rem}
By Elmendorf's Theorem, up to weak homotopy equivalence, a $G$-space $B$ is the same as the data of a functor 
\[B \colon \cO_G^\op \to \Top.\] 
Let $\Pi \colon \Top \to \Gpds$ denote the functor from topological spaces to groupoids that takes a space to its fundamental groupoid. 
The Grothendieck construction applied to the composite
    \[\xymatrix{\cO_G^\op \ar[r]^-{B} & \Top  \ar[r]^-{\Pi}  & \Gpds} \]
is precisely the category fibered in groupoids $\Pi B$ of \cref{def:unstraight fungroupoid}.
The Grothendieck construction is an equivalence of categories, so this is an equivalent perspective. 
 We refer the reader to Pronk and Scull \cite{pronk2019equivariant} for a detailed treatment, as well a precise formulation of the Grothendieck construction in this context.
 \end{rem}

\begin{example}
For a nonequivariant space $B$, i.e.\ when $G=e$ is trivial, $\Pi B$ is just the fundamental groupoid $\Pi B$ of the space $B$. 
\end{example}

\begin{remark}
\label{remark: how to think of PiGB}
    A $G$-map $x\colon G/H\rightarrow B$ is completely determined by its value at the identity coset $eH$. Furthermore, for $x$ to be a $G$-map, the image of $eH$ under $x$ has to be contained in $B^H$. So an object in $\Pi B$ is determined by a point $x(eH) \in B^H$.  Thus we also refer to the map $x\colon G/H\rightarrow B$ as a ``point'' in $B$.

Let $x\colon G/H\rightarrow B$ and $y\colon G/K\rightarrow B$, and let $\alpha\colon G/H\rightarrow G/K$.  A homotopy $\omega\colon G/H\times I\rightarrow B$ from $x\colon G/H\rightarrow B$ to $y\circ\alpha\colon G/H\rightarrow B$ is determined by a path in $B^H$ from $x(eH)$ to $y(\alpha(eH))$. Thus the set of morphisms 
    \[ \Pi B(x,y) \cong \coprod_{\alpha \colon G/H \to G/K}\Pi B^H (x(eH), y(\alpha(eH)) ).\] 
\end{remark}

We will often need to move between $\Top^G$ and $\Top^H$ for $H\subset G$ a subgroup. To do this, we will use the adjunction
\[\xymatrix{ G\times_H - \colon \Top^H \ar@<1ex>@{->}[r] & \Top^G \ar@<1ex>@{->}[l] \colon i^*_H } \]
where $i^*_HB$ is the $H$-space underlying a $G$-space $B$ and $G\times_H  C$ is the quotient of $G\times C$ by $(gh, c) \simeq (g,hc)$ for an $H$-space $C$.

The adjunction gives rise to the following change-of-group functor for fundamental groupoids.
\begin{defn}[Change-of-Groups]\label{defn:cog-PiB}
Let $H\subset G$ be a subgroup. Given a $G$-space $B$, define
\[ G\times_H  -  \colon \Pi  i^*_HB \to \Pi B\]
to be the functor which sends a point $b \colon H/K \to i^*_HB$ to $G\times_H b \colon G/K \to B$ and a morphism $(\alpha,\omega) \colon b \to b'$ to $(G\times_H\alpha, G\times_H\omega)$. 
\end{defn}

\subsection{Examples of $\Pi B$}
To describe concrete examples of the equivariant fundamental groupoid, it will be practical to describe a skeleton, that is an equivalent subcategory where no two objects are isomorphic.
By \cref{remark: how to think of PiGB}, two objects $x\colon G/H\rightarrow B$ and $y\colon G/K\rightarrow B$ are isomorphic in $\Pi B$ if and only if there is an isomorphism $\alpha \colon G/H \to G/K$ and a path from $x(eH)$ to $y(\alpha(eH))$ in $B^H$.
By abuse of notation, we will denote a skeleton of $\Pi B$ by $\Pi B$ as well. 

\begin{rem}
    When working with $G=C_2$, we use the motivic notation $\R^{p,q}$ for $C_2$-representations.  We write $\R^{p,q}=\R^{(p-q) + q\sigma}$ for the direct sum of $p-q$ copies of the one-dimensional trivial representation and $q$ copies of the one-dimensional sign representation $\sigma$. The one-point compactification of $\R^{p,q}$ is a $p$-dimensional $C_2$-sphere denoted $S^{p,q}$.
\end{rem}

\begin{example}\label{ex:funS11} 
We start with $B$ the $C_2$-space $S^{1,1}=S^\sigma$. 
This is \cite[Example~3.4]{pronk2019equivariant}. We will use this example extensively in the rest of the paper, so we redo the computation here and choose a skeleton for $\Pi S^{1,1}$.  We write $C_2 = \langle \tau \rangle$ and describe the functor $\phi \colon \Pi S^{1,1} \rightarrow \cO_{C_2}$ (see \cref{ExOrbitCats} and \cref{def:unstraight fungroupoid}).

First we consider the objects of $\Pi S^{1,1}$.  The underlying space $i^*_eS^{1,1}=S^1$ is path-connected and thus, up to isomorphism, there is only one object of the form $C_2/e\to S^{1,1} $. For the skeleton of $\Pi S^{1,1}$, we fix the object $b\colon C_2/e \to S^{1,1}$ which sends $e$ to $1$ and $\tau$ to $-1$. 
On the other hand, the fixed set
    $(S^{1,1})^{C_2}=\{0,\infty\}$ has two path components.  For the skeleton, we choose the two non-isomorphic objects $b_0\colon C_2/C_2\to  S^{1,1}$ with $b_0(C_2/C_2)=0$ and $b_1\colon C_2/C_2\to S^{1,1}$ with $b_1(C_2/C_2)=\infty$.
See \cref{fig:S11}, where the $C_2$-action on the circle is reflection across the vertical axis.

\begin{figure}[h]
\begin{center}
\includegraphics[width=0.5\textwidth]{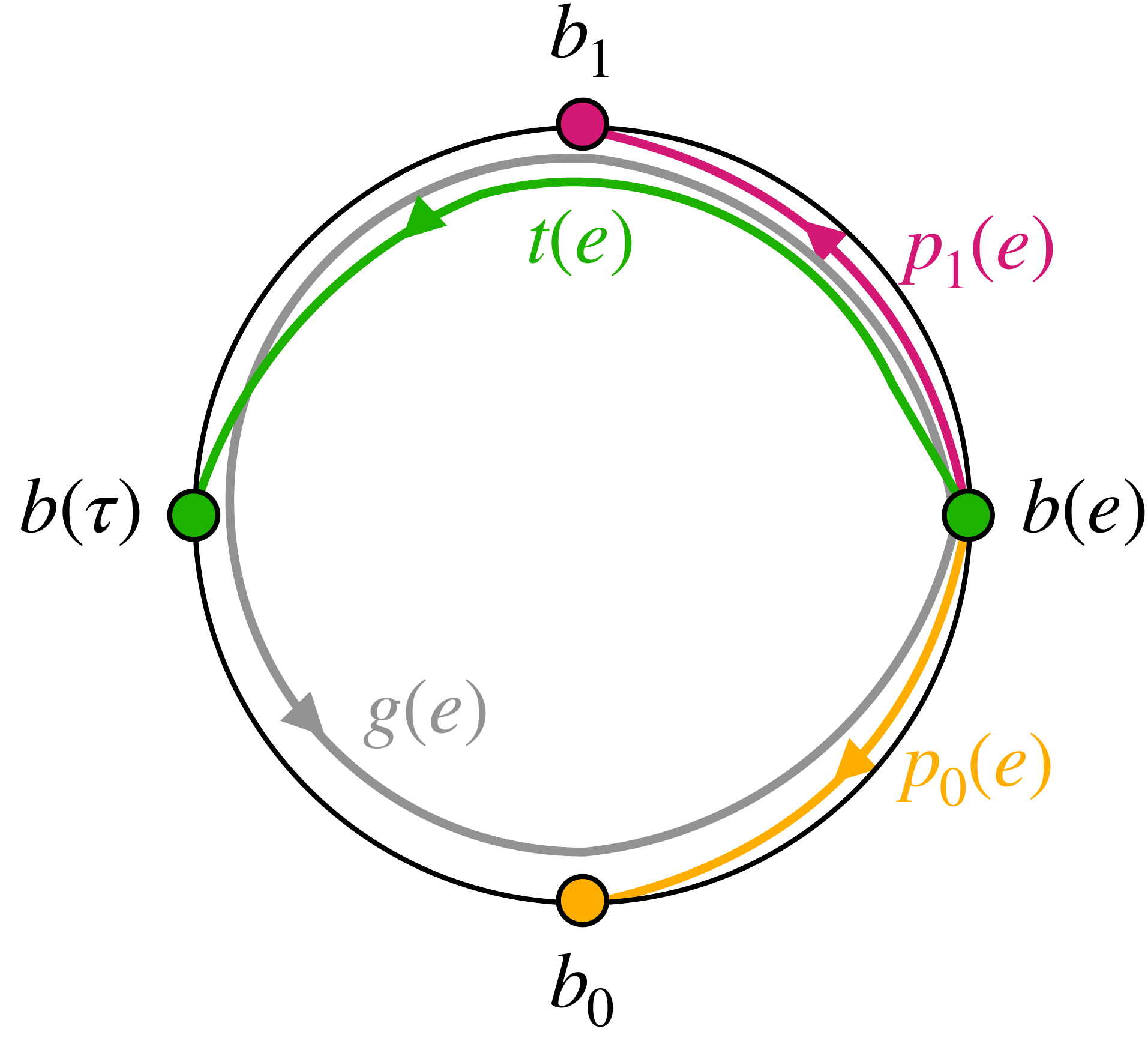}
\caption{Some objects and morphisms in $\Pi S^{1,1}$}
\label{fig:S11}
\end{center}
\end{figure}

    Next we find the morphisms in $\Pi S^{1,1}$ between the objects $b$, $b_0$ and $b_1$ using  \cref{remark: how to think of PiGB}.  We use the following diagram to keep track of this data.
\[\xymatrix@R=1.5pc{
&  \Pi S^{1,1} & & \ar[rrr]^-\phi  & & & & \cO_{C_2}\\
b_0 & & b_1   & & & & & C_2/C_2  \\
 & b \ar[ur]_-{\Z\{p_1\}} \ar[ul]^-{\Z\{p_0\}} \ar@(dr,dl)@{->}[]^-{\Z/2\{t\} \ltimes \Z\{g\} } &  &  & & &  & C_2/e  \ar[u]_-\rho \ar@(dr,dl)[]^{\tau}
}\]
The labels on the arrows denote the set of morphisms. Each set of morphisms is a right module over $\pi_1(S^1,b(e))\cong \Z$ and the labels keep track of the generators according to the choice of paths as labeled in \cref{fig:S11}. We have $\phi^{-1}(e) \cong \pi_1(S^{1},b(e))$ generated by $g$ and $\phi^{-1}(\tau) \cong \Pi S^{1}(b(e),b(\tau))$, the right $\pi_1(S^1,b(e))$-module generated by $t(e)$. We abuse notation and write 
\begin{align*}
p_0&=(\rho,p_0) &g&=(e,g) \\
p_1&=(\rho,p_1)  & t&=(\tau, t).
\end{align*}
Then we  have the following generating relations\footnote{To see these relations the reader may find it helpful to ``conduct'' the paths followed by the points in the orbits. That is, for each point $x: C_2/e \to B$, trace $x(e)$ with one finger and $x(\tau)$ with another.} among the morphisms.
\begin{enumerate}
    \item $p_0\circ g\simeq p_0 \circ t$
    \item $p_1 \circ t\simeq p_1$
    \item\label{3} $t \circ t\simeq \id$
     \item $t\circ  g= g^{-1} \circ t$
\end{enumerate}
For example, one can check relation \eqref{3} as follows, 
\[(\tau,t)\circ (\tau,t) = (\tau^2, t\tau \star t) = (\tau^2 , t^{-1}\star t) =(e,c_b) = \id \]
where $c_b$ is the constant path at $b$. 

Let us examine our claim that all of the endomorphisms of $b$ are automorphisms and the group structure is
\[\Pi S^{1,1} \cong \Z/2 \ltimes \Z.\]
Any endomorphism $(\alpha, \omega)$ of $b$ is invertible because $\alpha$ is an isomorphism of $C_2/e$ and the path $\omega$ can be reversed.  Moreover, the automorphisms $\Pi S^{1,1}(b,b)$ decompose as $\Z/2\{t\} \ltimes \Z\{g\}$ because $\Z\{g\}$ is a normal subgroup and conjugating gives $(\tau,t)(e,g)(\tau,t) = (e,g^{-1})$.
 
The endomorphisms of $b_0$ and $b_1$ are all trivial since these correspond to paths in $(S^{1,1})^{C_2}$, which is discrete.
\end{example}

\begin{example}
\label{ex:funRP2tw}
    Again we take $G=C_2= \langle \tau \rangle$, and we consider $B=\RPtwist$, the projective space of $\R^{3,1}$. 
    Note that 
    \[S^{1,1}\cong\mathbb{P}(\R^{2,1})\subset \RPtwist\] 
    is a subspace.  Thus we can use the skeleton of $\Pi S^{1,1}$ from \cref{ex:funS11} when describing the skeleton of $\Pi \RPtwist$.

\begin{figure}[h]
\begin{center}
\includegraphics[width=0.5\textwidth]{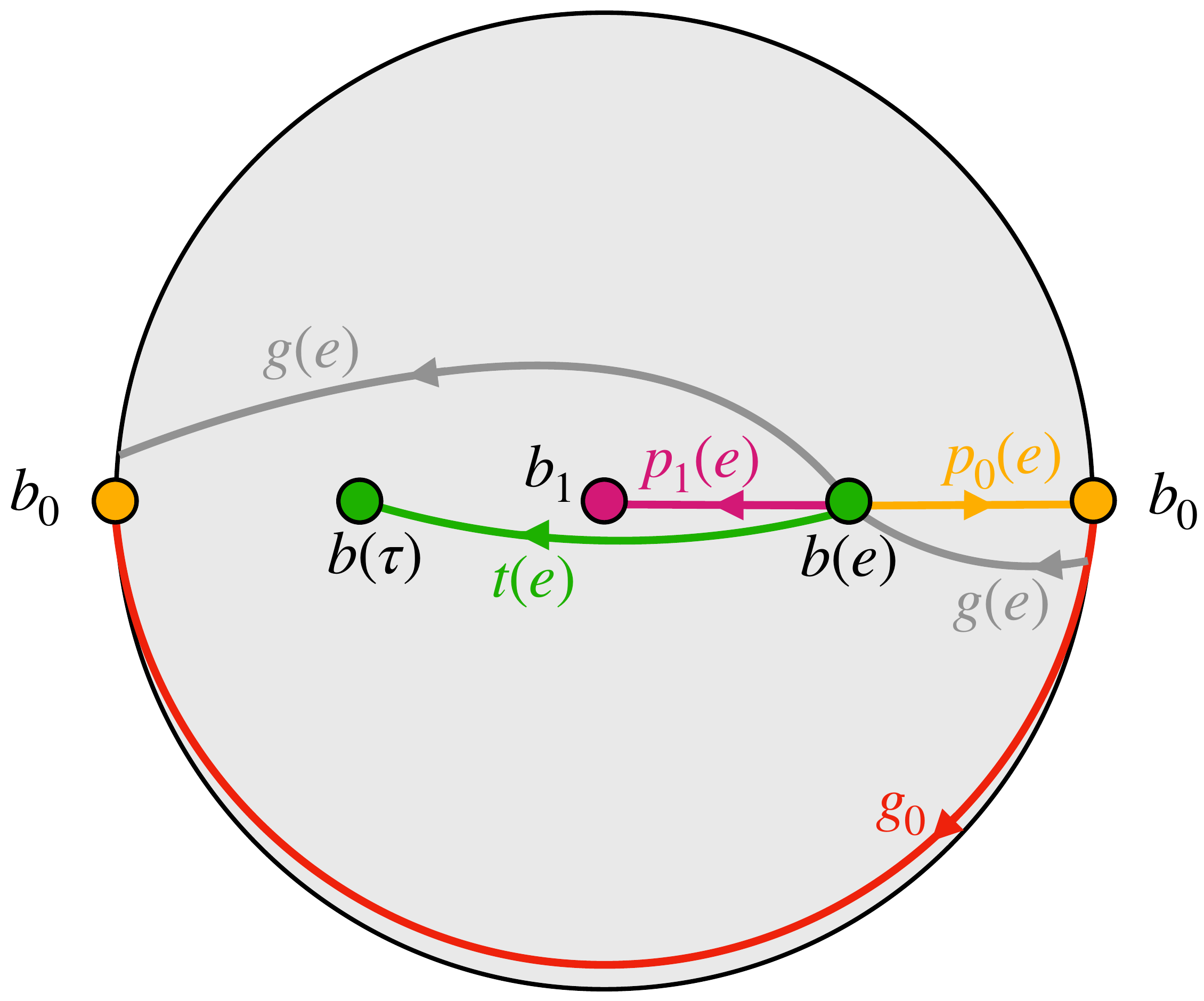}
\caption{Some morphisms in $\Pi \RPtwist$}
\label{fig:rp2tw}
\end{center}
\end{figure}
    
    The underlying space $i^*_e \RPtwist$ is $\R P^2$, which is path-connected and so has one object up to isomorphism. The fixed set consists of two connected components: one the isolated point $b_1\in S^{1,1}$ as in \cref{ex:funS11}, and one homeomorphic to $S^1$ containing $b_0$.  We depict these in \cref{fig:rp2tw}, with $b_1$ in the center and the fixed $S^1$ on the boundary of the disk.  The $C_2$-action in \cref{fig:rp2tw} is rotation about the center.

We can choose a skeleton of $\RPtwist$ with three objects all lying in the subspace $S^{1,1}$, so we choose the same ones as in the skeleton from \cref{ex:funS11}, namely $b\colon C_2/e\to \RPtwist $, $b_0\colon C_2/C_2\to \RPtwist$ and $b_1\colon C_2/C_2\to \RPtwist$.

The morphisms between these objects are depicted in \cref{fig:rp2tw} and give the following diagram over $\cO_{C_2}$. 
\[\xymatrix@R=1.5pc{
&  \Pi  \RPtwist & & \ar[rrr]^-\phi  & & & & \cO_{C_2}\\
\ar@(ur,ul)@{->}_-{\Z\{g_0\}}b_0 & & b_1 & & & & & C_2/C_2  \\
 & b \ar@{->}[ur]_-{\Z/2\{p_1\}} \ar@{->}[ul]^-{\Z/2\{p_0\}} \ar@(dr,dl)@{->}[]^-{\Z/2\{t\} \times \Z/2\{g\} }  & &  & & &  & C_2/e  \ar[u]_-\rho \ar@(dr,dl)[]^{\tau}
}\]
The reader will find a similar diagram and description in \cite[Example 3.1.2]{CW_book}. As in \cref{ex:funS11}, we abuse notation and write $p_0$ for $(\rho,p_0)$, $t$ for $(\tau, t)$, $g$ for $(e,g)$ and so on.  We have the following generating relations among the morphisms:
\begin{enumerate}
        \item $p_0\circ g\simeq p_0 \circ t$
    \item $p_1 \circ t\simeq p_1$
    \item $t \circ t\simeq \id$
    \item $g\circ g\simeq \id$
     \item $t\circ  g\simeq g^{-1} \circ t \simeq  g \circ t$
     \item $g_0\circ p_0\simeq p_0\circ g$.
\end{enumerate}
Here $g = g^{-1}$, so the automorphisms of $b$ decompose as a  product 
\[\Pi  \RPtwist(b,b) \cong \Z/2\{t\} \times \Z/2\{g\}.\]
\end{example}

\begin{example}
\label{example:equivariant fundamental groupoid for a fixed space}
    In the case that $B$ is a path-connected $G$-space with a trivial $G$-action, $\Pi B$ is simple to describe. Fix any point $b\in B$ and let $b \colon G/G \to B$ be the inclusion at $b$. Then $\Pi B$ has skeleton with one point for each orbit $G/H$, given by the composite
\[b_H \colon \xymatrix{G/H \ar[r]^-{\rho} &  G/G \ar[r]^-{b} &  B}\]
where $\rho$ is the quotient.
Furthermore, the morphisms are in bijection with 
\[\Pi B(b_H,b_K) \xrightarrow{\cong} \cO_G(G/H, G/K) \times \pi_1(i^*_eB,b), \]
and for 
\[b_e \colon \xymatrix{G/e \ar[r]^-{\rho} &  G/G \ar[r]^-{b} &  B}\] there is an isomorphism of groups
\[\Pi B(b_e,b_e) \cong G \times \pi_1(i^*_eB,b). \]
For example, for $G=C_2$ and $B$ path-connected with trivial $C_2$-action, we get a skeleton for $\Pi B$ of the form
\[\xymatrix@R=2pc{
\Pi B & \ar[rr]^-\phi   & & & \cO_{C_2}\\
b\ar@(ur,ul)@{->}_-{\pi_1(i^*_eB,b)}&  && & C_2/C_2  \\
 b_e \ar@{->}[u] \ar@(dr,dl)@{->}[]^-{C_2 \times \pi_1(i^*_eB,b)} & & && C_2/e  \ar[u]_-\rho \ar@(dr,dl)[]^{\tau}
}\]
for any point $b \in B$. 
\end{example}

\subsection{Representations of the fundamental groupoid}\label{sec:ROPiB}

We will now review the definition of the ring $\RO(\Pi B)$ of representations of the fundamental groupoid, which will form the extended grading for the Costenoble--Waner cohomology theories.  The grading is intended to incorporate information about vector bundles over $B$.  Indeed, every vector bundle over $B$ will give rise to an element of $\RO(\Pi B)$.  The group $\RO(\Pi B)$ generalizes $\RO(G)$ in the sense that when $B= G/G$, we will have  
\[\RO(\Pi G/G) \cong \RO(G) .\]
We will also give examples of computations of $\RO(\Pi B)$.  We will see that unlike $\RO(G)$, which is always a free abelian group, $\RO(\Pi B)$ may have torsion.

\begin{warn}\label{warn:ROGbad}
We recall Adams' warning \cite[\S6]{AdamsPrereq} about grading a cohomology theory on $RO(G)$. Similarly, it is a bit dishonest to claim that the cohomology theories are graded on $RO(\Pi B)$. In truth, the cohomology theories are graded on a category of which the group $RO(\Pi B)$ is a decategorification. This sleight of hand can lead to sign mistakes and other ambiguities if one is not careful.

Since this paper is just a user's guide,  we will not linger on this technical point. We refer the reader to \cite{CostenobleARXIV}, specifically Section 16 of the second version, for a detailed treatment of how to deal with this in this setting.
\end{warn}

We first define virtual vector bundles over orbits. Our treatment follows \cite{CMW}. 
Throughout, all vector bundles are equivariant real orthogonal vector bundles and all representations are real orthogonal representations. Bundle maps are assumed to induce isomorphisms on fibers.

\begin{defn}[{\cite[Def. 2.2 and p.321]{CMW}}]
 For each subgroup $H\subset G$ and for each isomorphism class of orthogonal representation, we choose a representative of the form $H \to O(n)$ and write $\Rep_H(n)$ for this set of representatives.
 \begin{enumerate}[(a)]
 \item
The category $\cV_G$ of vector bundles over $\cO_G$ is the category over $\cO_G$,
\[ \phi \colon \cV_G \to \cO_G\]
which is the disjoint union of categories
$\cV_G(n)$, indexed over non-negative integers $n$ defined as follows. The objects of $\cV_G$ are the bundles $G\times_HV \to G/H$ where $V$ is an element of $\Rep_H(n)$. A morphism is a pair $(\alpha,f)$ where $\alpha \colon G/H \to G/K$ is a map of $G$-sets and $f$  is a bundle map
\[ \xymatrix{ 
G\times_H V \ar[r]^-f \ar[d] & G \times_K W \ar[d]  \\
G/H \ar[r]^-\alpha & G/K},\]
 up to $G$-bundle homotopy.
The functor $\phi \colon \cV_G \to \cO_G$ sends $G\times_H V$ to $G/H$ and $(\alpha, f)$ to $\alpha$. 

\item The category $\vV_G$ of virtual vector bundles over $\cO_G$ is the category over $\cO_G$,
\[ \phi \colon \vV_G \to \cO_G\]
defined as the disjoint union of categories
$\vV_G(n)$
indexed on $n\in \Z$.  Each $ \vV_G(n) $ is defined as follows.
The objects of $ \vV_G(n) $ are pairs $(G\times_HV_1, G\times_HV_2)$ of objects in $\cV_G(n)$.
A morphism 
\[(\alpha, f_1, f_2 ) \colon G\times_H V \to G\times_K W \]
 consists of a $G$-map $\alpha \colon G/H \to G/K$ and an equivalence class of pairs 
of bundle maps
\[ \xymatrix{  G\times_H (V_i \oplus Z) \ar[d]  \ar[r]^-{f_i}   &  G\times_K (W_i\oplus Z') \ar[d] \\
G/H \ar[r]^-\alpha & G/K 
} \]
for $Z$ an orthogonal representation of $H$ and $Z'$ of $K$ with $|V_i|+|Z|=|W_i|+|Z'|$.  The equivalence relation is generated by $G$-bundle homotopy, and the following suspension relation: Given $T$ an orthogonal representation of $H$ and $T'$ of $K$ and a bundle map  $k \colon G\times_HT \to G\times_K T'$, then $(\alpha,f_1,f_2)$ is equivalent to $(\alpha, f_1\oplus k, f_2\oplus k)$.
The functor $\phi \colon \vV_G \to \cO_G$ maps $G\times_H V$ to $G/H$ and $(\alpha, f_1,f_2)$ to $\alpha$.
\end{enumerate}
There is a functor
\[ v \colon \cV_G \to \vV_G\]
which maps $G\times_H V$ to $(G\times_H V,0)$ and $(\alpha,f)$ to $(\alpha,f,0)$. A virtual bundle is called an \emph{actual bundle} if it is in the image of this functor, and similarly, an actual morphism is one in the image as well. 
\end{defn}

The following remark  adds alternative descriptions of $\vV_G$, each of which are equivalent to \cref{def: vVG}.

\begin{rem}\label{rem:vbarVG}
\begin{enumerate}[(1)]
\item The categories $\cV_G(n)$, $\cV_G$, 
 $\vV_G(n)$ and $\vV_G$ are skeleta of larger categories $\overline{\mathcal{V}}_G(n)$, $\overline{\mathcal{V}}_G$, 
 $v\overline{\mathcal{V}}_G(n)$,  and $v\overline{\mathcal{V}}_G$ built in a analogous way, but allowing all orthogonal $n$-bundles over orbits, not just those of the form $G\times_H V$ for $V\in \Rep_H(n)$. See \cite[Def. 2.1, 2.2 and IV.19]{CMW}, and also \cite[\S2]{CW_Duality}. 
 Using these larger categories can afford flexibility for various constructions. Choosing  fixed isomorphisms from each object of $\cV_G(n)$ to objects in $\overline{\mathcal{V}}_G(n)$ gives a retract $\overline{\mathcal{V}}_G(n) \to \cV_G(n)$, and similarly for the other categories listed above.
 \item   Let $U$ be a complete $G$-universe and let $BO_G(U)$ be the classifying space for orthogonal $G$-vector bundles. The category $\vV_G \cong \Pi(BO_{G} (U) )$ and $KO_G (X) = [X, BO_G (U)]_G$, see \cite[p.324]{CMW}. 
 \item Just as when we defined $\Pi B$, we could have instead obtained $\vV_G$ through the Grothendieck construction (see \cite[Definition 2.2]{pronk2019equivariant}) from an appropriately defined pseudofunctor
    \[ \vV_G \colon \cO_G^\op \to \Gpds.\]
Taking this perspective on $\Pi B ,\vV_G \colon \cO_G^\op \to \Gpds$, then $RO(\Pi B)$ corresponds to natural transformations modulo natural isomorphisms.
 \end{enumerate}
\end{rem}

We give some examples of the category $\vV_G$.
\begin{example}\label{ex:noneqvV}
 Let $G=e$ and $\vV=\vV_G$. For $p\in \Z$, we let $\R^p$ denote the virtual representation corresponding to the pair $(\R^{p+n},\R^{n})$ for $p+n>0$ in the Grothendieck completion. A morphism can be represented by a pair $(f,\id)$ for $f\in \pi_0O(p+n)$. Therefore, 
 \[
 \vV(\R^p,\R^p)\cong O(1) = \{\pm 1\}
 \]
 with the identification induced by the determinant.  That is, a morphism is determined by whether the map $f$ is orientation preserving or reversing. 
 \end{example}
 \begin{example}\label{example:VC2}
  Consider the case $G=C_2$. A $C_2$-vector bundle over $C_2/e$ is isomorphic to $C_2\times \R^p$ since any orthogonal representation of the trivial group is isomorphic to $\R^p$ for some natural number $p$. A $C_2$-vector bundle over $C_2/C_2$ is isomorphic to a representation $\R^{p,q}$. 
Isomorphism classes of virtual representations are thus labeled by pairs $(p,q) \in \Z$. We get three types of non-empty morphism sets:
    \begin{enumerate}[(i)]
        \item An element of $\vV_{C_2}(C_2\times \R^p, C_2\times \R^p ) $ is determined by a pair: a map $\alpha \colon C_2/e \to C_2/e$ together with a homotopy class of orthogonal maps  
        $e \times \R^p$ to $\alpha(e) \times \R^p$, i.e.\ an element of $O(1)$. So,
        \[\vV_{C_2}(C_2\times \R^p, C_2\times \R^p )  \cong C_2\times O(1).\]
 \item We have 
 \[\vV_{C_2}(C_2\times \R^p, \R^{p,q} )  \cong \vV(\R^p, i^*_e\R^{p,q} )  
 \cong O(1),\]
 and
 \item for $\vV_{C_2}(\R^{p,q},\R^{p,q})$, we represent $\R^{p,q}$ as a pair $(\R^{p+n,q+m}, \R^{n,m})$ with $p+n >q+m>0$. Any morphism can be represented by $(f,\id)$ for \[f \in \pi_0O(q+m) \times \pi_0O(p+n-(q+m)) \cong O(1)\times O(1).\]
 Therefore, 
 \[\vV_{C_2}(\R^{p,q},\R^{p,q}) \cong O(1)\times O(1).\]
    \end{enumerate}
\end{example}

We now turn to $\RO(\Pi B)$.
\begin{definition}\label{def: vVG}
An \emph{actual representation $\gamma$ of $\Pi B$} of dimension $n$ is a functor
\[\xymatrix{ \Pi B \ar[dr]\ar[rr]^{\gamma} &&  \cV_G(n) \ar[dl]\\
&\cO_G &
}\]
commuting with the functors to $\cO_G$. A \emph{virtual representation $\gamma$ of $\Pi B$} of dimension $n$ is a functor
\[\xymatrix{ \Pi B \ar[dr]\ar[rr]^{\gamma} &&  \vV_G(n) \ar[dl]\\
&\cO_G &
}\]
commuting with the functors to $\cO_G$. 
We let $RO(\Pi B)$ denote the natural isomorphism classes of virtual representations. 

Further, $RO(\Pi B)$ forms a ring under direct sums and tensor products in $\vV_G$, called the \emph{ring of representations of $\Pi B$.}
 \end{definition}

  \begin{remark}
  \begin{enumerate}[(1)]
 \item Instead of using virtual bundles to define $RO(\Pi B)$, we could have considered the group completion of isomorphism classes of actual representations.
As discussed in \cite{CW_Duality}, both approaches are equivalent if $\Pi B$ has finitely many isomorphism classes of objects.  In this paper, we will only consider such spaces $B$.
\item
One can generalize the definition of a representation by changing the target to get a functor $\gamma \colon \Pi B \to v\overline{\mathcal{V}}_G(n)$ as in \cref{rem:vbarVG}. This gives more flexibility when discussing representations. Each such functor is naturally isomorphic to one with target the full subcategory $\vV_G(n)$, so this  generalization does not change $\RO(\Pi B)$.
\end{enumerate}
\end{remark}

\begin{example}[Constant $V$ representation]\label{ex:constant}
Let $V$ be an $n$-dimensional orthogonal representation of $G$. There is a functor $\overline{\gamma}_V\colon \Pi B \to \overline{\mathcal{V}}_G(n)$ which sends $x \colon G/H \to B$ to the bundle $G/H \times V \to G/H$ and sends a morphism $(\alpha, \omega) \colon x \to y$, where $y \colon  G/K \to B$, to $\alpha \times \id_V  \colon G/H \times V \to G/K \times V$. The functor $\overline{\gamma}_V$ is naturally isomorphic to one of the form $\gamma_V\colon \Pi B \to \cV_G(n)$. We can see this by composing with a retract $\overline{\mathcal{V}}_G(n) \to \cV_G(n)$, a process which can be described explicitly as follows. For each $H\subset G$, we implicitly identify $i^*_HV$ with the corresponding representative in $\Rep_H(n)$. Applying the shearing isomorphism 
\[ G/H \times V \xrightarrow[\text{shear}]{\cong} G\times_H i^*_H V  \]
given by $(gH,v) \mapsto [g, g^{-1}v]$ allows us to identify $\overline{\gamma}_V(x)$ with an object in $\mathcal{V}_G(n)$ for each $x\in B$. Doing this for each $x$ and conjugating the morphisms with these identifications gives a representation $\gamma_V \colon \Pi B \to \cV_G(n)$. We call $\gamma_V$ the \emph{constant $V$ representation} and abuse notation, denoting it simply by $\gamma_V=V$.
 Using the universal property of Grothendieck completion, this construction gives a homomorphism
\[RO(G) \to RO(\Pi B).\]
\end{example}

Any orthogonal $G$-vector bundle over $B$ gives rise to an element of $RO(\Pi B)$ via pullback.   
We explain this construction here. See for example \cite[Proposition~2.1.3]{CW_book}, \cite{CW_Duality}, or \cite[\S 1.3]{CostenobleB}.

\begin{defn}\label{def:dimension}
    Let $\xi \colon E \to B$ be an $n$-dimensional orthogonal $G$-vector bundle over the $G$-space $B$. Then $\xi$ determines an element of $\RO(\Pi B)$, which we denote by
    \[\dim(\xi) \in \RO(\Pi B), \]
    or simply $\xi \in \RO(\Pi B)$ by abuse of notation, as follows. For each orbit $x\colon G/H \to B$ in $ \Pi  B$, the pullback gives a vector bundle
    \[x^*\xi \to G/H.\] 
        Given a path $(\alpha,\omega) \colon x \to y$, we use homotopy lifting for vector bundles to get a well-defined homotopy class of bundle maps
        \[x^*\xi \to y^*\xi.\]
        This describes a functor $\overline{\dim} ( \xi)   \colon \Pi B \to \overline{\mathcal{V}}_G(n)$. 
        Composing with a retract $\overline{\mathcal{V}}_G(n) \to \cV_G(n)$ gives an actual representation $\dim (\xi) \colon \Pi B \to \cV_G(n)$.
    \end{defn}
    
    \begin{rem}\label{rem:dimchoices}
        Regarding the choice of retract, we do this as follows in our computations. 
    For each $x\in B$,  the fiber of $\xi^*x$ over the identity coset $eH$ is an orthogonal $H$-representation, and so we can identify it with some $\xi_0(x)\in \Rep_H(n)$. This involves a choice.\footnote{In most examples we consider, making this choice will be equivalent to picking an orientation in each fiber of $\xi$. Note that we do not need to, and often cannot, make this choice continuously.}
   Once we have made this choice, we get an isomorphism
    \[\xymatrix{ \dim( \xi) (x)= G\times_H\xi_0(x) \ar[r]^-{\cong} & x^*\xi= \overline{\dim}(\xi) (x) }
    . \] 
    Given a morphism $(\alpha,\omega) \colon x \to y$, 
  we have a commutative diagram
    \[\xymatrix{ \dim( \xi) (x) = G\times_H\xi_0(x) \ar@{-->}[d] \ar[r]^-{\cong} & x^*\xi =\overline{\dim}(\xi)(x) \ar[d]^-{\overline{\dim}(\xi) (\alpha,\omega)} \\
    \dim (\xi) (y)=G\times_K\xi_0(y) \ar[r]^-{\cong} & y^*\xi = \overline{\dim}(\xi) (y) 
    }
     \] 
and the dashed arrow is obtained as follows. The map from the fiber over $eH$ in $x^*\xi$ to the fiber over $\alpha(eK)$ in $y^*\xi$ is induced by lifting the path $\omega(eH) \colon x(eH) \to y(\alpha(eK))$ in $B$ to the total space of $\xi$. This forces the map on $\xi_0(x)$. Equivariance determines the rest.
    \end{rem}

\begin{example}\label{ex:dimVconstant}
    For a $G$-representation $V$, the trivial bundle $\xi_V = B \times V$ has $\dim(\xi_V) = V$, the constant $V$ representation of \cref{ex:constant}.
\end{example}

\begin{defn}\label{defn:KO}
   By the universal property of group completion, the construction of \cref{def:dimension} induces a map from the isomorphism classes of virtual orthogonal $G$-vector bundles to representations of the equivariant fundamental groupoid
    \[\dim \colon KO_G(B) \to \RO(\Pi B) .\]
    Let $ KO(\Pi  B)$ be the subgroup of $\RO(\Pi B)$ corresponding to the image of the morphism $\dim$. In \cite[p.28]{CW_Thom}, the image of $\dim$ is denoted $DO_G(B)$.
\end{defn}

Base change and change-of-groups will lift to representations of the fundamental groupoid.

\begin{defn}\label{defn:cog-vV}\label{defn:cog-ROPiB}
Define the following functors and homomorphisms:
\begin{enumerate}[(a)]
\item
Let 
\[i_H^* \colon \vV_G \to \vV_H\]
be the functor which takes a $G$-bundle over $G/K$ to its pullback along the inclusion $H/H\cap K \to G/K$, and morphisms to the restriction applied to the pull-backs.

\item Let 
\[i^*_H \colon \RO(\Pi B) \to \RO(\Pi i^*_HB) \]
be the functor which sends $\gamma$ to $i^*_H\circ \gamma \circ (G\times_H - )$.

\item For a map $f\colon A \to B$ of $G$-spaces, let
\[f^* \colon \RO(\Pi B) \to \RO(\Pi A)\]
be the map which sends $\gamma$ to $\gamma \circ \Pi f$.
\end{enumerate}
\end{defn}
\begin{rem}
For any $b \in \Pi i^*_HB$, we have 
\[(i^*_H\gamma)_0(b) = \gamma_0(G\times_H b),\]
where $\gamma_0$ indicates the fiber over the identity coset as in \cref{def:dimension}.
\end{rem}

\begin{rem}\label{rem:ROGinKO}
If $\xi$ is a $G$-bundle over $B$ and $\dim(\xi) \in \RO(\Pi B)$ is the corresponding representation, then
\[ i^*_H\dim(\xi) = \dim(i^*_H\xi) \]
where $i^*_H\xi$ is the $H$-bundle underlying $\xi$ over $i^*_HB$.

For $\rho \colon B \to G/G$, the image of 
\begin{align}\label{eq:rhostar}
\rho^* \colon \RO(G)=\RO(\Pi G/G) \to \RO(\Pi B)\end{align} 
is the copy of $RO(G)$ described in \cref{ex:constant}.
This subgroup is always contained in $KO(\Pi B)$. Indeed, for $V$ an orthogonal representation of $G$, $\rho^*V = \dim( \xi_V)$ where $\xi_V =B\times V$ as in \cref{ex:dimVconstant}.
\end{rem}

\begin{example}\label{example: coefficients when G = e}
Nonequivariantly, let
$G=e$ and take $B$ to be path-connected. Let $\pi =\pi_1 (B,b)$ for some $b\in B$. Then $\Pi B$ has skeleton the groupoid with one object whose automorphisms are $\pi$. It follows that
\[RO(\Pi B) \cong \Z \times \Hom(\pi,O(1))\]
where $\gamma =(p, \gamma_b)\in \RO(\Pi B)$ has virtual dimension $p$  and $\gamma_b$ gives a homomorphism
\[ \pi = \Pi B(b,b) \to \vV(\R^p,\R^p) \cong O(1) .\]

 In this case, the map 
 \[\dim \colon KO_G( B) \to \RO(\Pi B) \]
is surjective and thus $KO(\Pi B) = \RO(\Pi B)$. To see this, note that for $\rho \colon B \to \pt$,
\[\rho^*(RO(\Pi \pt)) = \Z\times \{\id\}\]
and this subgroup is in the image of $\dim$ by \cref{rem:ROGinKO}. So to get the rest, it is enough to show that  $\gamma=(1,\gamma_b)$ is in the image of $\dim$. Construct the bundle 
\[E\pi \times^{\gamma}_\pi \R \to B\pi \]
with $\pi$ acting on $\R$ via the homomorphism $\gamma_b$. Here, $B\pi$ is the classifying space of $\pi$, and $E\pi \to B\pi$ the universal principal $\pi$-bundle.
Pull-back 
along the first Postnikov stage $B \to B\pi$ defines a bundle $\xi(\gamma) $ over $B$ with
$\dim \xi(\gamma)  = \gamma$.
\end{example}

\begin{example}
Take $G=C_2$ and assume that $B$ is a path-connected, space with a trivial $C_2$-action as in Example \ref{example:equivariant fundamental groupoid for a fixed space}. Let $\pi =\pi_1 (i^*_eB,b)$ for some $b\in i^*_eB$.
A representation in $RO(\Pi B)$ is completely determined by its value over $C_2/C_2$.  So, from \cref{example:VC2}, we have
\[RO(\Pi B) \cong RO(C_2) \times \Hom(\pi, O(1)\times O(1)).\]
Using a similar trick to \cref{example: coefficients when G = e} above, we can show that  $\dim$
is surjective. Indeed, $\rho^*(RO(C_2)) = RO(C_2) \times \{ \id \}$ is in the image of $\dim$, so to get the rest, it is enough to prove that $\gamma = (\R^{2,1},\gamma_b)$ is in the image. We use the diagonal embedding of $O(1)\times O(1)$ in $O(\R^{2,1})$, which gives an action of $\pi$ on $O(\R^{2,1})$ through $\gamma_b$.
Form the $C_2$-bundle
\[E\pi \times_\pi^{\gamma} \R^{2,1} \to B\pi\]
where $E\pi$ and $B\pi$ are trivial $C_2$-spaces. We get a $C_2$-bundle $\xi(\gamma)$ by pulling back as before along the first Postnikov stage $B \to B\pi$ and $\dim \xi(\gamma) = \gamma$.
\end{example}

\begin{rem}\label{rem:kappax}
For a based $G$-space $B$,
the inclusion of the base point $\iota \colon G/G \to B$ 
induces a homomorphism
\[\iota^* \colon RO(\Pi B) \to RO(G)\]
which splits the map $\rho^* \colon RO(G) \to RO(\Pi B)$ of \eqref{eq:rhostar}.
Letting $\ker(\iota^*)=\widetilde{RO}(\Pi B )$,
it follows that
\[RO(\Pi  B) \cong \widetilde{RO}(\Pi B ) \times RO(G). \] 
\end{rem}

\subsection{Some computations of $\RO(\Pi B)$}
We now compute some examples of $\RO(\Pi B)$. We let $G=C_2 = \langle \tau \rangle$ and compute $\RO(\Pi S^{1,1})$ and $\RO(\Pi \RPtwist)$.  We will use the notation introduced in \cref{ex:funS11}, \cref{ex:funRP2tw}, and \cref{example:VC2}.

   \begin{lem}\label{lem:ROpiS11}
   There is an isomorphism 
   \[\RO\pars{\Pi S^{1,1}} \cong \Z^3 \]
   which assigns to a triple $(p,q_0,q_1)$ the representation  $\gamma$, which on objects is given by \begin{align*}
   \gamma(b_0)&=\R^{p,q_0} &
   \gamma(b_1)&=\R^{p,q_1} & \gamma(b)& =C_2\times \R^{p},
   \end{align*} 
   and on morphisms by 
   \begin{align*}
   \gamma(g) &= e\times (-1)^{q_0+q_1} & 
   \gamma(t) &= \tau \times(-1)^{q_1} &
   \gamma(p_0)&=\gamma(p_1)  =1
   \end{align*}
   \end{lem}

     \[\xymatrix@C=1.5pc{
&  \Pi  S^{1,1}& & \ar[rr]^-\gamma  & & & & \vV_{C_2}\\
b_0 & & b_1  & & & & \R^{p,q_0} & & \R^{p,q_1}   \\
 & b \ar@{->}[ur]_-{p_1} \ar@{->}[ul]^-{p_0} \ar@(dr,d)@{>}[]^-{g}\ar@(d,dl)@{>}[]^-{t } &  & & & & & C_2\times \R^p \ar@{->}[ur]_-{1} \ar@{->}[ul]^-{1} \ar@(dr,d)@{>}[]^-{\ \ \ \ \
 (-1)^{q_0+q_1}}\ar@(d,dl)@{>}[]^-{ (-1)^{q_1} \ \ }
}\]
\begin{proof} 
 Up to isomorphism, we have $\gamma(b) =C_2\times \R^p $ and $\gamma(b_i) =\R^{n_i,q_i} $ for $i=0,1$.
Furthermore, we have bundle maps 
\[\gamma(p_i) \colon C_2\times \R^p \to\R^{n_i,q_i} \]
which are isomorphisms on fibers 
and so $p=n_0=n_1$.  It remains to determine the morphisms.

From \cref{example:VC2}, $\gamma(p_i) \in O(1) = \{\pm 1\}$, so we just need to determine the signs.  That is, up to homotopy, we only need to know whether the maps are orientation preserving or reversing on fibers. Also from  \cref{example:VC2}, we have $\gamma(t), \gamma(g) \in C_2 \times O(1)$.  Recall that $t = (\tau, t)$ and $g = (e,g)$ were shorthand, so we only need to determine the signs for these as well. 

Choose the canonical basis for $\R^p$ (considering $\R^p$ as the pair $(\R^{p+n}, \R^n)$ with $p+n >0$ if necessary). Let $v \in \R^p$. Suppose $\gamma(p_1)(e,v) = Av$ for some orthogonal matrix $A$.  Since $\gamma(p_1)$ is equivariant, $\gamma(p_1)(\tau, v) = \tau \cdot Av$, where the action of $\tau$ on the right-hand side is determined by $q_1$.  If $A$ is orientation preserving but $q_1$ is odd, $\gamma(p_1)(\tau,-)$ will be orientation reversing.

We write $\gamma_a(p_1) \colon a \times \R^p \to i^*_e\R^{p,q_1}$ with $a \in \{e,\tau\}$ for the orthogonal linear map on fibers induced by $\gamma(p_1)$.  From the argument above, we must have
\[\det(\gamma_\tau(p_1))= \det(\tau) \det(\gamma_e(p_1))=(-1)^{q_1} \det(\gamma_e(p_1)),\] 
where $\tau$ denotes the action on $\R^{p,q_1}$.  The element $\gamma(p_i) \in O(1)$ for the representation $\gamma$ is determined by $\det(\gamma_e(p_1))$.

Similarly, we write $\gamma_a(p_0)$, $\gamma_a(g)$, and $\gamma_a(t)$ for the orthogonal linear maps on the fibers.  Note that $\gamma_a(t) = a\times \R^p \to \tau a \times \R^p$ since $t=(\tau, t)$.  We need to find $\det(\gamma_e(-))$ in each case.
 
The relation $p_1 \simeq p_1 \circ t $ from \cref{ex:funS11} implies $\gamma(p_1)=\gamma(p_1)\circ \gamma(t)$, and this means
\[ \gamma_e(p_1) =\gamma_\tau(p_1) \gamma_e(t). \]
Applying determinants, we have
\begin{align*}\det(\gamma_e(p_1)) &= \det(\gamma_\tau(p_1)) \det(\gamma_e(t))\\
&= (-1)^{q_1} \det(\gamma_e(p_1)) \det(\gamma_e(t)). \end{align*}
Therefore
\[\det(\gamma_e(t))  = (-1)^{q_1}\] and we can conclude
\[\gamma(t)= \tau \times (-1)^{q_1} \in C_2\times O(1).\]

A similar argument using the relation $p_0 \circ g \simeq p_0 \circ t$ gives 
\[ \gamma_e(p_0) \gamma_e(g) =\gamma_\tau(p_0) \gamma_e(t)  . \]
Applying determinants we see that
\[ \det(\gamma_e(g))  = (-1)^{q_0+q_1}, \]
as desired. 
It is then straightforward to check that $\det(\gamma_e(p_0))$ and $\det(\gamma_e(p_1))$ can be each chosen to be $1$, and that this choice does not matter up to isomorphism. 
\end{proof}

\begin{remark}\label{rem:matrixrem}
    A representation $\gamma$ is determined by the succinct diagram shown in \cref{lem:ROpiS11}, but the depiction perhaps obscures two subtle but important features of a representation: first that $\gamma$ is a functor, and second that choices of orientations are involved.

To address the first: less succinctly and by abuse of notation, we could treat $C_2 \times \R^p$ as $\R^p_e \oplus \R^p_\tau$ and use matrices.  For example, writing
    \[
    \gamma(p_1) = \begin{pmatrix} 1 & (-1)^{q_1}\end{pmatrix} \qquad \text{and} \qquad \gamma(t) = \begin{pmatrix} 0 & (-1)^{q_1}\\ (-1)^{q_1} & 0 \end{pmatrix}
    \]
    makes it clear that $ \gamma(p_1) = \gamma(p_1) \circ \gamma(t)$, which follows from $p_1 = p_1 \circ t$.

For the second: notice that, by definition $\gamma(b) = C_2 \times_e \R^p$.  When we write $\gamma(b) = C_2 \times \R^p$, we have implicitly identified an underlying orientation of both copies of $\R^p$. That is, we have identified $\R^p$ with the $C_2$-representation $\R^{p,0}$ and applied the inverse of the shearing isomorphism. 
With this identification, we have that $\det(\gamma_e(t)) = \det(\gamma_\tau(t))$.  In general, identifying $G \times_H V$ with $G/H \times V$ involves choices of coset representatives and choices of orientations on each copy of $V$. 
\end{remark}

\begin{example}
        In the previous example, $\RO(C_2)$ corresponds to the constant or ``homogeneous'' representations in $\RO(\Pi S^{1,1})$.  Using the isomorphism of \cref{lem:ROpiS11}, these are the dimensions $(p,q,q) \in \Z^3$.
\end{example}

\begin{example}\label{ex:tautS11}
Since $S^{1,1}$ is homeomorphic to $\mathbb{P}(\R^{2,1})$, the $C_2$-equivariant tautological line bundle $\taut$ of $\mathbb{P}(\R^{2,1})$ gives rise to an element of $\RO(\Pi S^{1,1})$. If we write $\R^{2,1} \cong \R^{1,0} \oplus \R^{1,1}$, we can identify $\mathbb{P}(\R^{2,1})\cong S^{1,1}$ so that the line along the $\R^{1,0}$-axis corresponds to $b_0$, and the line along the $\R^{1,1}$-axis corresponds to $b_1$.  Then
\[ \taut=(1,0,1) \in \RO(\Pi S^{1,1})\]
where by our abuse of notation $ \taut \in \RO(\Pi S^{1,1})$ is the representation given by pullbacks (see \cref{def:dimension}).
Note that $\taut$ is the M\"obius bundle from \cite[Example 3.4]{Hazel_fund} and does not give an element in $\RO(C_2)$.
\end{example}

As we see in the next lemma, $\RO(\Pi B)$ need not be free abelian.  The notation here is as in \cref{ex:funRP2tw}, see also \cref{{fig:rp2tw}}.

\begin{lem}\label{lem:ROpiRP2}
    There is an isomorphism 
    \[RO\pars{\Pi \RPtwist}\cong \Z^3\times \Z/2\]
    which assigns to $(p,q_0,q_1,\epsilon_0)\in \Z^3\times \Z/2$ the representation $\gamma$ which, on objects, is given by that \begin{align*}
   \gamma(b_0)&=\R^{p,q_0} &
   \gamma(b_1)&=\R^{p,q_1} & \gamma(b)& =C_2\times \R^{p},
   \end{align*} 
   and on morphisms by 
   \begin{align*}
   \gamma(g) &= e\times (-1)^{q_0+q_1} & 
   \gamma(t) &= \tau \times(-1)^{q_1} \quad &
   \gamma(p_0) &=\gamma(p_1)  =1 \end{align*}
\begin{align*}
    \gamma(g_0)&=\begin{pmatrix}
       (-1)^{\epsilon_0}&0\\0&(-1)^{\epsilon_0+q_0+q_1}
   \end{pmatrix}
   \end{align*}
\end{lem}

        \[\xymatrix@C=1.5pc{
&  \Pi  \RPtwist& & \ar[rr]^-\gamma & & & & \vV_{C_2}\\ \\
b_0 \ar@(ur,ul)@{->}_-{g_0}& & b_1  & & &&  \R^{p,q_0} \ar@(ur,ul)@{->}_-{\tiny \begin{pmatrix}
    (-1)^{\epsilon_0}&0\\0&(-1)^{\epsilon_0+q_0+q_1}
\end{pmatrix}} & & \R^{p,q_1}   \\
 & b \ar@{->}[ur]_-{p_1} \ar@{->}[ul]^-{p_0} \ar@(dr,d)@{>}[]^-{g}\ar@(d,dl)@{>}[]^-{t} & && & & & C_2 \times \R^p \ar@{->}[ur]_-{1} \ar@{->}[ul]^-{1} \ar@(dr,d)@{>}[]^-{\ \ \  
 (-1)^{q_0 + q_1}}\ar@(d,dl)@{>}[]^-{(-1)^{q_1} \ }
}\]

\begin{proof}
Much of the proof is the same as the proof of \cref{lem:ROpiS11}.    Up to isomorphism, we have $\gamma(b) =C_2\times \R^p $ and $\gamma(b_i) =\R^{p,q_i} $ for $i=0,1$ since we have bundle maps
\[\gamma(p_i) \colon C_2\times \R^p \to\R^{n_i,q_i} \]
and so $p=n_0=n_1$. 
Applying $\gamma$ to the relation $p_1 \simeq p_1 \circ t$  implies that
\[\gamma(t)= \tau \times (-1)^{q_1} \in C_2\times O(1),\]
as in the proof of \cref{lem:ROpiS11}. 
This, together with the relation $p_0\circ g \simeq p_0 \circ t$,  then gives 
\[ \gamma(g) =e\times (-1)^{q_0+q_1}  \in C_2\times O(1) . \]
Applying $\gamma$ to the relation $g_0 \circ p_0 \simeq p_0 \circ g$  from \cref{ex:funRP2tw} implies that 
\[\det(\gamma_e(g_0))=(-1)^{q_0+q_1}.\]
Recall that
$\gamma(g_0)\in \Hom_{\vV_{C_2}}(\R^{p,q_0},\R^{p,q_0})\cong \mathcal{O}(1)\times \mathcal{O}(1)$ (see Example \ref{example:VC2}) and thus we get
\[\gamma(g_0)=\begin{pmatrix}
    (-1)^{\epsilon_0}&0\\ 0&(-1)^{\epsilon_0+q_0+q_1}
\end{pmatrix}\]
with $\epsilon_0\in \Z/2$. As before, $\det(\gamma_e(p_0))$ and $\det(\gamma_e(p_1))$ can each be chosen to be $1$, and this choice does not matter up to isomorphism. 
\end{proof}

\begin{example}
    The tautological line bundle $\taut$ over $\mathbb{P}(\R^{3,1})$ corresponds to
    \[ \taut= (1,0,1,0)\in \RO(\Pi \mathbb{P}(\R^{3,1})).\]
    The degrees of the constant or homogeneous representations coming from $\RO(C_2)$ are $(p,q,q,0)$. 
\end{example}

%% file: witpaper-para.tex

\section{Parametrized Homotopy Theory}\label{sec:para}

There are many treatments of parametrized homotopy theory in the literature. 
We will review just enough so that the reader familiar with equivariant homotopy theory can get a general idea and follow the definitions. Our presentation will loosely follow \cite{CW_book}. There are a number of technical details that arise in the study of parametrized homotopy theory which we do not dwell on. For more details, see for example \cite{CW_book}, \cite{malkiewich2023parametrizedlowtech}, or \cite{MaySig}.

\subsection{Parametrized Spaces}

We give an overview of the definitions from parametrized homotopy theory that we will use.   

Fix a finite group $G$. All our spaces are $G$-spaces, all maps are $G$-maps, and we do not clutter the exposition with that. We also assume any spaces are $k$-spaces.  Moreover, we fix a compactly generated $G$-space $B$ with the homotopy type of a $G$-CW complex.\footnote{Following \cite{CW_book}, we restrict the base space to be compactly generated (meaning a $k$-space that is weak Hausdorff) with the homotopy type of a $G$-CW complex.  However, we do not require the spaces over $B$ to be weak Hausdorff.  An alternative using compactly generated weak Hausdorff spaces throughout is discussed in \cite{malkiewich2023parametrizedlowtech}.}

\begin{definition}
    The category of \emph{$G$-spaces over $B$}, denoted $\PTop{G}{B}$, has as objects pairs $(X,p)$ where $p\colon X\to B$ is a $G$-map. Maps in $\PTop{G}{B}$ are $G$-maps $f\colon X \to Y$ that commute with the projections $p$. For each $b\in B$, we denote the fiber over $b$ by $X_b = p^{-1}(b)$. 
\end{definition}

When $B$ is the point $G/G$, the category of spaces over $B$ is equivalent to the category of $G$-spaces.  A based $G$-space is a $G$-space $X$ with a map $s \colon G/G \to X$, which can be thought of as a section of the unique map $p\colon X \to G/G$. This motivates the generalization to the based version. The parametrized analogue of based spaces is that of ex-spaces.\footnote{When we define ex-spaces below, we will implicitly assume that our sections are sufficiently nice. That is, we assume all our ex-spaces are \emph{well-sectioned} in the sense of \cite[Defn. 5.2.5]{MaySig}. This is a cofibrancy condition and we refer the reader to \cite[\S 5.2]{MaySig} for more details.}
\begin{definition}
  An \emph{ex-space over $B$} (or \emph{ex-$G$-space over $B$}) is a $G$-space $(X,p)$  over $B$ together with a section $s \colon B \to X$ of $p$. We denote this as $(X,p,s)$. Maps of ex-spaces are required to commute with both projections and sections. 
  We denote the category of ex-spaces over $B$ by $\PTopb{G}{B}$.
\end{definition}

\begin{rem}
Costenoble--Waner refer to maps in $\PTopb{G}{B}$ as $G$-maps over and under $B$, see \cite[\S 2.2]{CW_book}.
\end{rem}

In the parametrized context, one should always try to think ``fiberwise'' when generalizing standard constructions as in the following constructions.

For the wedge product of ex-spaces $X$ and $Y$ over $B$, we  form their ``parametrized'' wedge
\[X \vee_B Y := X \cup_B Y\]
by gluing along the image of the section. The fibers of the wedge are just $X_b \vee Y_b$ joined at the respective images $s(b)$ of the sections in $X_b$ and $Y_b$.

The parametrized smash product is then defined by the pushout
\[\xymatrix{
X \vee_B Y \ar[r]^-{\subset} \ar[d]^-p & X \times_B Y \ar[d] \\
B \ar[r] & X \wedge_B Y
}\]
where $X \times_B Y$ is the pullback along the projections. The fibers of the smash product are the spaces $X_b \wedge Y_b$.  In the formation of the smash product, we use $s(b)$ as base points in $X_b$ and $Y_b$. 
Furthermore, if $A$ is a based $G$-space (over the point)
and $X$ is an ex-$G$-space over $B$, we can form the external smash product
\begin{align}\label{eq:tensorover}
X \wedge A  = X \wedge_B  (B \times A) 
\end{align}
where $B\times A$ is the space over $B$ with projection $p$ onto $B$ and section $(b,a_0)$ for $a_0$ the base point of $A$. This space has fibers $X_b \wedge A$. We thus see that $\PTopb{G}{B}$ is tensored over based $G$-spaces, $\Top_*^G$.

\begin{defn}
If $V$ is an orthogonal representation of $G$ and $X$ is an ex-$G$-space over $B$, then the \emph{$V$th suspension of $X$} is the space
\[ \Sigma_B^V X = X \wedge S^V \]
as in \eqref{eq:tensorover}.
\end{defn}

If $Y\subseteq X$ and 
\[\xymatrix{Y \ar[rr]^-{\subseteq} \ar[dr]_-p & & X \ar[dl]^-p \\
& B& }\]
commutes, the
\emph{parametrized quotient $X/_BY$} is the pushout
\[\xymatrix{Y \ar[r]^-{\subseteq} \ar[d]_-p & X \ar[d] \\
B \ar[r]^-s  &X/_B Y .
}\]
The map $X/_B Y \to B$ is induced from the pushout diagram from the identity on $B$ and the projection on $X$. The space $X/_B Y$ is automatically an ex-space with the section $s$ as in the diagram.

A homotopy between maps $f,g \colon X \to Y$ in $\PTopb{G}{B}$ is a map $X \wedge I_+ \to Y$ which starts at $f$ and ends at $g$. In the unbased context, homotopies in $\PTop{G}{B}$ are defined similarly, using the cylinder $X \times I  \to Y$ and ignoring the compatibility with the sections.  Thus we can define the following categories.

\begin{defn}
Let $\pPTop{G}{B}$ be the category with objects $G$-spaces over $B$ and homotopy classes of morphisms.  Let $\pPTopb{G}{B}$ be the category with objects ex-$G$-spaces over $B$ and morphisms based homotopy classes of  maps of ex-spaces.
\end{defn}

A weak equivalence in $\pPTop{G}{B}$ or $\pPTopb{G}{B}$  is a map $f\colon X \to Y$ which is an equivariant weak equivalence of total spaces, i.e., when forgetting the projections and sections. The homotopy categories are thus obtained by formally inverting these maps.\footnote{To study homotopy categories, the usual move is to give these categories model structures. There are many subtleties associated with parametrized homotopy theory addressed in \cite{MaySig} and \cite{CW_book}, so we refer the reader to these for a more thorough discussion.} We denote them by $\hPTop{G}{B}$ and $\hPTopb{G}{B}$.

\subsection{Base change functors} Given a $G$-map  
\[F\colon A \to B,\] 
there are a few functors that allow one to move between $\PTopb{G}{A}$  and $\PTopb{G}{B}$. They form various adjunctions as discussed in \cite{CW_book} and \cite{malkiewich2023parametrizedlowtech}. 
We will review only one of these adjunctions here.

Let $X = (X,p,s)$ be an ex-$G$-space over $A$. Then there is an ex-$G$-space over $B$ denoted by 
\[F_!X = (X\cup_A B, q , t)\] 
defined by the diagram
\begin{align*}\label{eq:pushout}\xymatrix{ A \ar[r]^-F \ar[d]_-s& B \ar[d]^-t  \ar@/^2pc/[dd]^-{\id_B}&   \\
X  \ar[d]_-{p}\ar[r] & X\cup_A B \ar[d]^-{q} & \\
 A \ar[r]_-F &   B}\end{align*}
where the top square is a pushout and the projection $q$ to $B$ is the map obtained from the universal property of the pushout.  In this way we get a functor
\[F_! \colon \PTopb{G}{A} \to \PTopb{G}{B}. \]

\begin{example}
For any base $B$, we can apply this construction to the map $\rho \colon B \to G/G$. Then 
\[\rho_! X =  X/s(B)\]
is an ex-$G$-space over the point. The projection map is the unique map to the point and the section is the one that picks out the quotiented $s(B)$.
We thus get a functor to based $G$-spaces
\[\rho_! \colon \PTopb{G}{B} \to \PTopb{G}{*}. \]
\end{example}

As mentioned in \cite{malkiewich2023parametrizedlowtech}, one disadvantage of the classical Thom space construction is that it loses the information of parametrization over the base. The remedy is to consider the spherical bundle:
\begin{example}\label{ex:ThB}
Let $\xi \colon E \to B$ be a vector bundle. Let  $\Th_B(\xi)$ be the spherical bundle over $B$ obtained by one-point compactification on each fiber with section $s$ given by the inclusion of $B$ at infinity. The notation is meant to make one think of $\Th_B(\xi)$ as a ``$B$-relative'' Thom space.  Continuing to let $\rho  \colon B \to G/G$, then 
\[\rho_!(\Th_B(\xi)) = \Th(\xi),\]
where the right hand side is the classical Thom space, $\Th(\xi)=B^\xi$. 
\end{example}

There is also a pullback functor which takes an ex-$G$-space $Y = (Y,q,t)$ over $B$ to the ex-$G$-space 
\[F^*Y =(Y\times_B A, p,s) \]
over $A$ defined by the diagram
\begin{align}\label{eq:pullback}\xymatrix{ A \ar@/_2pc/[dd]_-{\id_A} \ar[r]^-F \ar[d]_-s& B \ar[d]^-t  &   \\
Y\times_B A  \ar[d]_-{p}\ar[r] & Y \ar[d]^-{q} & \\
 A \ar[r]^-F &   B}
 \end{align}
 where the bottom square is a pullback and $s$ is defined using its universal property. In this way we get a functor
\[F^* \colon \PTopb{G}{B} \to \PTopb{G}{A}. \]
 Moreover, the functors $(F_!, F^*)$ form an adjoint pair, with $F_!$ the left adjoint. Note that $F^*$ also has a right adjoint, but we do not need it now. In fact, these form a Quillen adjoint pair 
 and we get an adjunction
\[\xymatrix{ F_! \colon \hPTopb{G}{A} \ar@<1ex>@{->}[r] & \hPTopb{G}{B} \ar@<1ex>@{->}[l] \colon F^* } .\]

\subsection{Change of Groups}
There are various change-of-groups functors in the parametrized context. For example, if $H \subset G$ is a subgroup, recall that there is a functor
\[ i^*_H \colon \PTopb{G}{} \to \PTopb{H}{}  \]
given by restricting the action of $G$ along the inclusion of $H$. This induces restrictions
\[i^*_H \colon \PTopb{G}{B} \to \PTopb{H}{i^*_HB} .\]
For an $H$-space $A$, we also have a functor
\[i_!^H \colon \PTopb{H}{A} \to \PTopb{G}{G\times_H A} \]
which sends $(X,p,s)$ to $(G\times_H X, G\times_H p, G\times_H s )$. So, when $B$ is a $G$-space, we can use the $G$-map $\varepsilon \colon G\times_H i^*_H B \to B$ given by $\varepsilon(g,b)=gb$ and the base-change functor  to define
an induction functor
\[ G_+\wedge_H (-):=\varepsilon_!\circ i_!^H \colon \PTopb{H}{i^*_HB}  \to \PTopb{G}{B} .  \]
As usual, these functors form a Quillen adjoint pair and we get an adjunction
\[\xymatrix{ G_+\wedge_H (-) \colon \hPTopb{H}{i^*_HB} \ar@<1ex>@{->}[r] & \hPTopb{G}{B} \ar@<1ex>@{->}[l] \colon i^*_H } .\]

The functor $ i^*_H $ also has a right adjoint given by co-induction and constructed through a similar ``fiberwise'' construction. See \cite[\S 2.3]{MaySig}.

\subsection{Parametrized orbits and  spheres}

We introduce ex-$G$-spaces that  play the role of the spaces $G_+ \wedge_H S^V$ in the parametrized context and appear in filtration quotients of cell complexes.

\begin{defn}[Parametrized representation spheres] \label{defn:para-rep-sphere}
Let $V$ be an $H$-representation. Form the spherical bundle 
\[G\times_H S^V \to G/H.\] 
The section at infinity gives this the structure of an ex-$G$-space
over $G/H$. Given any point $b \colon G/H \to B$, we define as in \cite[\S 2.6]{CW_book} an ex-$G$-space over $B$ by
\[G_+ \wedge_H S^{V,b} := b_!(G\times_H S^V) = \pars{G\times_H S^V} \cup_{G/H} B .  \] 

We can visualize this as a space $G\times_H S^V$ together with a copy of $B$ glued along the points $G\times_H \infty$. So it is a copy of $B$ with one sphere attached to $b(gH)$ for each coset $gH$ in $G/H$.
See \cref{fig:bubblespace}.
\end{defn}

In the case when $S^V=S^0$, these parametrized spheres play the role of orbits in $\PTopb{G}{B}$.

\begin{defn}[Parametrized orbit]
For $b\colon G/H \to B$, we define the parametrized orbit
\[G/H_+^{b} := G_+\wedge_H S^{0,b},\]
as in \cref{defn:para-rep-sphere}.
\end{defn}

\begin{rem}[Whiskering]\label{rem:whisk}
In order to compute maps between these parametrized representation spheres over $B$, it will be helpful to work with homotopy equivalent spaces using an interval.

Again, let $V$ be an $H$-representation and let $b \colon G/H \to B$ as in \cref{defn:para-rep-sphere}.  Instead of gluing the point at infinity of each sphere to $b(gH)$, we glue the point at infinity to $1\in [0,1]$ and glue $b(gH)$ to $0 \in [0,1]$. In this way we get the ``balloon space'' (thinking of the intervals as strings and the spheres as balloons), which has $G\times_H (S^V \cup [0,1])$ glued to $B$ along the $G\times_H \{0\}$ orbit. 
See  again \cref{fig:bubblespace}.
 
In the language of May--Sigurdsson, we have applied the \emph{whiskering functor} $W$ of \cite[\S 8.3]{MaySig}\footnote{See also \cite[\S 2.5]{malkiewich2023parametrizedlowtech}.} to $G_+\wedge_H S^{V,b}$. There is a natural map 
\[
  W\pars{G_+\wedge_H S^{V,b}} \to G_+\wedge_H S^{V,b},
  \] 
which is an equivalence in $\hPTopb{G}{B}$.  We will often work implicitly with the whiskered replacement.
\end{rem}

\begin{figure}[h]
\begin{center}
\includegraphics[width=0.5\textwidth]{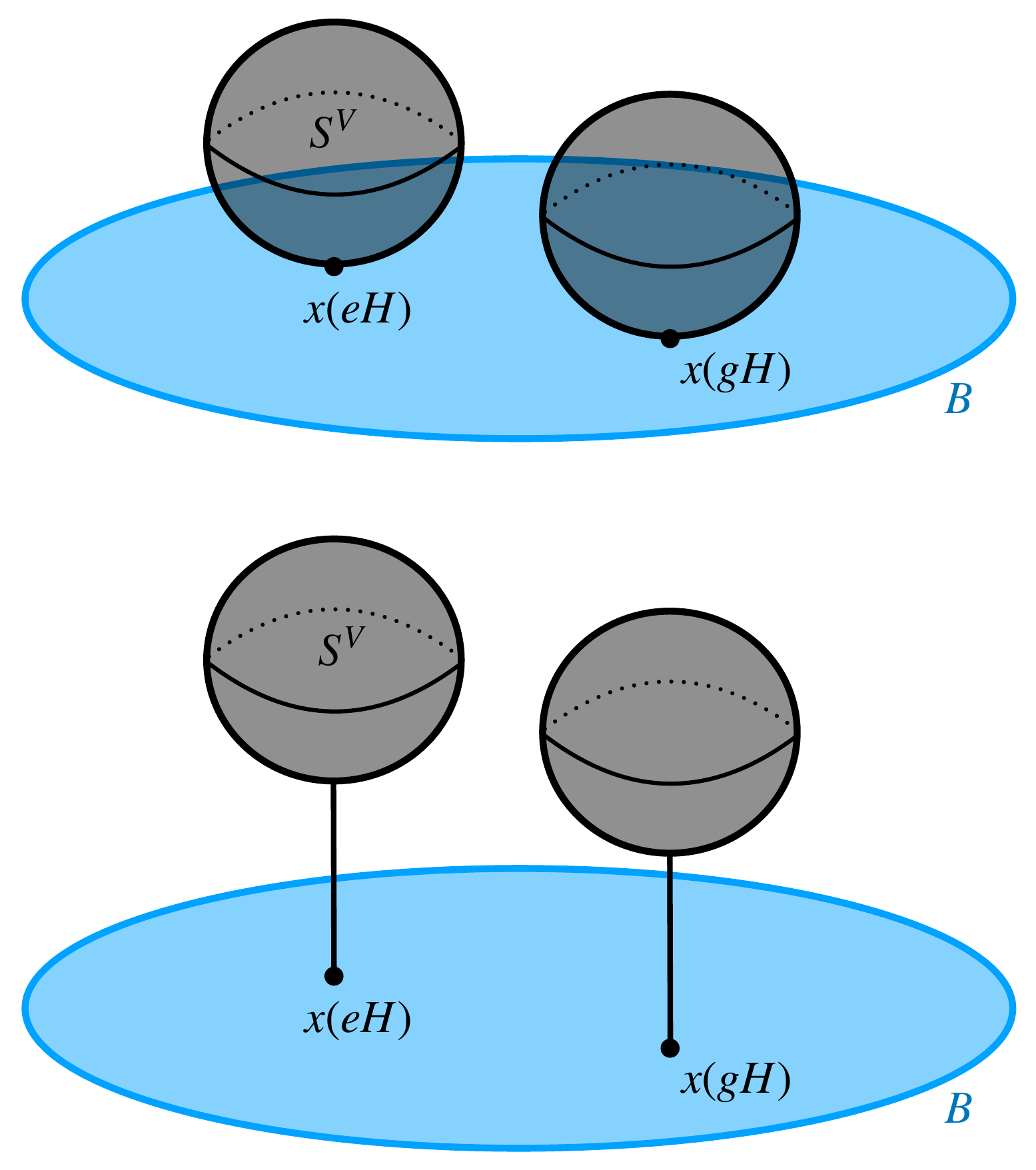}
\caption{$G_+\wedge_H S^{V,b}$ and its whiskered replacement}
\label{fig:bubblespace}
\end{center}
\end{figure}

\subsection{Spectra and Stable Maps}
We now move to the stable world and introduce a little bit of background on parametrized spectra.

A $G$-spectrum $E$ parametrized by $B$, also referred to as a $G$-spectrum over $B$, is the following data. For each finite dimension orthogonal representation $V$ in a complete $G$-universe, we have an ex-$G$-space $E(V)$ over $B$. Given $V\subset W $, there is a structure map
\[\Sigma_B^{W-V}E(V)\to E(W).\]
This gives rise to a parametrized version of the stable homotopy category, which we denote by $\mathcal{SH}^{G}_{B}$.  The morphisms from $E$ to $F$ in $\mathcal{SH}^{G}_{B}$ are denoted by
\[[E,F]^G_B :=\PSH{G}{B}(E,F). \]

There are the usual suspension and loop-space functors, with an adjunction
\[\xymatrix{ \Sigma^{\infty}_B \colon \hPTopb{G}{B} \ar@<1ex>@{->}[r] & \PSH{G}{B}\ar@<1ex>@{->}[l] \colon \Omega_B^\infty } .\]
As is usual, if $X$ is an ex-$G$-space over $B$, we simply write $X$ for $\Sigma^\infty_B X$.

Our goal here is to give a concrete description of stable maps 
\[[ G_+\wedge_H S^{V,b},  E]^G_B\]
for $b\colon G/H \to B$ a point of $B$ and $E$ a $G$-spectrum over $B$. We will use this group of stable maps in the definitions of stable homotopy groups and cellular homology. 

The description we give is a consequence of \cite[\S 2.5]{CW_book}, in particular the discussion after Corollary 2.5.17. To state it, we need to introduce the concept of \emph{lax maps}.

A \emph{lax map}  of ex-$G$-spaces over $B$ is a $G$-map $f \colon X \to Y$ with the property that, in the following diagram,
\[\xymatrix{ &B \ar[dr]^-s \ar[dl]_-s &\\
X \ar[rr]^f \ar[dr]_-p & & Y \ar[dl]^-p \\
& B & }\]
the top triangle commutes and the bottom triangle commutes up to (Moore) homotopy rel $s(B)$.  
A \emph{lax homotopy} is a lax map $X \wedge I_+ \to Y$.

\begin{rem}
In \cite{CW_book}, maps in $\PTopb{G}{B}$ are called strict maps. 
\end{rem}

\begin{defn}
Let $\pPTopblax{G}{B}$ be the category whose objects are ex-$G$-spaces over $B$ and morphisms are lax homotopy classes of lax maps.\footnote{In \cite{CW_book}, the authors use the decoration ``$\lambda$'' to indicate when they are working with lax maps and lax homotopies.
} 
\end{defn}

We can now give a description of certain stable maps. See \cite[Prop. 2.5.16]{CW_book} and the discussion immediately following it.
\begin{theorem}
Let $b \colon G/H \to B$ and $E$ be a $G$-spectrum over $B$. Let $V \in RO(H)$. 
There is an isomorphism
\[[G_+\wedge_H S^{V,b},E ]_{B}^G \cong \colim_{W} \pPTopblax{G}{B}(G_+\wedge_H S^{V,b}\wedge S^W, E(W)) \]
where the colimit runs over the $G$-representations $W$ in our complete $G$-universe.
\end{theorem}

For a $G$-map $F \colon A \to B$, we also get the base-change adjunction
\[\xymatrix{ F_! \colon \PSH{G}{A} \ar@<1ex>@{->}[r] & \PSH{G}{B} \ar@<1ex>@{->}[l] \colon F^* } \]
and the induction/restriction change-of-groups for $H\subset G$
\[\xymatrix{ G_+\wedge_H - \colon \PSH{H}{i^*_HB} \ar@<1ex>@{->}[r] & \PSH{G}{B} \ar@<1ex>@{->}[l] \colon i^*_H } .\]
In addition, there is a Wirthm\"uller isomorphism, so that $G_+\wedge_H-$ is also a right-adjoint to $i^*_H$.

\subsection{Restrictions and transfers}\label{cont:map}

We will give two constructions between parametrized orbits, which correspond to restrictions (or right-way maps) and transfers (or wrong-way maps). These restrictions and transfers will be twisted by a representation $\gamma \in \RO(\Pi B)$. For $B=G/G$ and $\gamma=0$, these maps are just the usual restrictions and transfers in $G$-spectra. In the next section, these constructions will allow us to define homotopy groups as coefficient systems.  

Many of the key points in the construction of parametrized restrictions and transfers can be found already in the proof of Lemma 4.2 in \cite{CW_Duality}, although they are not working in the parametrized setting there. 
 The constructions we present are extracted from Section 2.6 in \cite{CW_book}.
For simplicity, we describe the restrictions and transfers only for actual representations.

Throughout this section, we let
\begin{enumerate}[(a)]
\item $x \colon G/H \to B$ and $y\colon G/K \to B$ be objects of $\Pi B$,
\item  $f=(\alpha,\omega) \colon x \to y$ be a morphism in $\Pi B$,  
\item $\gamma \colon \Pi B \to \cV_G$ be an actual representation, 
\item with $\gamma_0(x)$ the fiber over the identity coset $eH$ of $\gamma(x)$, so that
\[\gamma(x)\cong G\times_H \gamma_0(x),\]
 and similarly for $y$. 
\end{enumerate}

We construct stable maps
\[\res_\gamma(f) \colon  G_+\wedge_H S^{\gamma_0(x),x}  \to G_+\wedge_K S^{\gamma_0(y),y} \]
as follows. Since $\gamma$ is a representation, it associates to $f=(\alpha,\omega)$ a 
map
\[\gamma(f)  \colon \gamma(x) \to \gamma(y),\]
i.e., a $G$ homotopy class of a $G$-bundle map
\[\xymatrix@C=3pc{ G \times_H S^{\gamma_0(x)} \ar[r]^-{\gamma(f)} \ar[d]&  G\times_K  S^{\gamma_0(y)} \ar[d] \\
G/H \ar[r]^-\alpha & G/K.
}\]
First we apply the whiskering functor to write 
\[G_+\wedge_H S^{\gamma_0(x),x} \simeq B\cup_{x} \pars{G \times_H [0,1]} \cup_{G\times_H \infty} \pars{G \times_H S^{\gamma_0(x)} } \] 
as in \cref{rem:whisk}, where we glue onto $B$ strings attached at the orbit and a sphere on the end of each string. Then the identity on $B$, the path $\omega$ on the whiskers $G \times_H [0,1]$ and the bundle map $\gamma(f)$ on $G \times_H S^{\gamma_0(x)}$ glue to give a lax map
\begin{align}\label{eq:resV}
    \res_\gamma(f) \colon G_+\wedge_H S^{\gamma_0(x),x} \to G_+\wedge_K S^{\gamma_0(y),y} .
\end{align}

\begin{defn}[Restrictions]
The \emph{$\gamma$-twisted restriction of $f$}  is the stable map
\[ \res_\gamma(f)  \in [G_+\wedge_H S^{\gamma_0(x),x}, G_+\wedge_K S^{\gamma_0(y),y} ]_{B}^{G} .\]
 Let
\[\res(f) = \res_0(f)\]
be the non-twisted restriction. 
\end{defn}

\begin{figure}[h]
\begin{center}
\includegraphics[width=\textwidth]{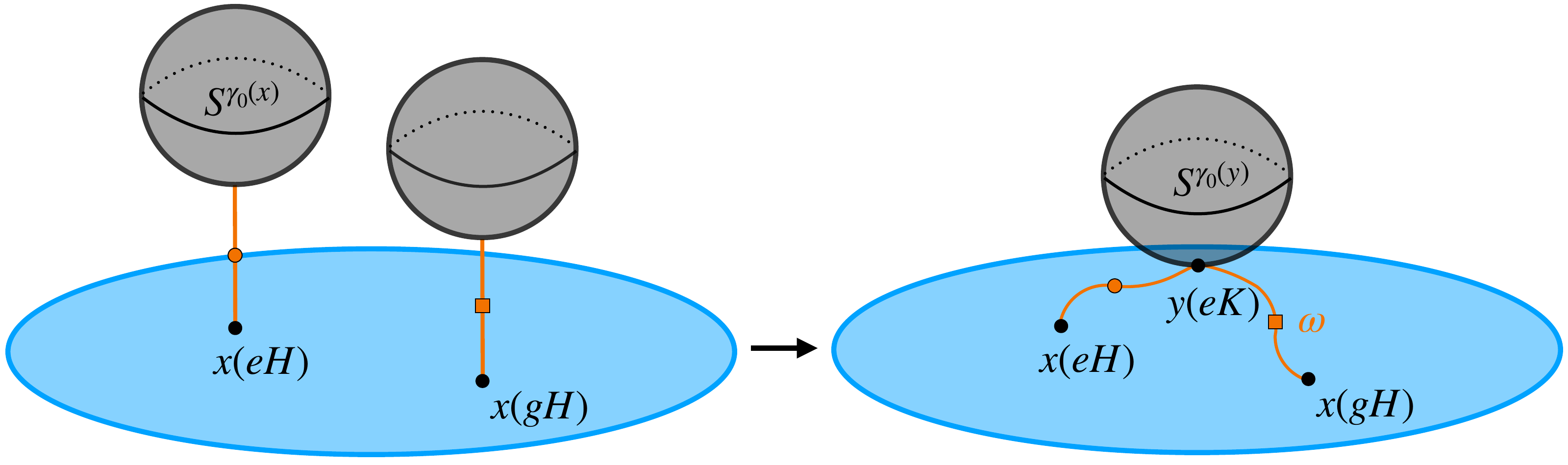}
\caption{The restriction} 
\label{fig:mapballoon}
\end{center}
\end{figure}

We now turn to the transfers. Our goal is to construct stable maps
\[  \tr_\gamma(f) \colon  G_+\wedge_K S^{\gamma_0(y),y}  \to   G_+\wedge_H S^{\gamma_0(x),x},\]
for any morphism $f\colon x\to y$ in $\Pi B$. This construction can be found in the proof of Theorem 2.6.4 in \cite{CW_book}. 

Recall that any map of orbits decomposes as an automorphism of conjugacy classes followed by a quotient.  We first construct  $\tr_\gamma(q)$ in the special case of a quotient map $q=(\rho,c) \colon y\rho \to y$
for $y\colon G/K \to B$, $L\subset K$ a subgroup, $\rho \colon G/L \to G/K$ the quotient, and $c$ the constant homotopy at $y\rho \colon G/L \to B$. The map $q$ is represented by the diagram
\[\xymatrix{ G/L \ar[d]_-{y\rho}\ar[r]^-\rho\ar@{=>}[dr]^-{c} & G/K \ar[d]^-y\\
B \ar@{=}[r] & B.
}\]

In this case, the transfer for $q$ is defined in a familiar  way using the Pontryagin--Thom collapse map, as we explain here. 
View $\rho \colon G/L \to G/K$ as the fiber bundle $G\times_K K/L \to G/K$. Choose an embedding $K/L \subset \gamma_0(y)+i^*_KV $ for a $G$-representation $V$.
Let $W=\gamma_0(y)+i^*_KV$.
Then $G/L$ embeds in the bundle $G\times_K W$. Consider the normal bundle to $K/L\subset W$. The Pontryagin--Thom collapse for the embedding followed by a shearing isomorphism gives a $K$-map 
\begin{align}\label{eq:trnew}\xymatrix{S^{\gamma_0(y)}\wedge S^{i^*_KV} \ar@{=}[d]\ar[rr] & &  K_+\wedge_L S^{i^*_L \gamma_0(y)}\wedge S^{i^*_KV}
\\
S^W \ar[r]^-{\text{Collapse}} &  K_+ \wedge  S^W \ar@{=}[r] &  K_+\wedge S^{\gamma_0(y)} \wedge S^{i^*_KV}. \ar[u]^-{\text{Shear}\wedge \id}
 }\end{align}
This then gives a $G$-map 
\[c_V\colon G\times_K (S^{\gamma_0(y)}\wedge S^{i^*_KV}) \to G\times_K ((K_+\wedge_L S^{i^*_L\gamma_0(y)}) \wedge S^{i^*_KV}).  \]
The strange mix of smash and times is intentional. In the target of $c_V$, $G\times_K (K_+\wedge_L S^{i^*_L\gamma_0(y)} \wedge S^{i^*_KV})$, each sphere in the wedge $K_+\wedge_L S^{i^*_L\gamma_0(y)} \wedge S^{i^*_KV}$
is attached to the same point of $B$, namely $y\rho(kL)  = y(K)$. So this glues along $B$ to gives a lax map
\[ \xymatrix@C=1pc{
G\times_K (S^{\gamma_0(y)}\wedge S^{i^*_KV})\cup_y B \ar[r] \ar[d]_\cong &  G\times_K ((K_+\wedge_L S^{i^*_L\gamma_0(y)}) \wedge S^{i^*_KV}) \cup_{y\rho} B
\ar[d]_\cong \\
G_+\wedge_K S^{\gamma_0(y),y} \wedge S^V \ar[r] &  G_+\wedge_L S^{i^*_L\gamma_0(y),y\rho} \wedge S^V \ar[d]^-\cong \\
& G_+ \wedge_L S^{\gamma_0(y\rho),y\rho} \wedge S^V
}\]
where in the last vertical arrow, we used the inverse of the isomorphism 
\[  \gamma_0(y\rho) \xrightarrow{\cong} i^*_L\gamma_0(y)\]
induced by $\gamma(q)$ to identify
\[G_+ \wedge_L S^{\gamma_0(y\rho),y\rho}  \cong G_+ \wedge_L S^{i^*_L\gamma_0(y),y\rho} .\]
Thus, we have a lax map
\[\tr_\gamma(q)(V)\colon  G_+\wedge_K S^{ \gamma_0(y),y} \wedge S^V \to G_+\wedge_L S^{\gamma_0(y\rho),y\rho} \wedge S^V . \]
We get a corresponding stable map
\[\tr_\gamma(q) \colon G_+\wedge_K S^{\gamma_0(y),y} \to G_+\wedge_L S^{\gamma_0(y\rho),y\rho} .\]

To define the transfer for a general $f\colon x\to y$, we factor $f$ as $f=qh$, using the fact that maps of orbits factor as an isomorphism followed by a quotient.  So we let 
$h=(\beta,\omega)  \colon x \to y\rho$ be an isomorphism, and 
we let $q = (\rho,c)\colon y\rho\to y$ for  $y\rho \colon G/L \to B$, where $L\subset K$ is conjugate to $H$, $\rho$ is the quotient map,  and $c$ is the constant homotopy at $y\rho$. Therefore, $f = (\alpha, \omega)$ can be represented as the composite in the diagram
\[ \xymatrix{ G/H\ar[d]^-x \ar[r]^-{\beta} \ar@{=>}[dr]^-\omega & G/L \ar[d]^-{y\rho} \ar[r]^-{\rho} \ar@{=>}[dr]^-{c} & G/K \ar[d]^y   \\
B \ar@{=}[r] & B \ar@{=}[r] &B.   }\]
We now appeal to the fact that the transfer of the isomorphism $h$ should be equal to the restriction of its inverse to finish the construction of the transfer of $f$. 
\begin{defn}[Transfers]
The \emph{$\gamma$-twisted transfer of $f=qh$} is the stable map
\[ \tr_\gamma(f) \in [G_+\wedge_K S^{\gamma_0(y),y},G_+\wedge_H S^{\gamma_0(x),x} ]_B^G\]
obtained as the composition
\[\tr_\gamma(f) := \res_\gamma(h^{-1})\tr_\gamma(q),\]
as described above.
\end{defn}

\begin{rem}
We can adapt these definitions to representations of the form $ \gamma+n$ for $\gamma$ actual and $n\in \Z$ 
using fiberwise suspension by $S^n$ throughout. For more general  representations, the definitions get more complicated due to the technicalities involved with defining the morphisms in $\vV_G$.
\end{rem}

\subsection{Stable orbit category of $B$ and parametrized Mackey functors}\label{sec:paramackey}

In $\RO(G)$-graded cohomology theories, we require coefficients in a  Mackey functor, an equivariant analogue of an abelian group. Parametrized equivariant cohomology requires coefficients in a parametrized Mackey functor over $B$. In this section we begin by recalling some categorical notions related to Mackey functors in the classical sense. We then describe how one translates these notions to the $\RO(\Pi B)$-graded setting.

There are many definitions of Mackey functors; we present multiple perspectives in this section and work with whichever is most convenient when explaining the connection to parametrized Mackey functors. We start by reviewing span categories and their Grothendieck completions.

Throughout this section, we let $\mathcal{C}$ be a small category with pullbacks and finite coproducts, such that the pullback distributes over the coproduct. 

\begin{defn}\label{def:span}
Define $\Span(\mathcal{C})$ to be the category whose objects are the objects of $\mathcal{C}$. The morphisms in this category are equivalence classes of spans 
\[ \xymatrix{ X & Z \ar[r] \ar[l]& Y,} \]
where two spans $X \leftarrow Z \rightarrow Y$ and $X \leftarrow W \rightarrow Y$ are equivalent if there exists an isomorphism $Z \cong W$ such that the following diagram commutes: 
\[\xymatrix@R=0.5pc{ &Z \ar[dd]^-\cong \ar[dr] \ar[dl] & \\
X  & & Y \\
& \ar[ur] \ar[ul] W & }\]
Morphism composition in this category is given by pullbacks. 
\end{defn}

The category $\Span(\mathcal{C})$ is not a pre-additive category, so as usual we define another category to fix this.
\begin{defn}
The category $\Span^+(\mathcal{C})$ has the same objects as $\Span(\mathcal{C})$. The morphisms $\Span^+(\mathcal{C})(X,Y)$ are given by the Grothendieck group completion of $\Span(\mathcal{C})(X,Y)$, where the monoid operation on spans is obtained from the coproduct
    \[\xymatrix{X & Z_1 \coprod Z_2 \ar[r] \ar[l] & Y.}\]
\end{defn}

Applying this construction to the category of finite $G$-sets $\Fin^G$, we get the Burnside category.
\begin{definition}\label{def:Burnside}
The \emph{Burnside category} of $G$ is the category
    \[
    \Burn{G} := \Span^+(\Fin^G).
    \] 
    Let $\sorb{G}$ be the full subcategory of $\Burn{G}$ whose objects are orbits.  We call $\sorb{G}$ the \emph{stable orbit category}.
\end{definition}

The name stable orbit category for $\sorb{G}$ is justified by a result
 in \cite[Corollary V.9.4]{LMS}.  There the authors show that  $ [G/H_+, G/K_+]^G$ can be described as the free abelian group on equivalence classes of spans in $\cO_G$, so that
 \begin{align}\label{eq:staborb}
 \Burn{G}(G/H,G/K) =\sorb{G}(G/H,G/K)\cong [G/H_+, G/K_+]^G .
 \end{align}
 From this, it follows that:
 \begin{lem}\label{lem:staborb}
The category $\sorb{G}$ is isomorphic to the full subcategory of the equivariant stable homotopy category $\mathcal{SH}^G$ whose objects are the suspension spectra 
 of orbits $G/H_+$. Equivalently, $\sorb{G}$ is isomorphic to a category whose objects are those of $\cO_G$, and in which the  morphisms from $G/H$ to $G/K$ are the stable maps $[G/H_+,G/K_+]^G$.
\end{lem}

\begin{rem}
In \cite[Ch. IX, \S 4]{MayAlaskanotes}, the category of \cref{lem:staborb} is referred to as both the Burnside category and the stable orbit category.
\end{rem}
We can now state the definition of Mackey functors that we will generalize in the parametrized context.
\begin{defn}[{\cite[Ch. IX, \S 4]{MayAlaskanotes}}]\label{def:staborb}
    A $G$-\emph{Mackey functor} is an additive functor
    \[ \xymatrix{ \uM \colon \pars{\sorb{G}}^\op \ar[r] & \mathrm{Ab}.} \] 
    Morphisms of Mackey functors are natural transformations.
\end{defn}

\begin{rem}\label{rem:mackeyfunctorburn}
Recognizing that any finite $G$-set is isomorphic to a disjoint union of orbits, we see that $\Burn{G}$ is generated by $\sorb{G}$ under disjoint unions. 
As a consequence,
the category of $G$-Mackey functors is equivalent to the category of additive   functors
    \[ \xymatrix{  \uM \colon \pars{\Burn{G}}^{\op} \ar[r] & \mathrm{Ab}} \]
that send disjoint unions to direct sums.
\end{rem}

More concretely, a $G$-Mackey functor $\uM$ is the data of an abelian group $\uM(G/H)$ for each orbit $G/H$, together with transfer maps and restriction maps subject to some composition rules (see for example \cite{Webb}). 
We explain how these maps arise from the Burnside category. 

\begin{defn}
For $\cC$ as in \cref{def:span} and a $\cC$-morphism $f\colon X \to Y$, define $\res(f) \colon X \to Y$  and $\tr(f) \colon Y \to X$  in $\Span(\cC)$ to be the spans
\begin{align*} \res(f)&:=\xymatrix@R=0.5pc{ X &X \ar[l]_-{\id_X} \ar[r]^f& Y  } & \tr(f) &:= \xymatrix@R=0.5pc{ 
 Y &X \ar[r]^-{\id_X} \ar[l]_f &X   }
 \end{align*}
 We also denote the images of $\tr(f)$ and $\res(f)$ in $\Span^+(\cC)$ by the same name. 
 \end{defn}
For any morphism of finite $G$-sets $f\colon G/K \to G/H$, the \emph{transfer} in the Mackey functor is
\[f_*:=\uM(\mathrm{tr}(f)) : \uM(G/K) \to \uM(G/H)\] 
and the \emph{restriction} is
\[f^*:=\uM(\mathrm{res}(f)): \uM(G/H) \to \uM(G/K).\] 
 This perspective on Mackey functors lends itself nicely to presentations in Lewis diagrams, as we shall see in the case of parametrized Mackey functors.

\bigskip

We are now ready to explain how to use the work of Costenoble--Waner in \cite[\S 2.6]{CW_book} to define the notion of Mackey functors parametrized by $B$.

We want to define a Burnside category $\BBurn{G}{B}$ for a $G$-space $B$. First, we need the following definition which will replace the role of finite $G$-sets in the classical construction of the Burnside category. 
\begin{defn}
Let $B$ be a $G$-space and $\BFin{G}{B}$ be the category whose objects are continuous functions $x \colon X \to B$, where $X$ is a finite $G$-set. For $y\colon Y \to B$ another object, a morphism $f\colon x \to y$ in $\BFin{G}{B}$ is an equivalence class of  pairs $f=(\alpha, \omega)$ where $\alpha \colon X \to Y$ is a morphism of finite $G$-sets and $\omega$ is a homotopy from $x$ to $y\circ \alpha$. Morphisms $f=(\alpha,\omega)$ and $f'=(\alpha',\omega')$ are equivalent if $\alpha=\alpha'$, and $\omega$ and $\omega'$ are homotopic relative their endpoints. 
\end{defn}

\begin{lem}
The category $\BFin{G}{B}$ has pullbacks and finite coproducts, and the pullback distributes over the coproduct.
\end{lem}
\begin{proof}
Let $x \colon X \to B$, $y \colon Y \to B$ be objects of $\BFin{G}{B}$. The coproduct of two objects is then
\[x \coprod y \colon X \coprod Y \to B\]
where $X \coprod Y$ is the coproduct in $\FinG$. 

For pullbacks, suppose $z \colon Z \to B$ is another object in $\BFin{G}{B}$ and that we have morphisms $f_x = (\alpha_x, \omega_x) \colon x \to z$ and $f_y = (\alpha_y, \omega_y)\colon y\to z$. 
Let $X\times_Z Y$ be the pullback in $\Fin^{G}$
\[\xymatrix{
X\times_Z Y \ar[r]^-{\pi_Y}\ar[d]_-{\pi_X} & Y \ar[d]^-{\alpha_y} \\ 
X \ar[r]_{\alpha_x} & Z,
}\]
so that the points of $X\times_Z Y$ are pairs $(a_x,a_y) \in X \times Y$ such that
\[\alpha_x(a_x) = \alpha_y(a_y) \in Z.\]
Let
\[x\times_Z y \colon X\times_Z Y  \to B \]
be the map
$x\times_Z y:= z \alpha_x \pi_X = z \alpha_y \pi_Y$.
We have maps
\begin{align*}
g_x &= (\pi_X, \omega_x^{-1} (\pi_X  \times \id_I)) \colon x\times_Z y \to x  \\
g_y &= (\pi_Y, \omega_y^{-1}  (\pi_Y \times\id_I  )) \colon x\times_Z y \to y,  
\end{align*}
where here $\omega_x^{-1}(a_x,t) =\omega_x(a_x,1-t)$ and similarly for $\omega_y^{-1}$. This
gives a diagram
\[\xymatrix{
x\times_z y \ar[r]^-{g_y} \ar[d]_-{g_x}  & y  \ar[d]^-{f_y}\\
x  \ar[r]_-{f_x} & z
}.\]
It is a nice exercise to prove that this satisfies the universal property of the pullback. Note that it is important in that argument to use that the homotopy morphisms in $\BFin{G}{B}$ are homotopy classes of maps. It is also straightforward to show that pullbacks distribute over coproducts.
\end{proof}

\begin{defn}
We let $\BBurn{G}{B}$ denote the category 
\[\BBurn{G}{B} = \Span^+(\BFin{G}{B}) \]
and call it the \emph{Burnside category} of $B$. We let
\[\Bsorb{G}{B} \subset \BBurn{G}{B}\]
be the full-subcategory whose objects are those of $\Pi B$. That is, the objects of $\Bsorb{G}{B}$ are morphisms $x \colon G/H \to B$ whose source is an orbit. 
\end{defn}

We will later call $\Bsorb{G}{B}$ the \emph{stable orbit category of $B$}, and this will be justified in a similar way as in \cref{def:Burnside}.

\begin{rem}
Morphisms 
\[\Bsorb{G}{B} (x,y) =\BBurn{G}{B}(x,y) \]
can be described as the free abelian group on equivalence classes of spans 
\[\xymatrix{x & z \ar[r] \ar[l] &y} \]
in $\Pi B$.
Here, $x \colon G/H \to B$, $y \colon G/K \to B$, and $z \colon G/L \to B$.
\end{rem}

The next goal is to connect $\Bsorb{G}{B}$ to a category with the same objects but whose morphisms are stable maps. 
\begin{defn}[{\cite[Def. 2.6.3]{CW_book}}]
For any representation $\gamma \in RO(\Pi B)$, the \emph{$\gamma$-twisted stable fundamental groupoid of $B$}, denoted $\widehat{\Pi}_\gamma B$, is the following category: objects are the same as those of  $\Pi B$, and morphisms are given by the stable maps
\[\widehat{\Pi}_\gamma B (x,y) = [G_+\wedge_H S^{\gamma_0(x),x}, G_+\wedge_K S^{\gamma_0(y),y} ]_{B}^{G},\]
where, as before, $\gamma_0(x)$ is the fiber over the identity coset in the bundle $\gamma(x)$. 
When $\gamma$ is zero, we call the category $\widehat{\Pi}_0 B$  the \emph{stable fundamental groupoid} of $B$.
\end{defn}
\begin{rem}
Since
\[\widehat{\Pi}_0 B(x,y) = [G/H_+^x, G/K_+^y]^{G}_B\]
this directly generalizes the category of \cref{lem:staborb}.
\end{rem}

Comparison functors between the stable orbit category of $B$ and the $\gamma$-twisted stable fundamental groupoids will play an important role throughout the rest of the paper. These functors are implicit in the treatment in \cite{CW_book}.
\begin{defn}\label{def:gammagamma}
Let 
\[\Gamma_\gamma \colon \Bsorb{G}{B} \to \widehat{\Pi}_\gamma B \]
be the additive functor which is the identity on objects, and on morphisms sends a span $s = x \xleftarrow{q} z \xrightarrow{p}y$ to
\[\Gamma_\gamma(s)= \res_{\gamma}(p)\circ \tr_\gamma(q) .\]
\end{defn}

We then have the following crucial result, which generalizes \eqref{eq:staborb} and  \cref{lem:staborb}.

\begin{theorem}[{\cite[Theorem 2.6.4]{CW_book}}]\label{thm:gammagamma}
For any $\gamma\in \RO(\Pi B)$, the functor $\Gamma_\gamma$ is an isomorphism of categories
\[ \Gamma_\gamma \colon\xymatrix{ \Bsorb{G}{B} \ar[r]^-{\cong}  &\widehat{\Pi}_\gamma B}. \]
\end{theorem}

\begin{remark}  In \cite{CW_book}, the category $\Bsorb{G}{B}$ is not explicitly defined (although the category $\Bsorb{G}{B}$ does appear in \cite[p.\ 333]{CW_Duality}, where it is denoted $\hat{\pi} X$, replacing $B$ with $X$).  This is because \cite{CW_book} treats the more general setting of compact Lie groups, and in that setting, the composition of spans is harder to define.

Rather, in the statement of Theorem 2.6.4 in \cite{CW_book}, the authors identify the morphisms in $\widehat{\Pi}_\gamma B$ as equivalence classes of certain kinds of spans.  For finite groups, this description is easy to relate to the morphisms of $\Bsorb{G}{B}$.   Indeed, their spans have a lax map in the right leg of the spans, but a morphism $(z\colon G/L \to B) \to (y\colon G/K \to B)$ in $\Pi B$ gives rise to a lax map of orbits over $B$. Namely,  $\alpha\colon G/L \to G/K$ is the map over $B$ and the path $\omega$ witnesses the fact that $\alpha$ does not commute with the maps $z,y$ to $B$, but rather commutes with them up to homotopy.
    \end{remark}

We can now generalize the definition of Mackey functors to the parametrized context. These will be the coefficients used in parametrized equivariant cohomology.

\begin{defn}\label{defn:parmackey}
A \emph{parametrized $G$-Mackey functor over $B$} is an additive functor
\[\uM  \colon \pars{\Bsorb{G}{B}}^{\op} \to \mathrm{Ab}.\]
\end{defn} 
\begin{rem}
If $B=G/G$, then a parametrized Mackey functor over $B$ is the same as a Mackey functor.
\end{rem}

\begin{warn}
Costenoble--Waner typically call their coefficients \emph{$\widehat{\Pi}_0B$-modules} 
rather than
``parametrized $G$-Mackey functors over $B$'', and they consider both covariant and contravariant versions. The variance is not so important in the case of a finite group $G$.
\end{warn}

\begin{rem}
As in \cref{rem:mackeyfunctorburn}, a parametrized $G$-Mackey functor over $B$ can be described equivalently as an an additive functor
\[\uM\colon \pars{\BBurn{G}{B}}^\op \to \Ab\]
which takes disjoint unions to direct sums.
\end{rem}

From a parametrized Mackey functor $\uM$, we can extract two functors
\begin{align*}
 \uM^* \colon \Pi B & \to  \mathrm{Ab}  & \uM_* \colon \Pi B & \to  \mathrm{Ab} 
 \end{align*}
 where $ \uM^*$ is contravariant and $ \uM_*$ is covariant. The functor $\uM^*$ is obtained by restricting $\uM$ along the covariant embedding $\Pi B \to \Bsorb{G}{B}$, which is the identity on each object and sends each morphism $f$ to its restriction $\res(f)$. The functor $\uM_*$ is obtained by restricting $\uM$ along the contravariant embedding $\Pi B^{\op} \to \Bsorb{G}{B}$  which is the identity on objects and sends $f$ to $\tr(f)$.  These functors agree on objects
 \[ \uM^*(x) = \uM(x) =  \uM_*(x) \]
 and for a morphism $f \colon x \to y$ in $\Pi B$,
\begin{align*}
f^*:= \uM^* (f) &= \uM(\res(f))  &  f_*:=\uM_* (f) &= \uM(\tr(f)).
 \end{align*}

Suppose that $\uM$ is a parametrized Mackey functor over $B$ and $F \colon A \to B$ is a $G$-map. Since the construction of spans and Grothedieck completion are functorial, we can pull back $\uM$ along $F$ to  a parametrized Mackey functor over $A$ 
\[F^*\uM  = \uM \circ \Span^+(F).\]

We now examine how to restrict and induce Mackey functors for a subgroup $H\subseteq G$.
Recall from \cref{defn:cog-PiB} that, if $H\subset G$ is a subgroup, we get a functor
\begin{equation}\label{eq:inmack} G\times_H  -  \colon \Pi i^*_HB \to \Pi B.\end{equation}
Since $\Span^+(-)$ is a functorial construction we get the following induction on Mackey functors.

\begin{defn}\label{rem:cog-PigammaB}
Let
\[ G\times_H  -  \colon \Bsorb{H}{i^*_HB} \to \Bsorb{G}{B}\]
be the functor on spans induced by \eqref{eq:inmack}, and let
\[ G\times_H - \colon \widehat{\Pi}_{i^*_H\gamma} i^*_HB \to   \widehat{\Pi}_{\gamma} B\]
be equal to \eqref{eq:inmack} on objects and given by the induction functor for stable $G$-maps on morphisms.
\end{defn}

The following is straightforward to verify:
\begin{lem}
 For $H\subseteq G$ a subgroup, there is a commutative diagram
\[\xymatrix{
 \Bsorb{H}{i^*_HB} \ar[rr]^-{G\times_H -} \ar[d]_-{\Gamma_{i^*_H\gamma}}& &   \Bsorb{G}{B} \ar[d]^-{\Gamma_\gamma}\\
 \widehat{\Pi}_{i^*_H\gamma} i^*_HB  \ar[rr]^-{G\times_H -}  &  & \widehat{\Pi}_{\gamma} B.
}\]
\end{lem}

\begin{defn}\label{defn:cog3} 
Let $B$ be a $G$-space and let $H\subset G$ be a subgroup.  Consider $\uM \colon  \pars{\Bsorb{G}{B}}^\op \to \Ab$ a parametrized $G$-Mackey functor over $B$. Applying the forgetful functor to $\uM$ defines a parametrized $H$-Mackey functor 
\[i^*_H\uM(-) = \uM( G\times_H-) \colon \pars{\Bsorb{H}{i^*_HB}}^\op \to \Ab.\]
\end{defn}

For $N$ an abelian group, we let $\uN$ denote the constant Mackey functor (over $B=G/G$). As a functor $\uN\colon (\Burn{G})^\op \to \Ab$, it assigns to a $G$-set $X$ the group
\[\uN(X) = \hom^G(X,N)\]
of $G$-maps (where $N$ has a trivial action). Restrictions are obtained by precomposition and transfers by summing over fibers. See for example \cite[\S3.2]{HHR}.
When restricted to orbits,  there is thus a canonical identification $\uN(G/H) = N$, and $\uN(\res(\alpha))=\id_N$ for every morphism $\alpha$ in $\cO_G$. If $\rho \colon G/H \to G/K$ is a quotient for $H\subset K$, then $\uN(\tr(\rho))$ is 
multiplication by the index $[K:H]$. Since any morphism in $\cO_G$ is a composition of quotients and isomorphisms and $\tr(\alpha)=\res(\alpha^{-1})$ for an isomorphism $\alpha$, this completely describes the Mackey functor $\uN$ on $\cO_G$.

\begin{defn}[Constant Mackey Functor]\label{def:constantMackey}
Let $N$ be an abelian group. Define the \emph{constant parametrized Mackey functor} over $B$ to be the pullback of the constant Mackey functor $\uN$ along the quotient map $\rho \colon B \to G/G$. We abuse notation and write
\[ \uN = \rho^*\uN \colon  \pars{\Bsorb{G}{B}}^{\op} \to \mathrm{Ab}. \]
\end{defn}

\begin{rem}
We write 
$ \uN$ for $F^*\uN$
for any map $F \colon A \to B$.  Similarly,  $i^*_H\underline{N}$ is again the constant parametrized Mackey functor, so we also write $\uN$ for $i^*_H\underline{N}$.
\end{rem}

The choice of notation in \cref{def:constantMackey} follows a common convention in equivariant homotopy theory, despite the risk that the constant Mackey functor $\underline{N}$ can easily be confused with an arbitrary Mackey functor.  The constant Mackey functors we work with will typically be $\uZ$ or $\mF$, as in the following examples.

\begin{example}\label{ex:lewisZ}
Using the skeleton of \cref{ex:funS11}, we can write a ``Lewis diagram'' for the constant parametrized Mackey functor $\uZ$ over the $C_2$-space $S^{1,1}$.  As usual, in the diagram we only depict the restrictions of the automorphisms $g$ and $t$.
 \[\xymatrix@C=1.5pc{
&  \uZ \colon \pars{\Bsorb{C_2}{S^{1,1}}}^\op \to \Ab& \\
\Z  \ar@/_/[dr]_-1 & & \Z  \ar@/_/[dl]_-1 \\
 & \Z \ar@/_/[ur]_-{2} \ar@/_/[ul]_-{2} \ar@(dr,d)@{>}[]^-{1}\ar@(d,dl)@{>}[]^-{1} & 
}\]
\end{example}

\begin{example}\label{ex:lewisF2}
Similarly, using the skeleton from \cref{ex:funS11} we can write a Lewis diagram for the constant parametrized Mackey functor $\mF$ over the $C_2$-space $S^{1,1}$.  Again, in the diagram we only depict the restrictions of the automorphisms.
 \[\xymatrix@C=1.5pc{
&  \mF \colon \pars{\Bsorb{C_2}{S^{1,1}}}^\op \to \Ab& \\
\F_2  \ar@/_/[dr]_-1 & & \F_2  \ar@/_/[dl]_-1 \\
 & \F_2 \ar@/_/[ur]_-{0} \ar@/_/[ul]_-{0} \ar@(dr,d)@{>}[]^-{1}\ar@(d,dl)@{>}[]^-{1} & 
}\]
\end{example}

\subsection{Stable homotopy groups}

Now that we have introduced all of the required ingredients, we can define parametrized stable homotopy groups as follows. These will be Mackey functors over $B$ and be graded on $\RO(\Pi B)$. 

\begin{defn}
For $\gamma\in \RO(\Pi B)$ and $X$ an ex-$G$-space over $B$, define a parametrized Mackey functor over $B$
\[\upi_\gamma X \colon \pars{\Bsorb{G}{B}}^{\op} \to \Ab \]
as follows. For each $x\colon G/H \to B$, we let
\[\upi_\gamma X (x) =[G_+\wedge_H S^{\gamma_0(x),x},X]_{B}^G. \]
Given a morphism $f\colon  x \to y$ in $\Bsorb{G}{B}$, we let
\[\upi_\gamma X (f) =-\circ \Gamma_\gamma(f) \colon \upi_\gamma X (y) \to \upi_\gamma X (x).\]
We call these the \emph{$\RO(\Pi B)$-graded stable homotopy parametrized Mackey functors of $X$}.
\end{defn}

This is a special case of homotopy groups of parametrized $G$-spectra, as defined in \cite[\S 3.7]{CW_book}.

%% file: witpaper-coh.tex

\section{Parametrized cellular cohomology}\label{sec:coh}
In this section, we review the construction of cellular $\RO(\Pi B)$-graded parametrized cohomology following \cite{CW_book}.  We first give the definition of $\CWg$-complexes along with a few examples, then we review the construction of the cellular chains and define cohomology. We examine in detail the case when $G$ is trivial, and we explain in depth why the Costenoble--Waner theory in \cite{CW_book} corresponds to cellular cohomology with local coefficients.

Often, we will focus on $B_+ =B\sqcup s(B)$ as an ex-space over $B$. Costenoble--Waner explain in \cite[Remark 3.8.5]{CW_book} how one can recover other computations from this example using base-change functors.

\subsection{Cell structures}
Fix a representation $\gamma \in \RO(\Pi B)$.  In \cite[\S 3]{CW_book}, Costenoble--Waner introduce CW-structures twisted by the representation $\gamma$, both for $G$-spaces over $B$ and for ex-$G$-spaces over $B$.  These are referred to as $\CWg$-complexes and ex-$\CWg$-complexes, respectively.  We repeat the definitions here and illustrate some examples. 
 For any actual (as opposed to virtual) $G$-representation $V$, we let $D(V)$ be the unit disk in $V$ and $S(V)$ be the unit sphere in $V$.

First we introduce the parametrized cells, following \cite[\S 3.1]{CW_book}. We simply write $n$ to denote the $n$-dimensional trivial representation. We refer the reader to Definition 1.3.1 of \cite{CW_book} for the notion of stable equivalence of virtual representations. Note that virtual $G$-representations are stably equivalent if and only if they represent isomorphic virtual bundles over $G/G$ in $\vV_G$.

\begin{defn}
A \emph{$\gamma$-cell} is a space over $B$ of the form $(G\times_H D(V),p)$ for $V$ an actual orthogonal $H$-representation satisfying the following condition: restricting to the center of the cell $p_0 = p|_{G\times_H 0}\colon G\times_H 0 \to B$, there is an integer $n$ such that
$\gamma_0(p_0)+n$ is stably equivalent to $V$. Here $\gamma_0(p_0)$ is the (possibly virtual) representation that $\gamma$ assigns to the identity coset at the center of the cell. The \emph{dimension} of the cell is $\gamma+n$ and its \emph{boundary} is $(G\times_H S(V),p)$.
\end{defn}

\begin{rem}
    The stable equivalence condition implies the underlying dimension $|\gamma + n|$ of such a $\gamma$-cell will be the same as the underlying dimension of the representation $V$.    Note further that there is a strict restriction on the centers of $\gamma$-cells. Namely, $\gamma_0(p_0)$ must be stably equivalent to an \emph{actual} orthogonal representation of $H$.\footnote{This motivates the following language (see \cite[Definition 3.2.1]{CW_book}): A point $x \colon G/H \to B$ is called \emph{$\gamma$-admissible} if $\gamma_0(x)+n$ is stably equivalent to an actual orthogonal $H$-representation for some integer $n$.}
\end{rem}

Now that we have $\gamma$-cells, we turn to $\CWg$-complexes.  Typically, a CW-complex is defined inductively by attaching $k$-dimensional cells at the $k$th stage.  Thus, in order to define $\CWg$-complexes, we take into account the underlying dimension $|\gamma|$.  Rather than attaching $\pars{\gamma+k}$-cells at the $k$th stage, we attach $\pars{\gamma - |\gamma| + k}$-cells, whose underlying dimension is $k$.

\begin{defn}[{\cite[Definition 3.2.2]{CW_book}}]
Let $\gamma \in RO(\Pi B)$, and denote the dimension of the fibers of $\gamma$ by $|\gamma|$.
A \emph{$\CWg$-complex} is a space $(X,p)$ over $B$ presented as a colimit
\[ (X,p)= \colim_n (X^{n}, p)\]
of subspaces $(X^{n}, p)$ for $n\geq 0$ where
\begin{enumerate}
\item[(a)] $(X^{0}, p)$ is a union of $(\gamma-|\gamma|)$-cells. 
Moreover, these are unions of points $x\colon G/H \to X$ with the property that $\gamma_0(px)$ is a trivial representation of $H$. 
\item[(b)] $(X^{n},p)$ is constructed from $(X^{n-1},p)$ by attaching $(\gamma-|\gamma|+n)$-cells.
\end{enumerate}
If $(A,p)$ is a subspace of $(X,p)$, a \emph{relative} $\CWg$-structure on $(X,A,p)$  is obtained in the usual way by letting $X^{0}$ be the disjoint union of $A$ with $\pars{\gamma-|\gamma|}$-cells, and proceeding as above for attaching of higher dimensional cells.

An \emph{ex-$\CWg$-structure} on an ex-space $(X,p,s)$ is a relative structure for the pair $(X,s(B),p)$.
\end{defn}

\begin{rem}
To avoid ambiguity with respect to the representation $\gamma$, we let
\[X^{\gamma+n}:= X^{|\gamma|+n},\]
as in the notation of \cite{CW_book}.
\end{rem} 

\begin{rem}
The definition of a $\CWg$-complex already puts rather strong conditions on the possible value of $\gamma_0(p_0)$, since $\gamma_0(p_0) +n$ must be stably equivalent to an actual $H$-representation. 
Moreover, for any $0$-cell we require that $\gamma_0(p_0)$ must be stably trivial, which implies it is a trivial representation. \end{rem}

\begin{convention}
In this paper, we will always compute  with representations $\gamma$ of virtual dimension zero. With this choice, we have $X^{\gamma+n} = X^n$ and  $\pars{\gamma+n}$-cells have underlying dimension $n$.  Requiring $\gamma$ to have virtual dimension zero is not a strong condition since a $\text{CW}(\gamma-|\gamma|)$-complex is equivalent to the data of a $\CWg$-complex.
\end{convention}

\begin{ex}\label{ex:CWstructureG=e}
Let $G=e$ and let $(X,p)$ be a space over $B$.  Let $X$ be a CW-complex in the usual sense with skeleta $X^n$. Let $\gamma\in \RO(\Pi B)$ be 
of virtual dimension zero.  There is a $\CWg$-structure on $(X, p)$ with skeleta $(X^n,p)$ and cells $(D(\R^n),p\circ \varphi)$ where $\varphi \colon D(\R^n) \to X^{n}$ is the characteristic map for the CW-structure on $X$.  
Similarly, we can give $X_+ = (X \sqcup B, p\sqcup \id, \varnothing\sqcup \id)$ an ex-$\CWg$-structure by letting 
\[X^n_+ = X^n \sqcup B. \]
\end{ex}

\begin{ex}
    For $G$ a finite group,  a $\mathrm{CW}(0)$-complex on $B_+$ is equivalent to a $G$-CW structure on $B$, which has cells of the form $G \times_H D(\R^n)$ where $\R^n$ is the trivial $n$-dimensional $H$-representation. 
\end{ex}

\begin{ex}\label{ex:cellS11}
Consider $G=C_2$ and let $B=\mathbb{P}(\R^{2,1})\cong S^{1,1}$. Let $\taut$ be the tautological line bundle of $\mathbb{P}(\R^{2,1})$ from \cref{ex:tautS11}.  Recall the bundle gives a representation $\taut =(1,0,1) \in \RO(\Pi S^{1,1})$ via \cref{lem:ROpiS11}.  Fix a representation
\[\gamma=\taut-1 = (0,0,1) \in \Z^3 \cong \RO(\Pi S^{1,1})\] 
so $\gamma$ has virtual dimension zero.  Then $\gamma(b_0) = \R^{0,0}$ and $\gamma(b_1)= \R^{0,1}$.

We give $S^{1,1}$ 
the following $\CWg$-complex structure.  Begin with two $(\gamma+0)$-cells 
\begin{align*}
b&\colon C_2/e \to S^{1,1} & b_0 &\colon C_2/C_2 \to S^{1,1}
\end{align*}
as depicted in green and orange on the left in \cref{fig:Ssigmacell}.  Recall these were both points in the skeleton we chose for $\Pi S^{1,1}$ in \cref{ex:funS11}.

\begin{figure}[ht]
\begin{center}
\includegraphics[width=\textwidth]{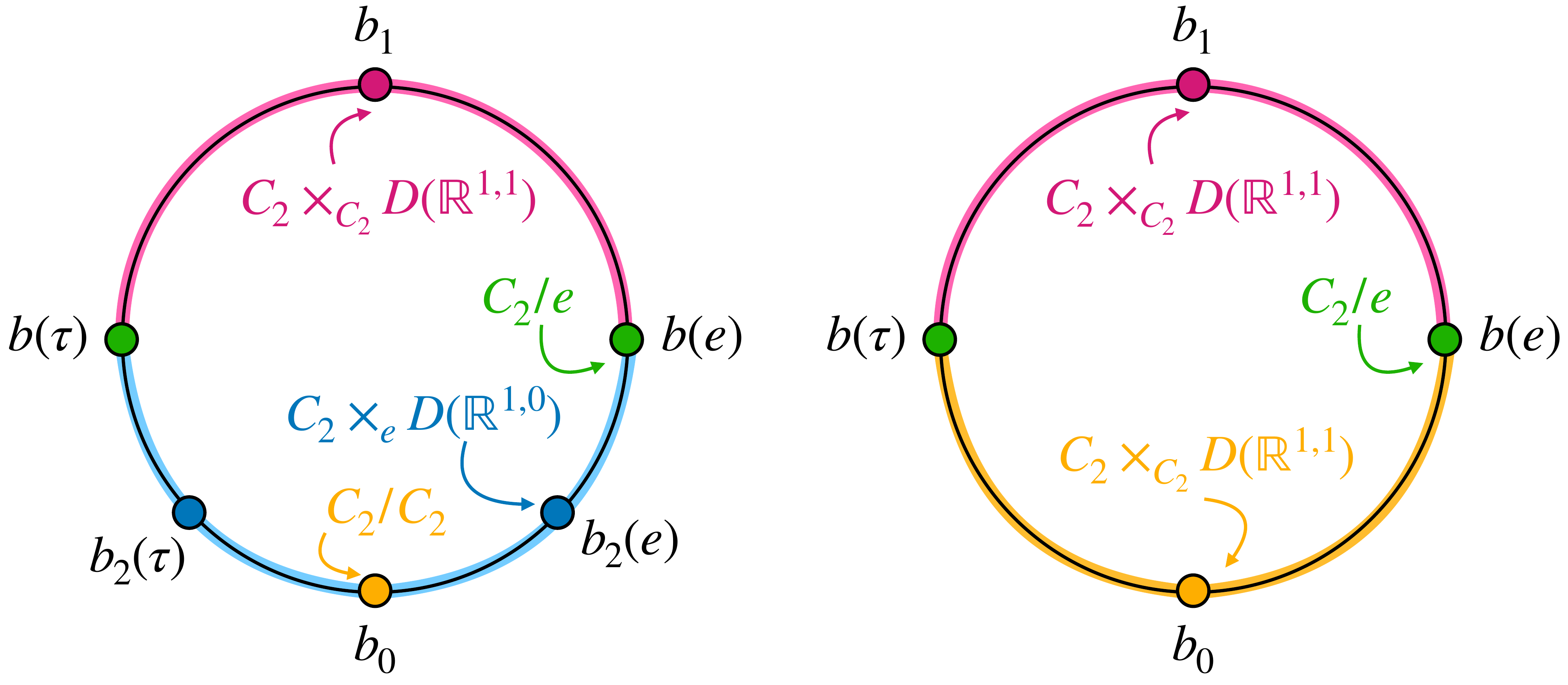}
\caption{Left: A $\CWg$ structure on $S^{1,1}$ for $\gamma= (0,0,1)$. Right: A $\CWg$ structure on $S^{1,1}$ for $\gamma= (0,1,1)$}
\label{fig:Ssigmacell}
\end{center}
\end{figure}

We will attach two $(\gamma+1)$-cells as well.  The first $(\gamma+1)$-cell is of the form $C_2\times_{C_2}D(\R^{1,1})\to S^{1,1} $, has center $b_1$, and has endpoints glued to $b$. 
For the second, it is helpful to name the point that will be the center of the cell, so we let $b_2\colon C_2/e \to S^{1,1}$ be the map such that $b_2(e) $ is the midpoint between $b(e)$ and $b_0$.
Then we attach a $(\gamma+1)$-cell $C_2 \times_e D(\R^{1}) \to S^{1,1}$ with center $b_2$, gluing the endpoints to $b_0$ and $b$ as shown on the left in \cref{fig:Ssigmacell}. These two $(\gamma+1)$-cells are shown in pink and blue, with their centers labeled. 

We can readily verify this describes a $\CWg$-complex.  The $(\gamma + 0)$-cells are objects of the skeleton to which $\gamma$ assigns trivial representations $\gamma(b) = C_2 \times \R^0$ and $\gamma(b_0) = \R^{0,0}$. The $(\gamma+1)$-cells have centers $b_1$ and $b_2$.  While $b_2$ is not in the skeleton, $b_2 \cong b$, so $\gamma$ assigns to the centers of the $(\gamma+1)$-cells $\gamma(b_1) = \R^{0,1}$ and $\gamma(b_2) = C_2 \times \R^0$. Now $\R^{0,1}+1 \cong \R^{1,1}$ and $\R^0 + 1 \cong \R^{1}$ as required.

Note that it would not be possible to give a $\CWg$-structure on $S^{1,1}$ with a $(\gamma+0)$-cell at $b_1$, because $\gamma(b_1)=\R^{0,1}$ is nontrivial.

One might be tempted to take a cell structure with only one $(\gamma+0)$-cell $b \colon C_2/e \to S^{1,1}$, and attach two $(\gamma+1)$-cells $C_2\times_{C_2}D(\R^{1,1}) \to S^{1,1} $ with centers $b_0$ and $b_1$. However, this is not a $\CWg$ structure for $\gamma= (0,0,1)$ since $\gamma_0(b_0)+1  \cong \R^{1,0}$ is not stably equivalent to $\R^{1,1}$. In fact, this describes a $\CWg$-structure for a different $\gamma$, namely $\gamma=(0,1,1) \in RO(\Pi S^{1,1})$. See \cref{fig:Ssigmacell} on the right.

Finally, it should be mentioned that there is a simpler $\CWg$-structure for our original choice of representation $\gamma=\taut-1=(0,0,1)$.  Take one $(\gamma+0)$-cell $b_0\colon C_2/C_2 \rightarrow S^{1,1}$ and attach one $(\gamma+1)$-cell $C_2\times_{C_2}D(\R^{1,1})\rightarrow S^{1,1}$ with center $b_1$.  We use the more complicated structure here, and again later in \cref{comp:S11taut}, to better demonstrate the details of the constructions. 
\end{ex}

 See \cite[Example 3.1.2]{CW_book} for more examples of $\CWg$-structures.

\begin{warn}\label{warn:CWfail}
It is not always possible to give a $G$-space $X$ over $B$ a $\CWg$-complex structure. 
One may be able to replace $X$ with a weakly equivalent space on which there is a $\CWg$-structure.\footnote{Chapter 3 of \cite{CW_book} contains $\CWg$-approximation results. See \cite[Theorems 3.2.12 and 3.2.13]{CW_book}. However, the replacements are only up to a notion called ``$\text{weak}_{\gamma}$-equivalence''. This class of morphisms is larger than that of weak equivalences. For example, $G/G$ does not admit a $\text{CW}(\gamma)$-structure for $\gamma = \R^{-1,-1}$ up to weak equivalence, only up to $\text{weak}_{\gamma}$-equivalence.}

For example, let $X=B=G/G$ and take $V$ a nontrivial actual $G$-representation. Let $\gamma=V-|V|$, so $\gamma$ has virtual dimension zero.\footnote{This is simply to follow our convention.  One could also take $\gamma = V$.} It is not possible to form a $(\gamma+0)$-cell using a representation that is both trivial and stably equivalent to $V-|V|$. So, in general, there is no $\CWg$-structure on $G/G$ as a space over itself! 

However, we can replace $G/G$ with the equivariantly contractible $D(V)$ and give a $\CWg$-structure to $D(V)$. We will see that such a cell structure will compute the Bredon cohomology of a point in degrees $V-|V|+*$, and so it is reasonable that giving a cell structure to $D(V)$ should be about as difficult as giving a $G$-CW structure to $S^V$. 

In particular, let $G=C_2$, $V=\R^{1,1}$, and consider $D(\R^{1,1})$.  As above, take $\gamma = V - |V|$.  We can give $D(\R^{1,1})$ one $(\gamma+0)$-cell, the inclusion $ C_2/e\cong S(\R^{1,1}) \to D(\R^{1,1})$. Now attach a single $(\gamma+1)$-cell of the form $C_2\times_{C_2} D(\R^{1,1})  =D(\R^{1,1}) \to D(\R^{1,1})$ attached along the identity map.
\end{warn}

\subsection{Cohomology}

In \cite[\S 3.3]{CW_book}, Costenoble--Waner define $RO(\Pi B)$-graded cohomology theories as an analogue of cellular cohomology. We review this here.

Fix a $\CWg$-structure on an ex-$G$-space $(X,p,s)$ over $B$. We assume that $X$ has finite type in the sense that it has finitely many $\gamma+n$ cells for each $n$. The finite type assumption is not necessary for the theory but will be the case in all our examples and it makes the exposition  easier.

In order to form a cellular chain complex, we need to understand the filtration quotients of such an ex-$\CWg$-complex. These filtration quotients $X^{\gamma+n}/_BX^{\gamma+n-1}$ are (laxly) homotopy equivalent to a wedge of parametrized spheres, as defined in \cref{defn:para-rep-sphere} (see \cite[Lemma 3.3.2]{CW_book} and also \cref{lem:onlycenters} and \cref{cor:onlycentersstable} below).  In each degree, the cellular chain complex given in \cite[Definition 3.3.3]{CW_book} uses stable maps of parametrized spheres to these filtration quotients.  We rephrase the definition here, making explicit the role of the isomorphism $\Gamma_\gamma \colon {\Bsorb{G}{B}} \xrightarrow{\cong} \widehat{\Pi}_\gamma B $ from \cref{thm:gammagamma}.  

\begin{defn}
For $X$ an ex-$\CWg$-complex, the \emph{cellular chains on $X$} are defined by the following chain complex of parametrized Mackey functors, 
\[\uC_{\gamma+*}(X) \colon \pars{\Bsorb{G}{B}}^\op\to \mathrm{Ab}.\]
\begin{enumerate}[(a)]
\item For $b$ an object of ${\Bsorb{G}{B}}$, 
\[\uC_{\gamma+n}(X)(b) :=[G_+\wedge_H S^{\gamma_0(b)+n,b}, X^{\gamma+n}/_BX^{\gamma+n-1}]_{B}^{G} . \] 
\item For $f \colon b \to c$ a morphism of ${\Bsorb{G}{B}}$,
\[\uC_{\gamma+n}(X)(f) = - \circ \Gamma_\gamma(f) \colon  \uC_{\gamma+n}(X)(c) \to \uC_{\gamma+n}(X)(b).\]
\item The boundary
\[d_{\gamma+n} \colon \uC_{\gamma+n}(X)  \to \uC_{\gamma+n-1}(X)  \]
is the natural transformation induced by post-composition with the map
\[\delta_{\gamma+n} \colon X^{\gamma+n}/_BX^{\gamma+n-1} \to \Sigma_B X^{\gamma+n-1}/_BX^{\gamma+n-2}. \]
\end{enumerate}
\end{defn}

The key observation for computation will be that, up to the ``twist'' by $\Gamma_\gamma$, each $\uC_{n+\gamma}(X)$ is a direct sum of representable functors. To see this, we first need a lemma describing the filtration quotients.

\begin{lem}[{\cite[Lemma 3.3.2]{CW_book}}]\label{lem:onlycenters}
There is a lax equivalence over $B$ given by
\[X^{\gamma+n}/_B X^{\gamma+n-1} \simeq_{\mathrm{lax}} \bigvee_{x} G_+ \wedge_H S^{V_x,px}  \]
where the wedge runs over the centers $x \colon G/H \to X$ of the $(\gamma+n)$-cells $G\times_H D(V_x)$ in the ex-$\CWg$-complex structure for $X$.
\end{lem}
We illustrate some of the details in the proof of \cref{lem:onlycenters} in order to help the reader visualize the maps.
See \cref{fig:laxequiv} for a figure showing the maps in the equivalence for a single cell.
\begin{figure}[ht]
\begin{center}
\includegraphics[width=\textwidth]{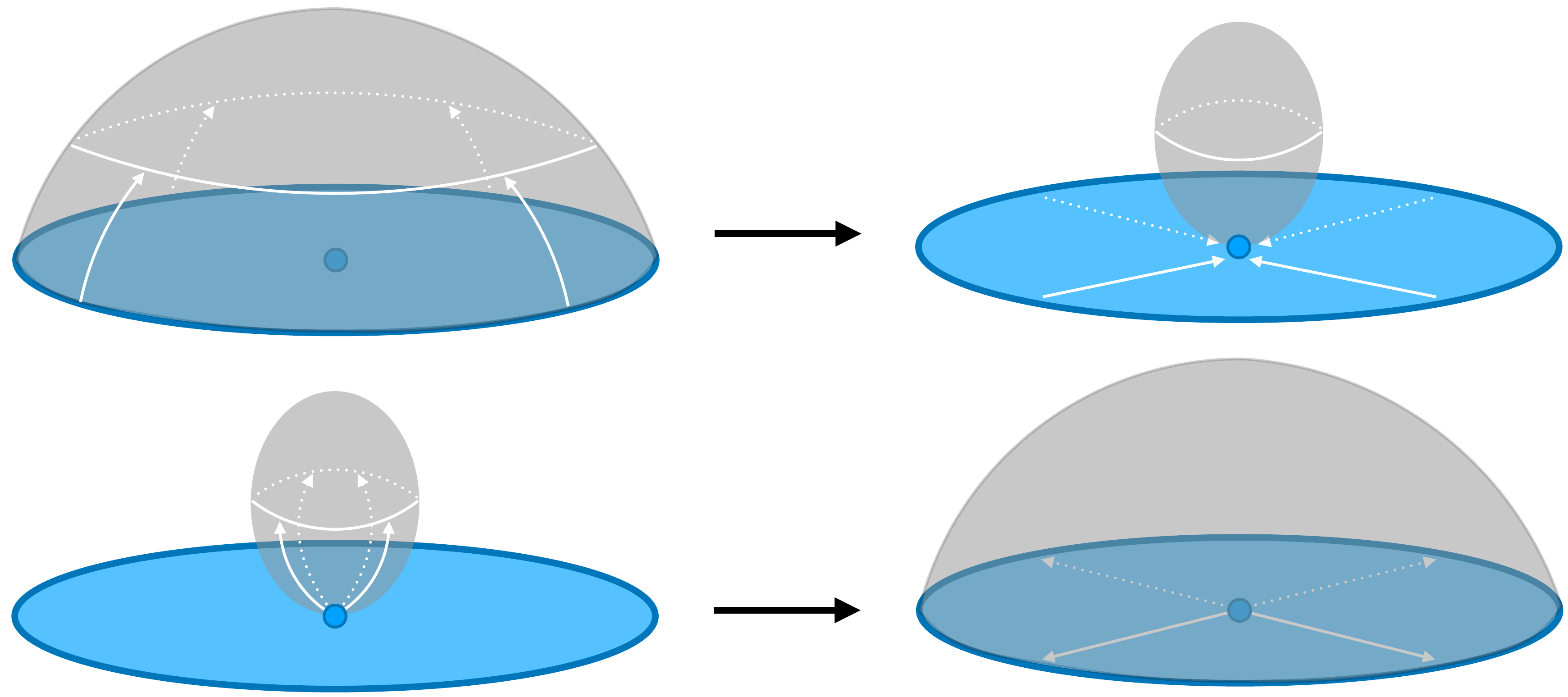}
\caption{The maps  \eqref{eq:h}  (top) and  \eqref{eq:k} (bottom).}
\label{fig:laxequiv}
\end{center}
\end{figure}

\begin{proof}
We explain the idea when there is a single $(\gamma+n)$-cell $G\times_H D(V)$ with center $x$. The generalization to multiple cells follows from applying the process below to all cells at once.

As in the proof in \cite{CW_book}, note that $X^{\gamma+n}/_B X^{\gamma+n-1}$ is given by a copy of $B$ (the image of the section) with the $(\gamma+n)$-cell $G\times_H D(V)$ glued along its boundary $G\times_{H} S(V)$ via the composite $p\partial\varphi$, where \[\partial\varphi \colon G\times_{H} S(V) \to X^{\gamma+n-1}\] is the attaching map for the cell.

We have an $H$-map 
\[
\xymatrix@C=3pc{D(V) \ar[r]^-{q} &  D(V) \vee S^V \ar[r]^-{ p\varphi\ \vee \ \id \ }  & S^{V,x}}
\]
where the first map is the quotient of $D(V)$ by the sphere of radius one half in $D(V)$, and the second map is given by $p\varphi$ on $D(V)$ and the identity on $S^V$.
The composite glues with the identity on $B$ to induce a map
\begin{equation}\label{eq:h}
 X^{\gamma+n}/_BX^{\gamma+n-1} \xrightarrow{h_{\gamma+n}}  G_+\wedge_H S^{V,x}. \end{equation}
The map $h_{\gamma+n}$ is a lax homotopy equivalence over $B$. (It is lax because it does not strictly commute with the projection to $B$.) The inverse is obtained as follows. Write
\[S^V = D(V)^+ \cup_{S(V)} D(V)^-\]
where $D(V)^+$ is the upper hemisphere (as shown in \cref{fig:laxequiv}) corresponding to the unit disk in $V$, and $D(V)^-$ is the lower hemisphere corresponding to vectors of length greater than one (glued to $B$ at infinity). 
Then consider the $H$-map
\[
S^V \longrightarrow D(V) \cup_{p\partial\varphi} B \subseteq i^*_H\pars{X^{\gamma+n}/_BX^{\gamma+n-1}}
\]
 which is the identity on $ D(V)^+$ and $p\varphi$ on $D(V)^-$. We can use the map above to get a map from the induced sphere $G \times_H S^V$ by adjunction, and then extend to the parametrized sphere (via the pushout with $\id$ on $B$).  This gives the inverse equivalence
\begin{align}\label{eq:k}  G_+\wedge_H S^{V,x} \xrightarrow{k_{\gamma+n}} X^{\gamma+n}/_BX^{\gamma+n-1},
\end{align}
proving the claim for one cell.

More generally, given multiple cells, we apply the process described above to all cells at once to construct lax maps
\[\xymatrix@C=3pc{ X^{\gamma+n}/_B X^{\gamma+n-1}  \ar[r]^-{h_{\gamma+n}} & \bigvee_{x} G_+ \wedge_H S^{V_x,x}  \ar[r]^-{k_{\gamma+n}} &  X^{\gamma+n}/_B X^{\gamma+n-1} }\]
giving the desired lax equivalence.
\end{proof}

\begin{cor}\label{cor:onlycentersstable}
There is a stable equivalence 
\[X^{\gamma+n}/_B X^{\gamma+n-1} \simeq \bigvee_{x} G_+ \wedge_H S^{\gamma_0(px)+n,px}  \]
as spectra over $B$, where the wedge runs over the centers $x \colon G/H \to X$ of the $(\gamma+n)$-cells in the ex-$\CWg$-complex structure for $X$.
\end{cor}
\begin{proof}
By \cref{lem:onlycenters},
\[ X^{\gamma+n}/_BX^{\gamma+n-1}  \simeq_{\lax}  \bigvee_x G_+\wedge_H S^{V_x,x} \]
as spaces over $B$ where $x \colon G/H \to B$ is the center of the cell $G\times_H V_x$.
The claim then follows from the fact that $V_x$ is stably equivalent to $\gamma_0(x) +n$, and the fact that lax homotopy equivalences give rise to stable equivalences via a zig-zag of equivalences, see Proposition 2.5.13 and Lemma 2.5.14 (b) of \cite{CW_book}.
\end{proof}

\begin{rem}\label{rem:laxequiv}
    Going forward, we will often abuse notation and simply write $\gamma_0(x)+n$ instead of $V_x$  since only the stable type of $V_x$ will matter for computing stable maps.  We will write
\[\xymatrix@C=3pc{ X^{\gamma+n}/_B X^{\gamma+n-1}  \ar[r]^-{h_{\gamma+n}} & \bigvee_{x} G_+ \wedge_H S^{\gamma_0(px)+n,x}  \ar[r]^-{k_{\gamma+n}} &  X^{\gamma+n}/_B X^{\gamma+n-1} }\] 
when referring to the maps described in the proof of \cref{lem:onlycenters}.
\end{rem}

\begin{rem}
It follow from \cref{cor:onlycentersstable} that the boundary map $\delta_{\gamma+n}$ on the filtration quotient that induces the differential in the cellular chain complex can be described by a map 
\[ \bigvee_{x,y} \delta_{\gamma+n}^{x,y},\] 
where $x$ runs over the centers of the $(\gamma+n)$-cells, $y$ over the centers of the $(\gamma+n-1)$-cells, and 
 $ \delta_{\gamma+n}^{x,y}$ is a stable map 
\[ \xymatrix{
G_+\wedge_H S^{\gamma_0(x)+n,x} \ar[r]^-{\delta_{\gamma+n}^{x,y}} &   G_+\wedge_H S^{\gamma_0(y)+n,y}  
}. \]
\end{rem}

We return to the cellular chains parametrized Mackey functor.
\begin{cor}
The functor $\uC_{\gamma+n}(X)$ is a direct sum 
\[ \uC_{\gamma+n}(X)(-) 
 \cong  \bigoplus_x {\Bsorb{G}{B}}\pars{-,\Gamma_\gamma^{-1}(px)}  \]
and under this isomorphism, 
\[d_{\gamma+n}= \bigoplus_{x,y} \Gamma_\gamma^{-1}\pars{\delta_{\gamma+n}^{x,y}} \circ - \] for an element 
\[ \pars{\delta_{\gamma+n}^{x,y}} \in \bigoplus_{x,y} \widehat{\Pi}_\gamma B(px,py) .\]
\end{cor}
\begin{proof}
We have
\begin{align*}
\uC_{\gamma+n}(X)(b) &=\Big[G_+ \wedge_K S^{\gamma_0(b)+n,b} ,\bigvee_x G_+ \wedge_H S^{\gamma_0(px)+n,px} \Big]_{B}^{G} \\
&\cong \bigoplus_{x}\widehat{\Pi}_\gamma B(b,px),
\end{align*}
where again the wedge and sum run over the centers of the cells. 
Now apply the Yoneda lemma, taking into account the isomorphism $\Gamma_\gamma$.
\end{proof}

As usual, we apply $\Hom$ to get cellular cochains.  This will give us a chain complex of abelian groups.\footnote{As described in \cite{CostenobleB}, it is possible to extend the definition of cohomology to be valued in $G$-Mackey functors using the induction and restriction adjunctions.}

\begin{defn}
Let $X$ be an ex-$\CWg$-complex and $\uM$ be a parametrized Mackey functor over $B$. The \emph{cellular cochains on $X$ with coefficients in $\uM$} is the chain complex given by 
\[C^{\gamma + *}(X;\uM) \cong \mathrm{Nat}( \uC_{\gamma + *}(X), \uM),\]
where the right-hand side denotes the abelian group of natural transformations from $ \uC_{\gamma + *}(X)$ to $\uM$. The coboundary $d^{\gamma+n}$ is obtained by precomposition with $d_{\gamma+n}$. The  \emph{parametrized cellular cohomology of $X$ with coefficients in $\uM$} is defined by 
\[\wH^{\gamma+*}(X;\uM) := H^*(C^{\gamma + *}(X;\uM), d^{\gamma+*} ).\]
When $X=B_+= B\sqcup s(B)$, we write
\[H^{\gamma+*}(B;\uM) := \wH^{\gamma+*}(X; \uM ).\]
\end{defn}

We get the following result from the Yoneda lemma.
\begin{cor}
\label{cor:yoneda} 
There is an isomorphism of cochain complexes:
\[\xymatrix{\cdots \ar[r] & C^{\gamma+n-1}(X ; \uM) \ar[rrr]^-{ d^{\gamma+n} } \ar[d]^-{\cong}&& &C^{\gamma+n}(X ;\uM) \ar[d]^\cong \ar[r] & \cdots \\
\cdots \ar[r] &\bigoplus_y \uM(py)\ar[rrr]^-{\bigoplus_{x,y}  \uM(\Gamma_\gamma^{-1}(\delta_{\gamma+n}^{x,y}))} &&  & \bigoplus_x \uM(px)\ar[r] & \cdots
} \]
\end{cor}
\begin{proof} 
We apply the Yoneda lemma to $\Nat(\widehat{\Pi}_\gamma B(-,px), \uM \Gamma_{\gamma}^{-1}) $ and  use that $\Gamma_\gamma$ is an isomorphism of categories (see \cref{thm:gammagamma}) to get 
\[ \uM(\Gamma_\gamma^{-1}(px)) \cong \Nat(\widehat{\Pi}_\gamma B(-,px), \uM \Gamma_{\gamma}^{-1}) \cong \Nat({\Bsorb{G}{B}}(\Gamma_\gamma(-),px),\uM). \]
Since $\Gamma_\gamma$ and $\Gamma_\gamma^{-1}$ are the identity on objects (see \cref{def:gammagamma}), \[\uM(\Gamma^{-1}(px)) = \uM(px)\] 
for any center of a cell $x$, and the claim follows.
\end{proof}

\begin{rem} 
The elements $\delta_{\gamma+n}^{x,y}$ play a crucial role in computations, and so we discuss a bit more here. Namely, 
for each pair $x,y$, we construct an explicit diagram of lax maps that describes $\delta_{\gamma+n}^{x,y}$. This diagram is useful when trying to understand these maps in computations. We construct
\[ 
\xymatrix{
G_+\wedge_HS^{\gamma_0(x)+n,x}  \ar[d]  \ar[r]^-{\delta_{\gamma+n}^{x,y}} &G_+\wedge_KS^{\gamma_0(y)+n,y}   \\
\bigvee_{x} G_+\wedge_HS^{\gamma_0(x)+n,x} \ar[d]^-{k_{\gamma+n}} \ar[r]^{\delta_{\gamma+n}}  &  \bigvee_y G_+\wedge_KS^{\gamma_0(y)+n,y}  \ar[u]\\
X^{\gamma+n}/_B X^{\gamma+n-1} \ar[r]^-{\delta_{\gamma+n}} & \Sigma_B X^{\gamma+n-1}/_B \Sigma_B X^{\gamma+n-2}  \ar@{..>}[u]^-{\Sigma_B h_{\gamma+n}}  
 }\]
 where the two top unlabelled vertical arrows are the inclusions and the fiberwise quotient by the summands corresponding to $y'\neq y$. We will describe a dotted arrow which makes the diagram commute up to lax homotopy. Although we will call this dotted arrow  ``$\Sigma_B h_{\gamma+n}$'' we do not mean this literally since the lax map $h_{\gamma+n}$ is not a map of spaces over $B$ (it does not preserve the fibers), and so its fiberwise suspension over $B$ is not well-defined.
 
 We again, do this in the case when there is a single cell (see \cref{fig:hsig}), this time with center $y$,  with attaching map $\varphi \colon G\times_K D(W) \to X$ so that
 \[\Sigma_B X^{\gamma+n-1}/_B \Sigma_B X^{\gamma+n-2}  \cong \Sigma_B (G \times_K D(W) \cup_{p\partial \varphi} B) \]
  where $W$ is stably equivalent to $\gamma_0(y)+n-1$. 
We have that
 \[\Sigma_B X^{\gamma+n-1} = G\times_K (D(W)\times [-1,1]) \cup_{G\times_K(  D(W)\times  \{\pm 1\})} B\]
 where the gluing uses the projection onto $D(W)$ followed by $p\varphi$. We describe a map
\[  D(W)\times [-1,1] \to S^{W,y}\] 
 inducing the dotted arrow after the appropriate induction and gluing. On $D(W) \times \{t\}$, the map is given by collapsing the disk of radius $(1-|t|)/2$ in $D(W)$ to obtain a wedge $D(W) \vee S^W$, applying $p\varphi$ to the  $D(W)$ factor and mapping $S^W$ using the identity. This can be done continuously in $t$ by being careful about the identification of the quotient with the wedge.
This then induces a lax map
\begin{equation}\label{eq:sigh}\Sigma_B h_{\gamma+n} \colon \Sigma_B X^{\gamma+n-1}/_B \Sigma_B X^{\gamma+n-2} \to G_+\wedge_K S^{W,y} .
\end{equation}
\end{rem}

\begin{figure}[ht]
\begin{center}
\includegraphics[width=\textwidth]{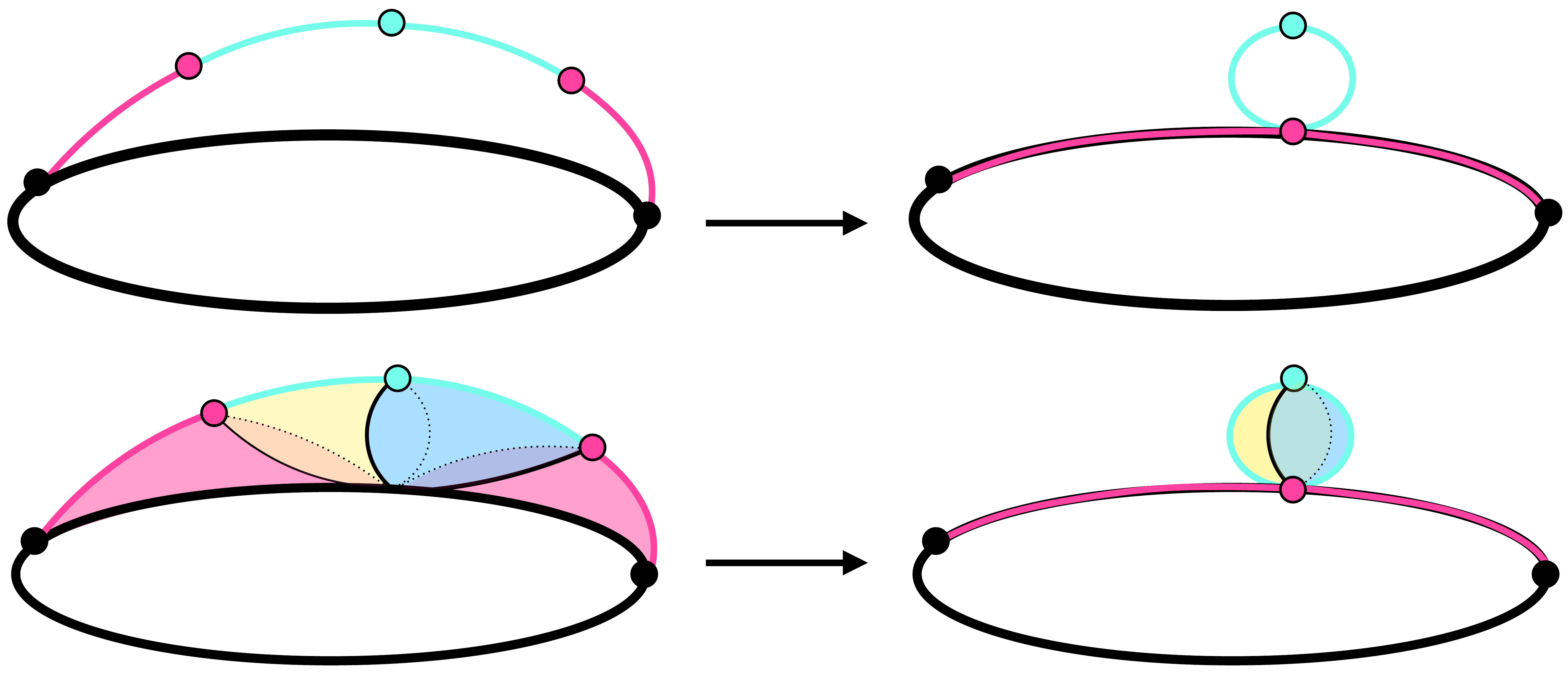}
\caption{The maps  \eqref{eq:h}   (top) and \eqref{eq:sigh} (bottom).}
\label{fig:hsig}
\end{center}
\end{figure}

\subsection{$RO(\Pi B)$ cohomology for the trivial group}\label{sec:trivialgroupcohomology}
Assume that $G=e$ is the trivial group and $B$ is a path-connected CW-complex of finite type. We examine the parametrized cellular cohomology for the ex-space $X=B_{+}$ whose total space is $B \sqcup s(B)$.  Here $s(B)$ is a copy of $B$ that is the image of the section $s=\id_B$ and the map $p$ is the identity on both factors.

 Let $\gamma \in \RO(\Pi B)$ have virtual dimension zero.  In this setting, the nonequivariant $\CWg$-structures do not depend on $\gamma$ (see \cref{ex:CWstructureG=e}). 
So in our CW-structure we may ignore $\gamma$, and we simply write $X^{\gamma+n} = X^n$ and $\delta_{\gamma+n} = \delta_n$.  Furthermore, we automatically get a  $\CWg$-structure on $X = B_+$ from a CW-structure on $B$ as in \cref{ex:CWstructureG=e}.  The $n$-skeleton is given by
\[X^{n} =B^n \sqcup B  .\]
The filtration quotients are
\[X^{n}/_B X^{n-1} = B^n \cup_{B^{n-1}} B \simeq \bigvee_x S^{n,x} \]
where $x$ runs over the centers of the $n$-cells in $B^n$. Since $px=x$, we have omitted $p$ from the notation in the wedge, and we will continue to do.

The connecting homomorphism is thus a map
\[\xymatrix{
X^{n}/_B X^{n-1}  \ar[d]^-\simeq \ar[r]^-{\delta_n} &  \Sigma_B X^{n-1}/_B X^{n-2} \ar[d]^-\simeq
\\
\bigvee_x S^{n,x} \ar[r]^-{\bigvee_{x,y} \delta_{n}^{x,y}} & \bigvee_y S^{n,y}
}\]
with 
\[\delta_n^{x,y} \in [S^{n,x}, S^{n,y}]_B =  \widehat{\Pi}_0 B(x,y) \xrightarrow[\cong]{\Gamma_0} {\Bsorb{e}{B}}(x,y).\]

Since the group $G=e$ is trivial, $\Pi B$ is a groupoid and all morphisms are invertible. This means that spans in $\Pi B$ are equivalent to morphisms in $\Pi B$. Therefore, ${\Bsorb{e}{B}}$ is the category with objects the points of $B$ and morphisms the free abelian group on homotopy classes of paths from $x$ to $y$:
\[{\Bsorb{e}{B}}(x,y) = \Z[\Pi B(x,y)].\]
The connecting homomorphisms are thus determined by elements  
\[ \delta_{n}^{x,y} \in \Z[\Pi B(x,y)]. \]

\begin{rem}\label{rem:rhoshriek}
To relate these maps from the nonequivariant parametrized setting to classical cellular chains, consider $\rho \colon B \to \pt$ inducing
\[\rho_!(X^{n}/_B X^{n-1}) = B^n/B^{n-1}. \]
Observe that $\rho_!(\delta_{n})$ is the map
\[B^n/B^{n-1} \to \Sigma B^{n-1}/B^{n-2}\]
 giving rise to the usual connecting homomorphism for the cellular chains. The map $\rho_! \colon [S^{n,x},S^{n,y}]_B \to [S^n, S^n]$ can be identified with the augmentation
\[[S^{n,x},S^{n,y}]_B \cong \Z[\Pi B(x,y)] \to \Z\cong [S^n, S^n] \]
which sends each generator in $\Pi B(x,y)$ to $1$, and so $\delta_n$ is not simply determined by the boundary in the classical cellular chains of $B$.
\end{rem}

We next describe how to compute the elements $\delta_n^{x,y}$. Consider the diagram
\[ \xymatrix{
\coprod_x S(\R^n) \ar[d] \ar[r]^-{\coprod_x \partial_x} & B^{n-1} \ar[d] \\
\coprod_x D(\R^n) \ar[r]^-{\coprod_x \varphi_x} & B^n 
} \]
constructing $B^{n}$ from $B^{n-1}$.
Choose an $n$-cell labeled by its center $x$ and an $(n-1)$-cell labeled $y$.  Consider the composite $f_{x,y}$ given by 
\begin{equation}
\label{eq:definition fxy}
\xymatrix{ S(\R^n) \ar[r]^-{\partial_x} & B^{n-1} \ar[r] &  B^{n-1}/B^{n-2} \ar[r] &  S^{n-1}_y}, 
\end{equation}
where the middle map is the quotient by the $(n-2)$-skeleton, and the last map is the quotient to the wedge summand corresponding to the $(n-1)$-cell with center $y$. We also denote by $y$ its image in $S^{n-1}_y$. 
Following the approach on page 192 of \cite{Bredon}, we can deform this map so that $y \in S^{n-1}_y$ is a regular value with preimage $f_{x,y}^{-1}(y) \subset S(\R^n) $. For $z\in f_{x,y}^{-1}(y)$, the  map 
\[D_zf_{x,y} \colon T_z S(\R^n) \to T_y S^{n-1}_{y}\]
is an isomorphism.  Using the canonical orientation coming from the fact that our cells are built from subspaces of $\R^n$, we can define the local degree
\[\ldeg{f}{z}:=\mathrm{sgn}\det(D_zf_{x,y}) \in \{\pm 1\}.\]
As usual, this degree is the same as that obtained from the composites 
\[\xymatrix{D(\R^n)/S(\R^n) \ar[r]^-{\varphi_x} & B^n/B^{n-1} \ar[r] & \Sigma B^{n-1}/B^{n-2} \ar[r] & \Sigma S^{n-1}_y}\]
in the computation of degrees for the cellular boundary of $B$, however the former description will be easier to work with.

\begin{notation}\label{not:omegaz} For any $z \in f_{x,y}^{-1}(y)$, with $f_{x,y}^{-1}$ as described above, let $ \mu_z \colon [0,1] \to D(\R^n)$ be the straight-line path starting at the center of the disk and ending at $z$. Then 
\[\omega_{x,y}^z :=\varphi_x \mu_z\]  is a path in $B$ from $x$ to $y$. 
\end{notation} 

\begin{lem}\label{thm:deltanxy}
The components of $\delta_n$ are given by
\[\delta_{n}^{x,y} = \sum_{z\in f_{x,y}^{-1}(y)} \ldeg{f}{z}\omega_{x,y}^z.  \]
\end{lem}
\begin{proof}
\cref{fig:hdelta2} gives a diagrammatic description of this proof when $B=\R P^2$ with the standard cell structure.
To compute 
\[\delta_n^{x,y} \in [S^{n,x}, S^{n,y}]_B \cong \Z[\Pi B(x,y)],\] 
we need to identify the
composites $F_{x,y}$ given by
\[\xymatrix@C=1.5pc{S^{n,x} \ar[r]^-{k_n} & D(\R^n)\cup_{\partial \varphi_x} B \ar[r]^-{\varphi_x}  & X^{n}/_B X^{n-1} \ar[r]^-{\delta_{n}} \ar[r] & \Sigma_B X^{n-1}/_B X^{n-2} \ar[r] & S^{n,y}   ,} \]
where $k_n$ is the map from \cref{rem:laxequiv}. 
If we apply $\rho_!$, we get 
\[\rho_!F_{x,y} \simeq \Sigma f_{x,y} \colon S^n_x \to S^n_y\]
for $f_{x,y}$ as above. We can deform $F_{x,y}$ relative to $B$ so that each $z\in F_{x,y}^{-1}(0)$ lies on the equator of $S^{n,x}$ and furthermore so that, in disjoint $\epsilon$-neighborhoods $U_z$ for each $z$, the map $F_{x,y}$ is differentiable with local degree $\ldeg{f}{z}$. Let $\nu_z$ be the radial path in the lower hemisphere of the $S^n$ in $S^{n,x}$ from the attachment point $x\in B$ to the boundary of the neighborhood $U_z$.  Then $\varphi_x k_n (\nu_z) = \omega_{x,y}^z$.

We get a factorization, up to homotopy relative to $B$, of $F_{x,y}$  as
\begin{equation}\label{eq:factorFxy}
S^{n,x} \to \bigvee_{z\in F_{x,y}^{-1}(0)} S^{n,x}_z \to S^{n,y}
\end{equation}
where 
\[S^{n,x}_z:= S^n \cup_{\infty= 1} [0,1]\cup_{0= x} B, \]
the whiskered space equivalent to $S^{n,x}$.
The first map in \eqref{eq:factorFxy} is the composite of a collapse map that pinches $S^{n}$ in $S^{n,x}$ into a wedge of spheres where each sphere contains a unique $U_z$, followed by the map that collapses $U_z$ to $S^n$ and the outside of $U_z$ continuously onto the interval $[0,1]$. The resulting map 
\[S^{n,y}_z \to S^{n,y} \]
is the map $\ldeg{f}{z}\omega_{x,y}^z$ and the claim follows.
\end{proof}

\begin{figure}[ht]
\begin{center}
\includegraphics[width=\textwidth]{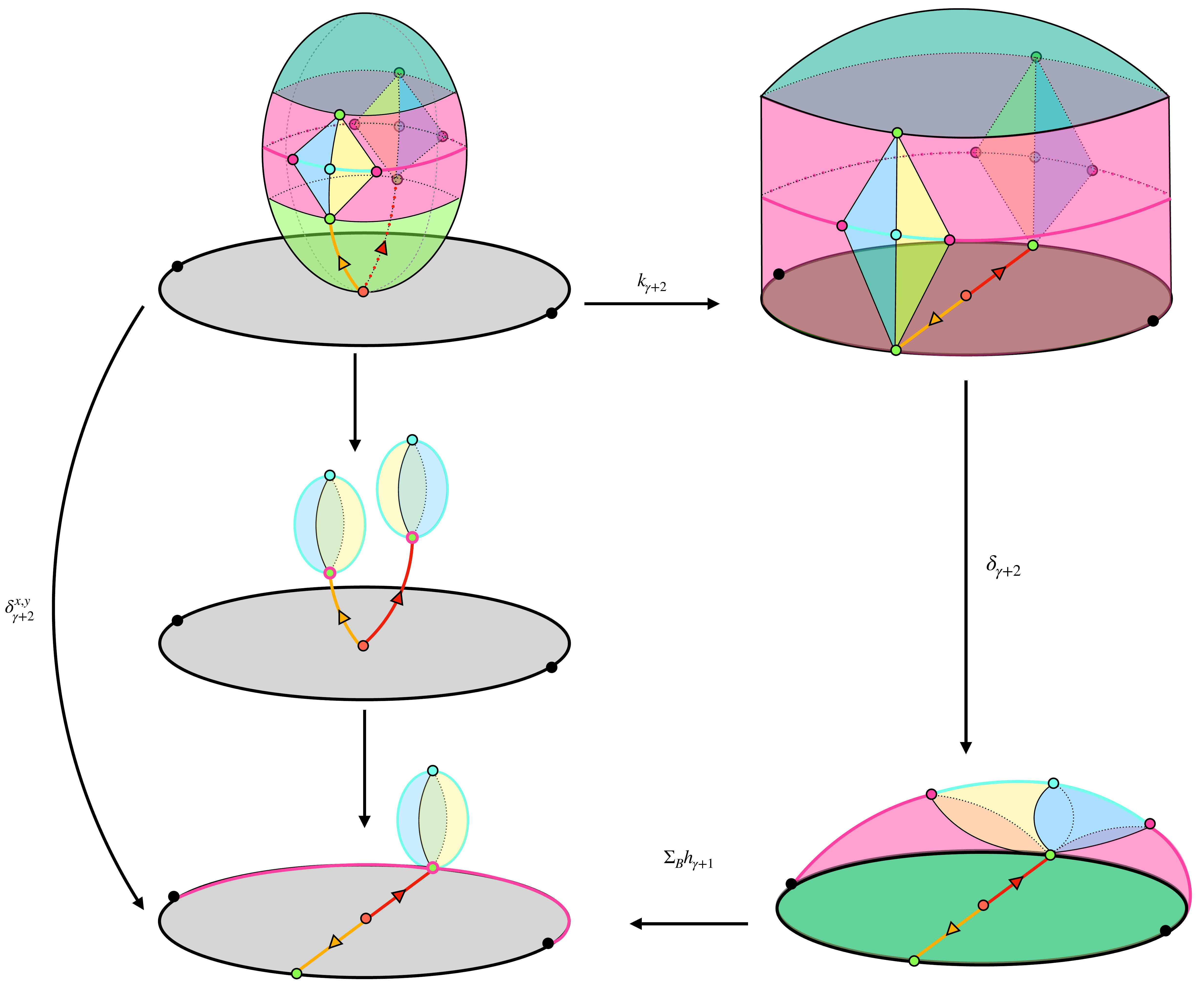}
\caption{The connecting homomorphism for $B=\mathbb{R}P^2$ with one cell in each dimension.}
\label{fig:hdelta2}
\end{center}
\end{figure}

Although $\gamma$ did not play a prominent role in the CW-structure on the nonequivariant $X=B_+$, it will play a role in the parametrized cellular chain complex. 
We are ready to discuss the twist coming from $\gamma$.

In order to be careful with signs, we make the following convention. 
\begin{convention}\label{ass:gamma} 
For $\gamma\in \RO(\Pi B)$ of dimension $|\gamma|$, we fix a representative $\gamma \colon \Pi B \to \vV(|\gamma|)$ such that $\gamma= \nu-n$ for an actual representation $\nu$. The existence of such a representative follows from \cref{example: coefficients when G = e}.
For all $x\in \Pi B$, we then have
\[\gamma(x)  =(\R^{|\nu|}, \R^n) =: \R^{|\gamma|}.\]
We thus get a function 
\[\mono{\gamma}  \colon \Pi B(x,y) \to \vV(\R^{|\gamma|},\R^{|\gamma|}) = O(1)\]
where the identification of $\vV(\R^{|\gamma|},\R^{|\gamma|})$ with $O(1)$ comes from \cref{ex:noneqvV} and $\mono{\gamma}(\omega)  \in O(1)$ is the homotopy class of bundle morphism  $\R^{|\gamma|} \to \R^{|\gamma|}$ associated with the path $\omega$.
\end{convention}
First, note the following useful observation, which follows from the equivalence of stable maps
\[\widehat{\Pi}_\gamma B(x,y) = [S^{|\gamma|,x}, S^{|\gamma|,y}]_B =  [S^{0,x}, S^{0,y}]_B = \widehat{\Pi}_{0} B(x,y)\]
for $\gamma\in \RO(\Pi B)$ as in \cref{ass:gamma}.

\begin{lem}\label{lem:actgamma}
Let $\gamma \in \RO(\Pi B)$. Then $\widehat{\Pi}_\gamma B =\widehat{\Pi}_{0} B$ and so
\[\Gamma_0^{-1}   \colon  \widehat{\Pi}_\gamma B \xrightarrow{\cong} {\Bsorb{e}{B}}  \]
is an isomorphism of categories. In particular,  each $\gamma$ gives rise to a ``twisting functor'' or automorphism
\[  \Gamma_0^{-1} \Gamma_\gamma \colon  {\Bsorb{e}{B}}  \xrightarrow{\cong} {\Bsorb{e}{B}}\]
which is the identity on objects, and on morphisms $ {\Bsorb{e}{B}}(x,y)=\Z[\Pi B(x,y)]$, the automorphism
\[  \Gamma_0^{-1} \Gamma_\gamma  \colon  \Z[\Pi B(x,y)] \to  \Z[\Pi B(x,y)]\]
given by
\[ \Gamma_0^{-1} \Gamma_\gamma (\omega) = \mono{\gamma}(\omega) \omega.\]
\end{lem}

We summarize this discussion and the consequence for cohomology in the following result.
\begin{thm}\label{thm:cohunderlying}
Let $B$ be a CW-complex, and $X=B_+$ with CW-structure induced by that of $B$. For any $\gamma$ in $\RO(\Pi B)$, the parametrized cellular chain complex is the chain complex of parametrized Mackey functors
\[\uC_{\gamma-|\gamma|+*}(X) \colon \pars{\Bsorb{e}{B}}^{\op} \to \Ab, \]
which is isomorphic to the chain complex given in degree $n$ by
\[\uC_{\gamma-|\gamma|+n}(X) (- ) \cong \bigoplus_x \Z[\Pi B( -,x)]\]
together with the differential given by composition with 
\[d_{\gamma-|\gamma|+n}= \bigoplus_{x,y} \sum_{z \in f_{x,y}^{-1}(y)} \ldeg{f_{x,y}}{z} \mono{\gamma}(\omega_{x,y}^z)\omega_{x,y}^z, \]
where $x$ runs over the centers of the $n$-cells of $B$ and $y$ over the $(n-1)$-cells.

In particular, for a parametrized Mackey functor $\uM \colon \pars{\Bsorb{e}{B}}^{\op} \to \Ab$, we have a commutative diagram of isomorphic cochain complexes
\[\xymatrix{
\cdots \ar[r] & C^{\gamma-|\gamma|+n-1}(X;\uM )   \ar[d]^-\cong \ar[rrr]^-{d^{\gamma-|\gamma|+n }} & & &  C^{\gamma-|\gamma|+n}(X; \uM )  \ar[d]^-\cong \ar[r] & \cdots  \\
\cdots \ar[r] & \bigoplus_y \uM(y) \ar[rrr]^-{ \bigoplus_{x,y} \sum_{z \in f_{x,y}^{-1}(y)} \ldeg{f}{z} \mono{\gamma}(\omega_{x,y}^z)\uM(\omega_{x,y}^z) }  & & &   \bigoplus_x \uM(x) 
\ar[r] & \cdots}  \]
whose cohomology is $\wH^{\gamma-|\gamma|+*}(X;\uM) =H^{\gamma-|\gamma|+*}(B;\uM) $.
\end{thm}

Next, we compare with the classical construction of cohomology with local coefficients using the universal cover $\widetilde{B}$ of $B$.   Let $b
\in B$ be the basepoint and let $\widetilde{B}$ be the universal cover with \[\widetilde{B} = \{[\lambda] \mid \lambda \text{ is a path in } B \text{ starting at } b\}\]
as usual. 
We fix the constant path $\widetilde{b} \in \widetilde{B}$ which maps to $b$ under the covering map
\[c\colon  \widetilde{B} \to B,\]
sending $[\lambda] \mapsto \lambda(1)$ to its endpoint.  A cellular structure on $B$ defines a cellular structure on $\widetilde{B}$ as follows. 

Consider the characteristic map $\varphi_{x} \colon D(\R^n) \to B$ for an $n$-cell of $B$ labelled by its center $x$. Given any element $\alpha_x \in \Pi B(b,x)$, there is a lift $\widetilde{\alpha}_x$ to $\widetilde{B}$ starting at $\widetilde{b} $ and whose end point $ \widetilde{x}_{\alpha_x}$ is uniquely determined by the homotopy class of $\alpha_x$ (relative its endpoints, i.e., up to path homotopy). 
Since $\widetilde{x}_{\alpha_x}$ maps to $x = \varphi_x(0)$, the lifting property for covering spaces gives a unique lift $\widetilde{\varphi}_{\alpha_x} \colon D(\R^n) \to \widetilde{B}$. We choose the CW-structure on $\widetilde{B}$, which has the $\widetilde{\varphi}_{\alpha_x}$ as characteristic maps. 

Thus we can describe the cellular chain complex $C_*(\widetilde{B})$, up to isomorphism, as follows. In degree $n$, it is the free abelian group on the set of homotopy classes of paths (relative endpoints) $\alpha_x$ from $b$ to $x$, where $x$ runs over all centers of the $n$-cells of $B$. In other words,
\[C_n(\widetilde{B})\cong \bigoplus_x\Z[\Pi B(b,x)].\]
Moreover, $\Pi B(b,b)$ acts on $C_n(\widetilde{B})$ by precomposition, making $C_n(\widetilde{B})$ a right  $\Pi B(b,b)$-module (where we are writing path composition from right to left as in the composition of functions).

For $\gamma=0$ in $\RO(\Pi B)$, let $\uC_{*}(X)=\uC_{\gamma+*}(X)$. Note that for $b\in B$,  $\uC_{*}(X)(b)$ is also a right $\Pi B(b,b)$-module. 
\begin{theorem}
    Let $B$ be a CW-complex, and $X=B_+$ with CW-structure induced by that of $B$.  Let $\widetilde{B}$ be the universal cover of $B$ with CW-structure as above.  There is an isomorphism
    \[C_*(\widetilde{B}) \cong \uC_{*}(X)(b)\]
    of chain complexes of right $\Pi B(b,b)$-modules. 
\end{theorem}
\begin{proof} 
By Theorem \ref{thm:cohunderlying}, $\uC_n(X)(b)$ is also isomorphic to $\bigoplus_x\Z[\Pi B(b,x)]$ as a right $\Pi B(b,b)$-module and
we get an isomorphism of $\Pi B(b,b)$-modules 
\[ C_n(\widetilde{B}) \cong \bigoplus_x \Z[\Pi B(b,x)] \cong \uC_n(X)(b)  \]
in every degree $n$, where $x$ runs over the centers of the $n$-cells in $B$. 
To show that the chain complexes $C_*(\widetilde{B})$ and $\uC_n(X)(b)$ are isomorphic, it remains to prove that the boundary $d_n$ of $\uC_*(X)(b)$ agrees with the boundary $\widetilde{d}_n$ of $C_*(\widetilde{B})$ under the isomorphisms above.

Fix $x$, the center of an $n$-cell in $B$.  By \cref{thm:deltanxy}, the differential in $\uC_n(X)(b)$ on $\alpha_x \in \Pi B(b,x)$ is given by
\[d_n(\alpha_x) = \sum_{z\in f_{x,y}^{-1}(y)} \ldeg{f}{z}  (\omega_{x,y}^z \star \alpha_x),\]
where $y$ runs over the centers of the $(n-1)$-cells in $B$, and $\omega_{x,y}^z \star \alpha_x$ is the path $\alpha_x\colon b\to x$ followed by the path $\omega_{x,y}^z \colon x \to y$ defined in \cref{not:omegaz}.

We show that the same formula holds for $\widetilde{d}_n$.  For an $n$-cell with center $x$ in $B$ and $\alpha_x\in \Pi B(b,x)$, let $\widetilde{x}_{\alpha_x}$ be the endpoint of the lift of $\alpha_x$ starting at $\widetilde{b}$. Note that $\widetilde{x}_{\alpha_x}$ is the center of the $n$-cell of $\widetilde{B}$ corresponding to $\alpha_x$. Similarly, for $y$ the center of an $(n-1)$-cell in $B$ and $\beta_y\in \Pi B(b,y)$, let $\widetilde{y}_{\beta_y}$ be the endpoint of the lift of $\beta_{y}$ starting at $\widetilde{b}$. Again, $\widetilde{y}_{\beta_y}$ is the center of the $(n-1)$-cell labelled by $\beta_y$.

 Recall that the characteristic map $\widetilde{\varphi}_{\alpha_x}$ is the unique lift of the characteristic map $\varphi_x$ sending the center of $D(\R^n)$ to $\widetilde{x}_{\alpha_x}$.
 We orient our cells in $B$ and $\widetilde{B}$ using the orientation induced from the standard orientation of $D(\R^n)$ via the characteristic maps $\varphi_x$ and $\widetilde{\varphi}_{\alpha_x}$. Consequently, the covering map $c \colon \widetilde{B} \to B$ has degree 1 when restricted to any cell.

We let $\widetilde{\partial}_{\alpha_x}=\widetilde{\varphi}_{\alpha_x}\vert_{\partial D(\R^n)}$. Define $\widetilde{f}_{\alpha_x,\beta_y}$ by the composition
\[\widetilde{f}_{\alpha_x,\beta_y}\colon \xymatrix{S(\R^n)\ar[r]^-{\widetilde{\partial}_{\alpha_x}} & \widetilde{B}^{n-1} \ar[r] &  \widetilde{B}^{n-1}/\widetilde{B}^{n-2}  \ar[r] &  S_{\beta_y}^{n-1}}\]
where $S^{n-1}_{\beta_y}$ is the component of the wedge $\widetilde{B}^{n-1}/\widetilde{B}^{n-2}$ corresponding  to the $(n-1)$-cell with characteristic map $\widetilde{\varphi}_{\beta_y}$.
We have the following commutative diagram: 
\[\xymatrix{
& & & S^{n-1}_{\widetilde{y}_{\beta_y}} \ar[d]
\\
S(\R^n)\ar[r]_-{\widetilde{\partial}_{\alpha_x}} \ar[d]^-= \ar@/^2pc/[rrru]^-{\widetilde{f}_{\alpha_x,\beta_y}}  & \widetilde{B}^{n-1}\ar[d]^-c \ar[r]& \widetilde{B}^{n-1} /\widetilde{B}^{n-2} \ar[r] \ar[d]^-c & \bigvee_{\beta_y \in \Pi B (b,y)} S^{n-1}_{\widetilde{y}_{\beta_y}}\ar[d]^-c
\\ 
S(\R^n) \ar@/_2pc/[rrr]_-{f_{x,y}} \ar[r]^-{\partial_x}  & B^{n-1} \ar[r] & B^{n-1}/B^{n-2}  \ar[r] & S^{n-1}_y \\
\\
}\]
Here, the vertical arrows are induced by the covering map $c$.
As above, we assume that $y \in S^{n-1}_y$ is a regular value of $f_{x,y}$. Since $c$ is a covering map, the induced map
\[ S^{n-1}_{\widetilde{y}_{\beta_y}} \to S^{n-1}_y \]
is a homeomorphism (of degree one).
Together with the commutativity of the diagram, this implies that $\widetilde{y}_{\beta_y}$ is a regular value of  $\widetilde{f}_{\alpha_x,\beta_y}$. Moreover, 
\[\coprod_{\beta_y \in \Pi B(b,y)} \widetilde{f}_{\alpha_x,\beta_y}^{-1}(\widetilde{y}_{\beta_y}) = f_{x,y}^{-1}(y),\]
and 
\[\deg(D_z\widetilde{f}_{\alpha_x,\beta_y}) = \deg(D_z{f}_{x,y}) =\ldeg{f}{z} \]
for $z$ a preimage of $\widetilde{y}_{\beta_y}$ under $\widetilde{f}_{\alpha_x,\beta_y}$. Note that $z$ is thus also a preimage of $y$ under $f_{x,y}$.

For such a $z\in \widetilde{f}_{\alpha_x,\beta_y}^{-1}(\widetilde{y}_{\beta_y})$, let $\mu_{z}$ be the affine linear path from the center of $D(\R^n)$ to $z$ and let $\widetilde{\omega}^{z}_{\alpha_x,\beta_y}=\widetilde{\varphi}_{\alpha_x}\circ \mu_{z}$.
Then 
\[c\circ \widetilde{\omega}^{z}_{\alpha_x,\beta_y} = c \circ \widetilde{\varphi}_{\alpha_x}\circ \mu_{z} = \varphi_x \circ \mu_z= \omega_{x,y}^z  \]
where $\omega_{x,y}^z$ is as above.
Therefore,
\[(c\circ \widetilde{\omega}^{z}_{\alpha_x,\beta_y})\star\alpha_x = \omega_{x,y}^z \star \alpha_x.\]
Note that $ \omega_{x,y}^z \star \alpha_x$ is a path in $B$ from $b$ to $y$ and the endpoint of its lift to $\widetilde{B}$ with starting point $\widetilde{b}$ is $\widetilde{y}_{\beta_y}$. Since $\widetilde{B}$ is simply connected, there is a unique path up to homotopy from $\widetilde{b}$ to $\widetilde{y}_{\beta_y}$, and we get a path homotopy 
\[\omega_{x,y}^{z} \star \alpha_x \simeq\beta_y.\]

Now we can give an explicit formula for \[\widetilde{d}_n \colon \bigoplus_x \Z[\Pi B(b,x)] \to \bigoplus_y \Z[\Pi B(b,y)].  \] 
On the cell labelled by the path $\alpha_x \in \Pi B(b,x)$, 
we have that
\begin{align*}
\widetilde{d}_n(\alpha_x)
&=\sum_{\beta_y \in \Pi B(b,y)} \deg (D_z\widetilde{f}_{\alpha_x,\beta_y})\beta_y \\
&= \sum_{\beta_y \in \Pi B(b,y)} 
 \sum_{z \in \widetilde{f}_{\alpha_x,\beta_y}^{-1}(\widetilde{y}_{\beta_y})} \deg( D_{z}\widetilde{f}_{\alpha_x,\beta_y} ) \beta_y \\
 &= \sum_{\beta_y \in \Pi B(b,y)} 
 \sum_{z \in \widetilde{f}_{\alpha_x,\beta_y}^{-1}(\widetilde{y}_{\beta_y})} \ldeg{f}{z}  (\omega_{x,y}^z \star \alpha_x) \\
  &= \sum_{z\in f_{x,y}^{-1}(y)} \ldeg{f}{z}  (\omega_{x,y}^z \star \alpha_x).
 \end{align*}
This is the same as the boundary formula for $\uC_n(X)(b)$.
\end{proof}

\begin{cor}
Let $\gamma \in RO(\Pi B)$, $X=B_+$ and $b\in B$.
 There is an isomorphism 
       \[C_*(\widetilde{B})_\gamma \cong \uC_{\gamma-|\gamma|+*}(X)(b),\]
       where $C_*(\widetilde{B})_\gamma $ is the right $\Pi B(b,b)$-module $C_*(\widetilde{B})$ with action twisted by $\gamma$. That is, the action of $\omega \in \Pi B(b,b)$ on $C_*(\widetilde{B})_\gamma$ is given by $c\mapsto \mono{\gamma}(\omega) c\omega$ for $c\in C_*(\widetilde{B})$ and $c\mapsto c\omega$ the canonical right action of $\Pi B(b,b)$ on $C_{*}(\widetilde{B})$ given by deck transformations.
\end{cor}
\begin{proof}
This follows from $ \Gamma_0^{-1} \Gamma_\gamma (\omega) = \mono{\gamma}(\omega)\omega$ for $\omega \in  \Pi B(b,b)$.
\end{proof}

An immediate consequence is the following theorem.
\begin{theorem}\label{thm:localispofinal}
Let $\uN \colon \pars{\Bsorb{e}{B}}^\op \to \Ab$ be the parametrized constant Mackey functor at $N$ and let $X = B_+$. 
There is an isomorphism 
\[ H^{\gamma+*}(B;\uN)=\wH^{\gamma+*}(X ;\uN) \cong H^{|\gamma|+*}(B ; N_\gamma) \] 
where the right hand side is cellular cohomology with local coefficients, and 
\[N_\gamma \colon \Pi B \to \Ab\] 
is the local coefficient system with value $N$ on objects and with $N_\gamma(\omega) = \mono{\gamma}(\omega)$ on morphisms.
That is, parametrized nonequivariant cohomology corresponds to cohomology with local coefficients.
\end{theorem}

%% file: witpaper-examples.tex

\section{Computations of parametrized cellular cohomology}\label{sec:ex}
In this section, we compute some examples of equivariant parametrized cellular cohomology in the case when $G=C_2$ and $G=C_4$. We use constant $\uZ$-coefficients and $\underline{\F}_2$-coefficients. We do not give complete $\RO(\Pi B)$-graded computations, but we compute some degrees beyond $\RO(G)$.  Given a $\CWg$-complex, the most difficult part is to determine the coboundaries in the cochain complex from the attaching maps.

In the computations that follow, so much depends on the center of cells that we abuse notation and use $x$ both to denote the center of a cell in $B$ and to denote the characteristic map of the cell.  That is, for a $\pars{\gamma+n}$-cell, we write $x \colon G \times_H D(V) \to B$ to mean there is a characteristic map with $x \colon G/H \to B$  the center of the cell. To identify the coboundary in constant $\uZ$-coefficients: if $x$ is the center of a $\pars{\gamma+n}$-cell and $y$ the center of a $\pars{\gamma+n-1}$-cell, we express the component $\delta_{\gamma+n}^{x,y}$ of the connecting homomorphism $\delta_{\gamma+n}$ as $\Gamma_\gamma(f^{x,y})$ for some $f^{x,y}\in \Bsorb{G}{B}(x,y)$.

Once we have identified $\delta_{\gamma+n}^{x,y} = \Gamma_\gamma(f^{x,y})$, the corresponding component of the coboundary $d^{\gamma+n}$ of the cellular cochain complex is given by
\[\uZ(\Gamma_\gamma^{-1}(\delta_{\gamma+n}^{x,y}))=\uZ(f^{x,y}).\]
Recall from \cref{def:constantMackey} that $\uZ = \rho^*\uZ$ where $\uZ$ is the constant (nonequivariant) Mackey functor and $\rho\colon B \to G/G$.
So we can write the component of the coboundary as
\[\uZ(f^{x,y}) = \rho^*\uZ(f^{x,y})  = \uZ(\rho_!f^{x,y}),\]
where the latter is described in the discussion before \cref{def:constantMackey}.  This gives the identity on restrictions and multiplication by the index for transfers.
 For a point $b \colon G/H \to B$, let  $b^*\in \uZ(b)$ denote the canonical generator of 
\[\uZ(b)=\hom^{G}(\rho_!b,\Z) = \hom^{G}(G/H,\Z) \cong \Z \]
determined by $b^*(eH)=1$.

In order to illustrate the method of computation, we begin with familiar examples and do a few computations in $RO(G)$-degrees. As expected, these agree with Bredon cohomology, which can be verified using the reader's favorite $\RO(G)$-graded computational methods.

\subsection{A $C_2$-equivariant disk}
A first goal might be to recover the $RO(C_2)$-graded cohomology of a point (as described in \cite{Hazel_cohZ}, for example, and originally due to unpublished work of Stong).  One quickly runs into the issue that the point $C_2/C_2$ does not admit a $\CWg$-structure for most choices of $\gamma$. See \cref{warn:CWfail}. 

In particular, $C_2/C_2$ does not admit a $\CWg$-structure for $\gamma=\R^{1,1}-1$, but the unit disk $D(\R^{1,1})$ does. So we start with the disk.

\subsubsection{$G=C_2$, $B=D(\R^{1,1})$, and $\gamma =\R^{1,1}-1$}
We give $D(\R^{1,1})$ the $\CWg$-structure with a single $\pars{\gamma+0}$-cell
\[
b\colon C_2/e\rightarrow D(\R^{1,1})
\]
and a single $\pars{\gamma+1}$-cell
\[
b_1\colon C_2 \times_{C_2} D(\R^{1,1})\rightarrow D(\R^{1,1})
\]
as illustrated in \cref{fig:D(sigma)}.

To compute the coboundary, we must understand the attaching map.  Here the attaching map is from a cell with a fixed center $C_2/C_2$ to a cell with center a free orbit $C_2/e$, so requires a transfer map.\footnote{Notice the $\CWg$-structures require equivariant parametrized cellular cohomology to use both restrictions and transfer maps, whereas $\RO(G)$-graded cohomology of a $G$-CW complex can be computed with only restrictions.}  One observes that the connecting homomorphism is the composition 
\[\delta=\delta^{b_1,b}=\res_\gamma(\id,\omega)\circ\tr_\gamma(\rho,c)\] with $c$ constant, $\rho\colon C_2/e\rightarrow C_2/C_2$ the quotient, and $\omega \colon C_2 \times I \to D(\R^{1,1})$ the outward path depicted in \cref{fig:D(sigma)}.
This is seen by inspection, but one difficulty with computations is obtaining the correct signs, so we provide a bit more detail.
\begin{figure}[t]
    \centering
\includegraphics[width=0.5\textwidth]{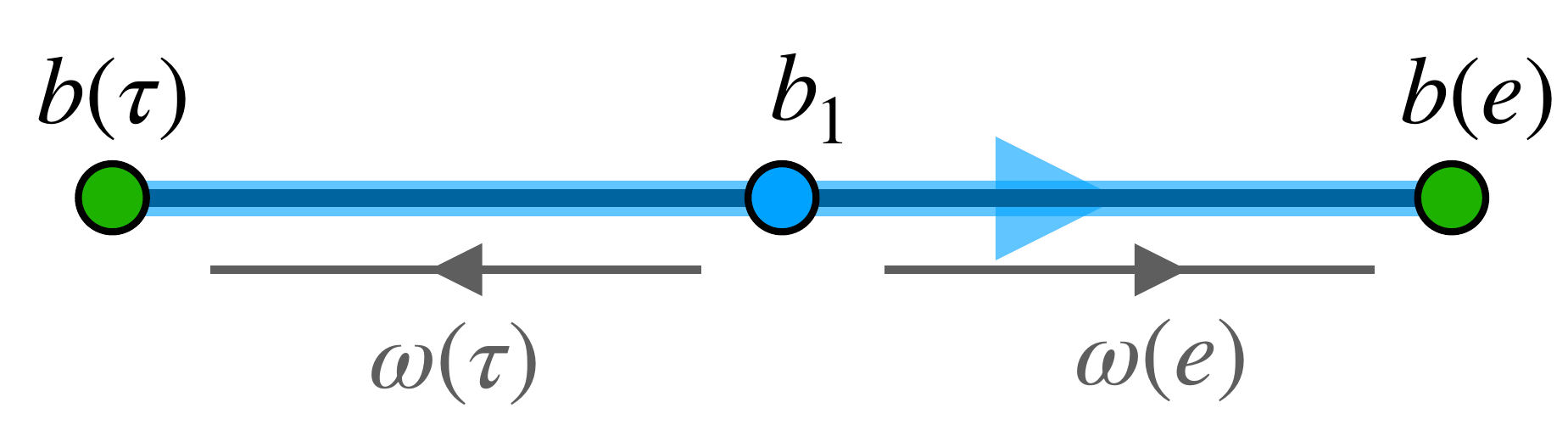}
    \caption{$\CWg$-structure for $D(\R^{1,1})$ with $\gamma= \R^{1,1}-1$}
    \label{fig:D(sigma)}
\end{figure}

We write $\omega(e) = \omega(e,t) \colon I \to i^*_eD(\R^{1,1})$ and similarly for the other paths.  Then orienting the $\pars{\gamma+1}$-cell as indicated by the blue arrow in \cref{fig:D(sigma)}, we find that 
\[i^*_e\delta_1^{b_1,b} = \begin{pmatrix}\omega(e) \\ -\omega(\tau)\end{pmatrix},
\]
while
\[i^*_e\tr_\gamma(\rho,c )= \begin{pmatrix}c(e) \\ -c(\tau) \end{pmatrix}
\quad \quad \text{and} \quad \quad  
i^*_e\res_\gamma(\id,\omega)= \begin{pmatrix}\omega(e) & 0 \\ 0 & \omega(\tau)\end{pmatrix}.
\]
The sign in the transfer comes from the shear map in its definition, see \eqref{eq:trnew}. For the signs on the restriction, we use the definition of the constant representation $\R^{1,1}$ that assigns to as in \cref{ex:constant} to compute
\[\xymatrix@C=8pc{
\gamma(b_1\rho)\ar@{=}[d] \ar[r]^-{\gamma(\id,\omega)} & \gamma(b)\ar@{=}[d] \\
C_2\times_e \R^1 \ar@{-->}[r]^-{\begin{pmatrix} +1 & 0 \\
0 & +1
\end{pmatrix}} & C_2\times_e \R^1
\\
 C_2/e \times \R^{1,1} \ar[r]_-{(\id,\id)=\begin{pmatrix} +1 & 0 \\
0 & +1
\end{pmatrix}}\ar[u]_-{\cong}^-{\mathrm{shear} =\begin{pmatrix} +1 & 0 \\
0 & -1
\end{pmatrix}} & C_2/e\times \R^{1,1}\ar[u]^-{\cong}_-{\mathrm{shear} =\begin{pmatrix} +1 & 0 \\
0 & -1
\end{pmatrix}}
}\]
where we abused notation with our matrix notation as described in \cref{rem:matrixrem}. 

Notice that the path did not matter for this computation.  As an aside, if we needed to compute $\gamma(\tau,\omega)$ we would have computed the following:
\[\xymatrix@C=5pc{
\gamma(b_1\rho)\ar@{=}[d] \ar[r]^-{\gamma(\tau,\omega)} & \gamma(b)\ar@{=}[d] \\
C_2\times_e \R^1 \ar@{-->}[r]^-{\begin{pmatrix} 0 & -1 \\
-1 & 0
\end{pmatrix}} & C_2\times_e \R^1
\\
 C_2/e \times \R^{1,1} \ar[r]_-{\begin{pmatrix} 0 & +1 \\
+1 & 0
\end{pmatrix}}\ar[u]_-{\cong}^-{\mathrm{shear} =\begin{pmatrix} +1 & 0 \\
0 & -1
\end{pmatrix}} & C_2/e\times \R^{1,1}\ar[u]^-{\cong}_-{\mathrm{shear} =\begin{pmatrix} +1 & 0 \\
0 & -1
\end{pmatrix}.}
}\]
Here and in all the examples below, we have  ignored the shift by $-1$ in $\gamma$. The shift was used to place it in degree zero and is not important for this part of the computation. So, in reality, this commutative diagram is for $\gamma_{\R^{1,1}}$ rather than $\gamma =\gamma_{\R^{1,1}}-1 $.

Now we can compute $\uZ(\Gamma_\gamma^{-1}(\delta))=\uZ(\res(\id,\omega)\circ\tr(\rho,c)) = 2$ and we get the cochain complex
\[0\rightarrow \Z\{b^*\}\xrightarrow{\, 2\, } \Z\{b_1^*\}\rightarrow 0.\]
Finally, we compute cohomology
\[H^{\gamma+n}(D(\R^{1,1});\uZ)=\begin{cases}\Z/2 & n=1\\0 & \text{otherwise.}\end{cases}\]
We note that this agrees with the Bredon cohomology $H^{n,1}(\pt ;\uZ)$, as expected.

Similarly, in $\underline{\F}_2$-coefficients  $\underline{\F}_2(\Gamma_\gamma^{-1}(\delta))=\underline{\F}_2(\res(\id,\omega)\circ\tr(\rho,c))=0$. 
So the cellular cochain complex is 
\[0\rightarrow \F_2\{b^*\}\xrightarrow{\, 0\, } \F_2\{b_1^*\}\rightarrow 0,\]
with cohomology
\[H^{\gamma+n}(D(\R^{1,1});\underline{\F}_2)=\begin{cases}\F_2 & n=0, 1\\0 & \text{otherwise.}\end{cases}\]

\bigskip

Taking $B = D(\R^{q,q})$ and $\gamma = \R^{q,q} - q$, one could use similar techniques to compute Bredon cohomology $H^{n,q}(\pt ;\uZ)$ or $H^{n,q}(\pt ; \underline{\F}_2)$.

\subsection{A $C_2$-equivariant circle}
We compute parametrized cohomology for the $C_2$-space $S^{1,1}$, the one point compactification of $\R^{1,1}$.  We first compute cohomology in a few $RO(C_2)$-degrees and then give an example in the extended grading.
\subsubsection{$G=C_2$, $B=S^{1,1}$, and $\gamma =0$}
There is a $\CWg$-structure on $S^{1,1}$ consisting of two $\pars{\gamma + 0}$-cells
\begin{align*}
    b_0\colon C_2/C_2\rightarrow S^{1,1}\\
    b_1\colon C_2/C_2\rightarrow S^{1,1}
\end{align*}
and one $\pars{\gamma+1}$-cell
\[
b\colon C_2 \times_e D(\R^1)\rightarrow S^{1,1}
\]
as depicted in \cref{fig:gamma=0}.
The components of the connecting homomorphism are given by
$\delta^{b,b_1}=\res_\gamma(\rho,\omega_1)$ and $\delta^{b,b_0}=-\res_\gamma(\rho,\omega_0)$, with $\omega_1$ and $\omega_0$ as shown in \cref{fig:gamma=0}. Applying $\uZ(\Gamma_\gamma^{-1}(-))$, we get the following cochain complex
\[0\rightarrow \Z\{b_1^*,b_0^*\}\xrightarrow{\begin{pmatrix}
    1& -1
\end{pmatrix}}\Z\{b^*\}\rightarrow 0.\]
Thus the cohomology is
\[H^{\gamma+n}(S^{1,1};\uZ)=\begin{cases}
    \Z& n=0\\
    0& \text{otherwise,}
\end{cases}\]
which again, agrees with Bredon cohomology.

Similarly, in $\underline{\F}_2$-coefficients
\[0\rightarrow \F_2\{b_1^*,b_0^*\}\xrightarrow{\begin{pmatrix}
    1& 1
\end{pmatrix}}\F_2\{b^*\}\rightarrow 0,\]
gives cohomology 
\[H^{\gamma+n}(S^{1,1};\underline{\F}_2)=\begin{cases}
    \F_2 & n=0\\
    0& \text{otherwise.}
\end{cases}\]

\begin{figure}
    \centering
\includegraphics[width=0.5\textwidth]{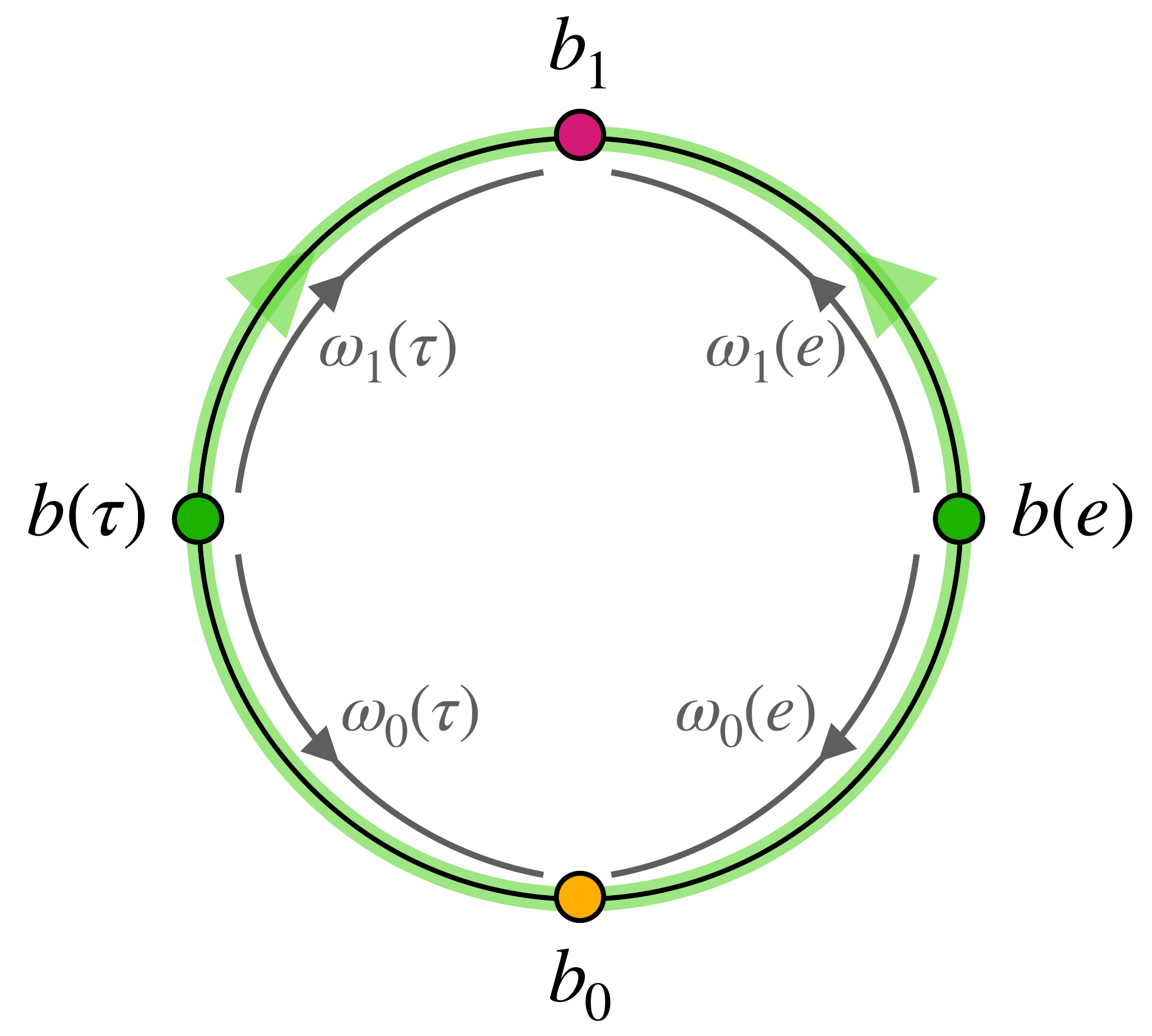}
    \caption{$\CWg$-structure for $S^{1,1}$ with $\gamma=0$}
    \label{fig:gamma=0}
\end{figure}

\bigskip

\subsubsection{$G=C_2$, $B=S^{1,1}$, and $\gamma =\R^{1,1}-1$}
There is a $\CWg$-structure on $S^{1,1}$ consisting of one $\pars{\gamma+0}$-cell
\[
b\colon C_2/e\rightarrow S^{1,1}
\]
and two $\pars{\gamma+1}$-cells
\begin{align*}
    b_0\colon C_2 \times_{C_2} D(\R^{1,1})\rightarrow S^{1,1}\\
     b_1\colon C_2 \times_{C_2} D(\R^{1,1})\rightarrow S^{1,1},
\end{align*}
with orientations as shown in \cref{fig:gamma=sigma-1}.

The nontrivial components of the connecting homomorphism are given by
\[\delta^{b_1,b}=\res_\gamma(\id,\omega_1)\circ \tr_{\gamma}(\rho,c)\]
and
\[\delta^{b_0,b}=\res_\gamma(\id,\omega_0)\circ \tr_{\gamma}(\rho,c),\]
where by abuse of notation $c$ is the constant path at $b_1$ and at $b_0$ respectively, $\omega_1$ and $\omega_0$ are as in \cref{fig:gamma=sigma-1} and $\rho\colon C_2/e\rightarrow C_2/C_2$ is again the quotient.
\begin{figure}
    \centering
    \includegraphics[width=0.5\textwidth]{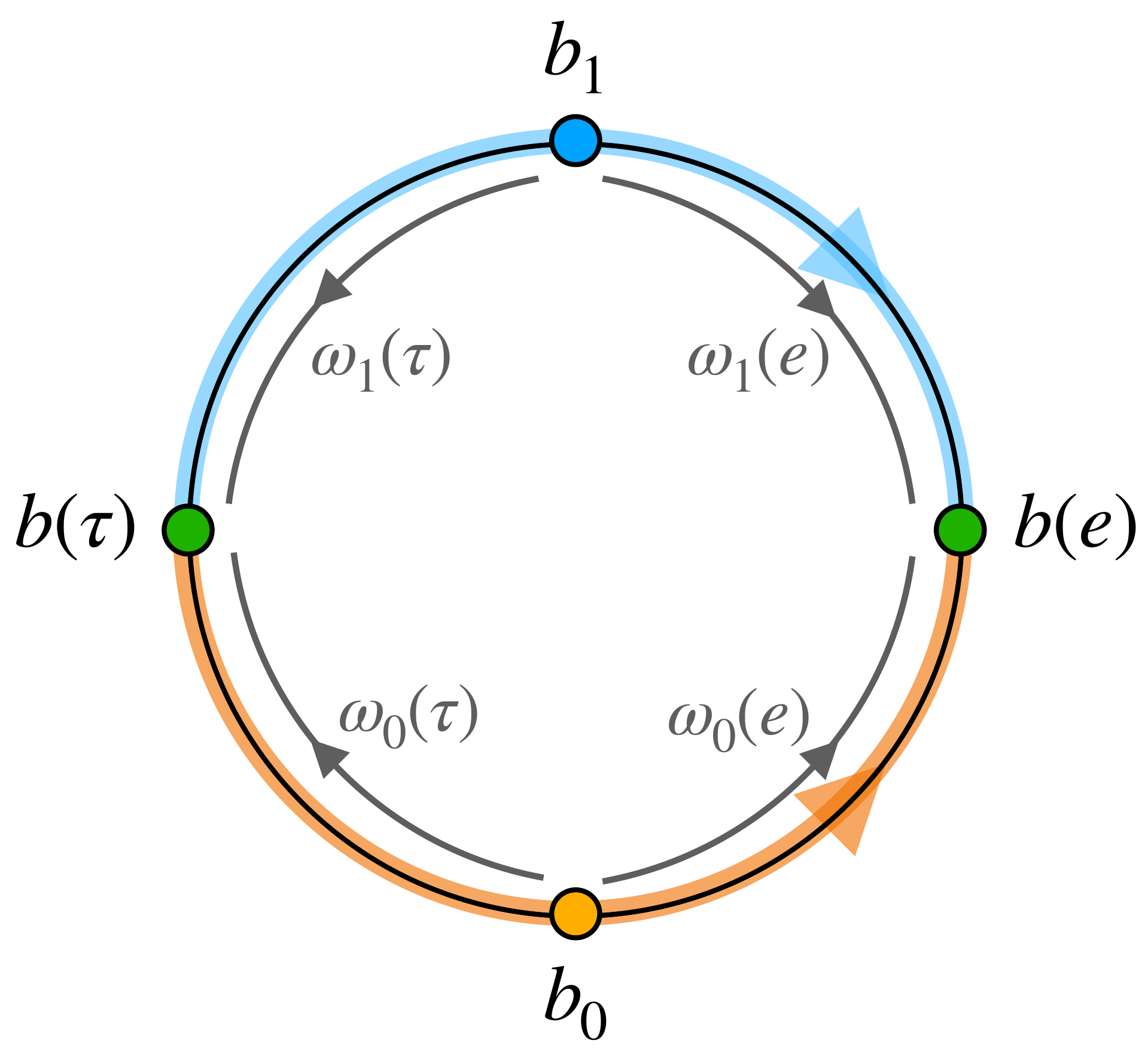}
    \caption{$\CWg$-structure for $S^{1,1}$ with $\gamma= \R^{1,1}-1$}
    \label{fig:gamma=sigma-1}
\end{figure}

So $\uZ(\Gamma_\gamma^{-1}(\delta^{b_0,b}))=\uZ(\Gamma_\gamma^{-1}(\delta^{b_1,b}))=2$
and we get the following cochain complex
\[0\rightarrow \Z\{b^*\}\xrightarrow{\begin{pmatrix}2\\2\end{pmatrix}}\Z\{b_0^*,b_1^*\}\rightarrow 0.\]
Taking cohomology, we get
\[H^{\gamma+n}(S^{1,1};\uZ)=\begin{cases}
    \Z\oplus \Z/2 & n=1\\
    0& \text{otherwise.}
\end{cases}\]
Again, this agrees with Bredon cohomology by the suspension isomorphism, since
\begin{align*} 
H^{\gamma+n}(S^{1,1} ; \uZ) &\cong H^{n,1}(S^{1,1} ; \uZ) \\
&\cong \widetilde{H}^{n,1}(S^{1,1} ; \uZ) \oplus  H^{n,1}(\pt ; \uZ) \\
&\cong  {H}^{n-1,0}(\pt; \uZ) \oplus  H^{n,1}(\pt;\uZ).  
\end{align*}

In $\underline{\F}_2$-coefficients, we have 
\[0\rightarrow \F_2\{b^*\}\xrightarrow{\begin{pmatrix}0\\0\end{pmatrix}}\F_2\{b_0^*,b_1^*\}\rightarrow 0,\]
with cohomology
\[H^{\gamma+n}(S^{1,1};\underline{\F}_2)=\begin{cases}
    \F_2  & n=0\\
\F_2 \oplus \F_2 & n=1\\
    0& \text{otherwise.}
\end{cases}\]

\bigskip

\subsubsection{$G=C_2$, $B=S^{1,1}$, and $\gamma =\taut-1$}\label{comp:S11taut} This is our first example beyond the $RO(G)$-grading. 
Recall from \cref{lem:ROpiS11} that $\RO(\Pi S^{1,1}) \cong \Z^3$, and from \cref{ex:tautS11} that
\[ \taut=(1,0,1) \in \RO(\Pi S^{1,1})\]
is the representation defined by tautological line bundle $\taut$. 
Let \[\gamma=\taut-1=(0,0,1)\in \RO(\Pi S^{1,1}).\]
We compute $H^{\gamma+*}(S^{1,1};\underline{\Z})$ and $H^{\gamma+*}(S^{1,1};\underline{\F}_2)$ with $*\in \Z$.

In \cref{rem:easiercellstructure}, we will present the cohomology computation using a simpler cell structure. We first use a more complicated cell structure to better illustrate computational techniques in the extended grading.

Recall the $\CWg$-structure on $S^{1,1}$ from \cref{ex:cellS11} with two $(\gamma+0)$-cells 
\begin{align*}
b \colon C_2/e \to S^{1,1}\\
b_0 \colon C_2/C_2 \to S^{1,1}
\end{align*}
and two $(\gamma+1)$-cells
\begin{align*}
b_1 \colon C_2\times_{C_2} D(\R^{1,1}) \to S^{1,1}\\
b_2 \colon C_2\times_e D(\R^1) \to S^{1,1}
\end{align*}
with orientations as depicted on the left in \cref{fig:CWS11}.

With the cell structure above, the cochain complex is of the form
\[0\rightarrow \Z\{b_0^*,b^*\}\xrightarrow{\uZ(\Gamma^{-1}_\gamma(\delta))}\Z\{b_1^*,b_2^*\}\rightarrow 0\]
where $\delta=\begin{pmatrix}
    \delta^{b_1,b_0}& \delta^{b_2,b_0}\\ \delta^{b_1,b}& \delta^{b_2,b}
\end{pmatrix}$
is the connecting homomorphism.

As usual, we have the quotient $\rho \colon C_2/e \to C_2/C_2$ and $c$ denotes a constant path.  We write $\omega_{b_1,b}\colon C_2/e\times I\rightarrow S^{1,1}$ for the path whose restriction to $e\times I$ is the shortest path from $b_1(C_2/C_2)$ to $b(e)$, and denote this restriction by $\omega_{b_1,b}(e)$.  We use similar notation for the paths shown in \cref{fig:CWS11}.
\begin{figure}[b]
\includegraphics[width=\textwidth]{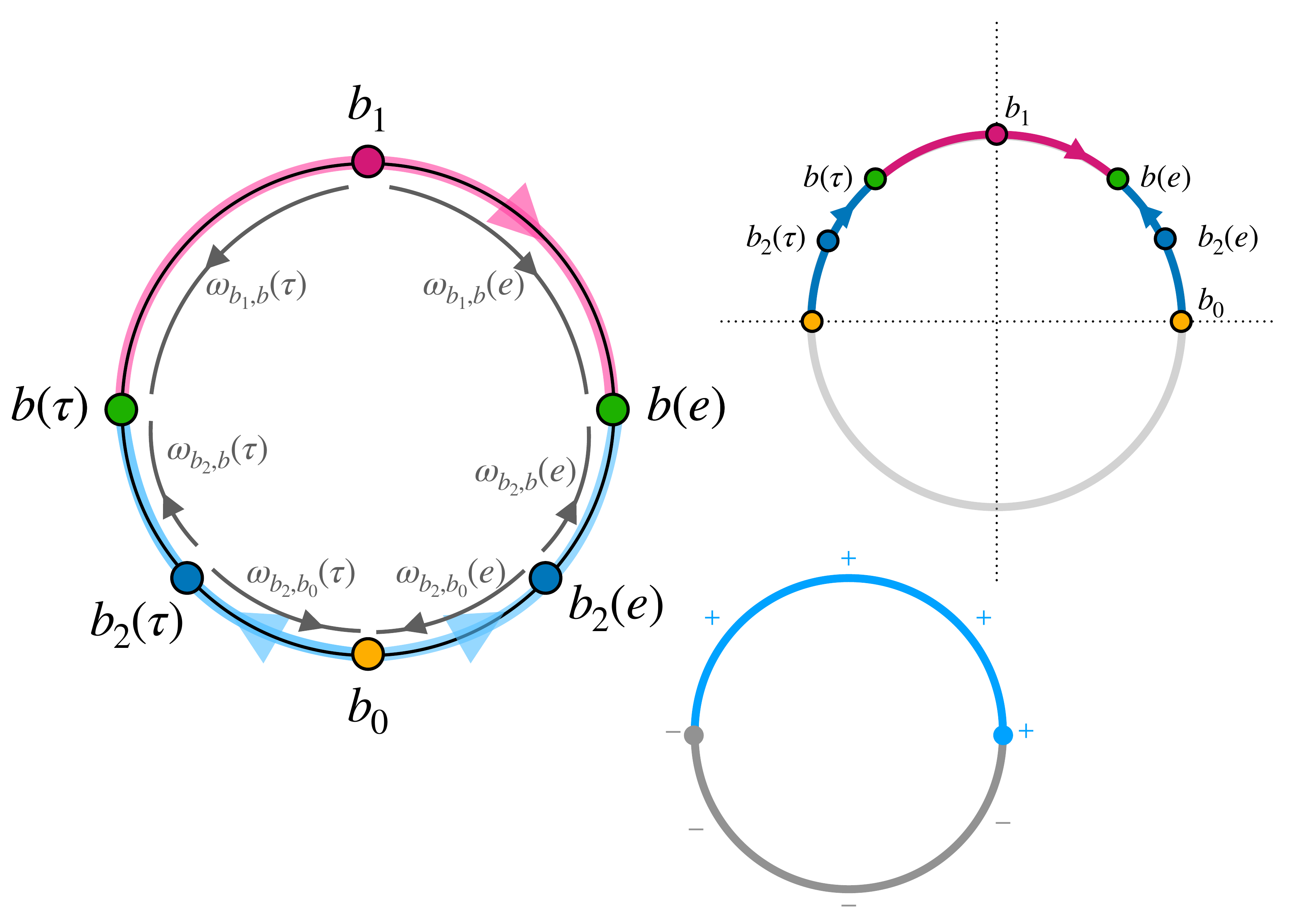}
    \caption{$\CWg$-structure for $S^{1,1}$ with $\gamma=\taut-1$}
    \label{fig:CWS11}
\end{figure}

One feature of $\taut$ is that it comes from an actual vector bundle and is not naturally in the standard form of a representation.  Thus computing the relevant signs for paths not in the skeleton of $\Pi S^{1,1}$ requires a bit more care.

The upper right picture in Figure \ref{fig:CWS11} depicts $\R^{2,1}$ and shows $S^{1,1}$ as $\mathbb{P}(\R^{2,1})$. We use this picture to understand what the representation $\gamma=\taut-1$ (or equivalently $\taut$) assigns to a morphism in $\Pi S^{1,1}$. The fibers of $\taut$ are the lines in $\R^{2,1}$ through the origin. We make an explicit choice of basis/orientation for each of these lines. We take the picture shown on the upper right and decorate it in the bottom picture to depict the orientation we have chosen, orienting most lines with the basis element in the upper half plane. This is shown in the bottom picture: the basis element on the line through a point in the circle is marked by a $+$. 
The $-$ means that this is $-1$ times the basis element. 

Now when we have a morphism $(\alpha, \omega)$ in $\Pi S^{1,1}$, 
the path $\omega(e)$ in $S^1$ lifts uniquely to a path on the circle in the bottom picture.\footnote{Uniqueness is because the restriction of the tautological line bundle to the unit vectors is a covering map.} The monodromy of $\omega(e)$ is $+1$ if the signs of the start and endpoint of $\omega(e)$ agree and $-1$ if not. For example if $\omega(e)$ is the nontrivial loop at $b_0$, then we follow the half loop in the bottom picture, which has sign $-1$.  It is always possible to choose a starting point for $\omega(e)$ with a positive basis element $+1$.
By definition of the representation $\gamma = \dim \taut-1$, the sign computed in this manner is exactly the degree of the map
\[\gamma_0(\alpha,\omega) \colon \gamma_0(x) \to \gamma_0(y),\]
on the fiber $\gamma_0(x)$ over the identity coset. 
The degree on the other fibers is determined by equivariance.\footnote{The signs on the other fibers should be computed using this equivariance. Path lifting for the other fibers is not the correct approach and will lead to sign mistakes as we will see in \cref{rem:infoextra}.}

As before, we compute $i_e^*\delta^{x,y}$ for each component $\delta^{x,y}$ of $\delta$ to find $f^{x,y}$ such that $\Gamma_\gamma(f^{x,y})=\delta^{x,y}$. 
We have
\begin{align*}
i^*_e \delta^{b_1,b} &= \begin{pmatrix} \omega_{b_1,b}(e) \\ -\omega_{b_1,b}(\tau)
\end{pmatrix} 
\end{align*}
The signs here are just determined by the degree of the attaching map.
Below, we compute
\begin{align*}
i^*_e\res_\gamma(\id,\omega) &= \begin{pmatrix} \omega_{b_1,b}(e)  & 0 \\  0 &\omega_{b_1,b}(\tau)
\end{pmatrix} & 
i^*_e\tr_{\gamma}(\rho,c) &=
\begin{pmatrix} c \\ -c
\end{pmatrix}
\end{align*}
It will follow that $i_e^*\delta^{b_1,b}=i_e^*\res_\gamma(\id,\omega_{b_1,b})\circ i_e^*\tr_\gamma(\rho,c)$. We note that although the restriction of $(\id,\omega)$ has no signs, it is the shearing in the transfer that puts in the signs required to give the right answer for $\delta^{b_1,b}$.

The sign for the transfer comes from the fact that $\gamma_0(b_1) = \R^{1,1}$ so the shearing in the transfer puts a $-1$ on the $\tau$-sphere.

For the restriction, we have to compute the map $\gamma(\id,\omega)$. We will get
\[\xymatrix@C=5pc{
\gamma(b_1\rho)\ar@{=}[d] \ar[r]^{\gamma(\id,\omega_{b_1,b_0})} & \gamma(b) \ar@{=}[d]  \\
C_2\times_e \R^1  \ar@{-->}[r]^-{\begin{pmatrix} +1 & 0 \\ 0 & +1
\end{pmatrix}} & C_2\times_e \R^1}\] 
where again we abused notation with matrix notation as described in \cref{rem:matrixrem} and for the moment consider $\gamma = \taut$ rather than $\taut -1$.
To compute the dashed arrow in the diagram above, we use path lifting for $e$-fiber and equivariance for the $\tau$-fiber. We lift $\omega_{b_1,b}(e)$ to the unit sphere and note that this path has monodromy $+1$, as it travels through the blue region. So the degree on the $e$-fiber is $+1$. The degree on the $\tau$-fiber is determined by equivariance and so is the same, $+1$.

We also have
\begin{align*}
i^*_e \delta^{b_2,b_0} &= \begin{pmatrix} -\omega_{b_2,b_0}(e) & -\omega_{b_2,b_0}(\tau)
\end{pmatrix} 
\end{align*}
and we claim that this is equal to 
\[
i^*_e \delta^{b_2,b_0} = -i_e^*\res_\gamma(\rho,\omega_{b_2,b_0}).
\]
To see this we need to argue that $\gamma(\id,\omega_{b_2,b_0})$ can be identified with the dashed arrow in the following diagram.
\[\xymatrix@C=5pc{
\gamma(b_2)\ar@{=}[d] \ar[r]^{\gamma(\id, \omega_{b_2,b_0} )} & \gamma(b_0) \ar@{=}[d]  \\
C_2\times_e \R^1  \ar@{-->}[r]^-{\begin{pmatrix} +1 & +1 
\end{pmatrix}} &  \R^{1,0}}\]
The degree on the $e$-fiber is obtained by looking at the monodromy of the path $\omega_{b_2,b_0}(e)$.  This is $+1$ according to our choices of orientations for the fibers, since the obvious lift for $\omega(e)$ travels only within the positive (blue) region. Therefore, by equivariance, 
the degree for the $\tau$-fiber is also $+1$.
\begin{warn}
\label{rem:infoextra}
The path $\omega_{b_2,b_0}(\tau)$ has monodromy $-1$, which is not what we get by equivariance in the computation above. This shows that the latter monodromy is not relevant for this computation! 
 The reason we follow the $e$-fiber instead of the $\tau$-fiber is explained in \cref{rem:dimchoices}, and has to do with our definition of $\dim(\taut)$.
\end{warn}

Finally, it is not hard to see that 
\[i_e^*\delta^{b_2,b}=\begin{pmatrix}
    \omega_{b_2,b}(e)&0\\
    0& \omega_{b_2,b}(\tau)
\end{pmatrix}=i_e^*\res_\gamma(\id,\omega_{b_2,b})\]
and that $\delta^{b_1,b_0}=0$.

So we showed that each component $\delta^{x,y}$ of $\delta$ is given by
\begin{itemize}
    \item $\delta^{b_1,b_0}=0$,
    \item $\delta^{b_1,b}=\res_\gamma(\id, \omega_{b_1,b})\circ\tr_\gamma(\rho,c)$,
    \item $\delta^{b_2,b_0}=-\res_\gamma(\rho, \omega_{b_2,b_0})$, and 
    \item $\delta^{b_2,b}=\res_\gamma(\id,\omega_{b_2,b})$. 
\end{itemize}

Applying $\uZ\circ \Gamma_\gamma^{-1}$ gives
\begin{itemize}
    \item $\uZ\pars{\Gamma_\gamma^{-1}\pars{\delta^{b_1,b_0}}} = 0$,
    \item $\uZ\pars{\Gamma_\gamma^{-1}\pars{\delta^{b_1,b}}}=2$,
    \item $\uZ\pars{\Gamma_\gamma^{-1}\pars{\delta^{b_2,b_0}}}=-1$, and
    \item $\uZ\pars{\Gamma_\gamma^{-1}\pars{\delta^{b_2,b}}}=1$,
\end{itemize}
and the cochain complex becomes
\[
0\rightarrow \Z\{b_0^*,b^*\}\xrightarrow{\begin{pmatrix}
    0& 2\\
   -1& 1
\end{pmatrix}}\Z\{b_1^*,b_2^*\}\rightarrow 0.
\]
The cohomology is then
\[H^{\gamma+n}(S^{1,1};\uZ)=\begin{cases}
    \Z/2 & n=1\\
    0& \text{otherwise.}
\end{cases}\]
Now, of course, there is no comparison with Bredon cohomology.  This lies entirely within the extended grading.

In $\underline{\F}_2$-coefficients, we have 
\[
0\rightarrow \F_2\{b_0^*,b^*\}\xrightarrow{\begin{pmatrix}
    0& 0\\
   -1& 1
\end{pmatrix}}\F_2\{b_1^*,b_2^*\}\rightarrow 0
\]
and cohomology
\[H^{\gamma+n}(S^{1,1};\underline{\F}_2)=\begin{cases}
    \F_2  & n=0, 1\\
    0& \text{otherwise.}
\end{cases}\]

\begin{rem}
    The nonzero class in degree $n=1$ above, corresponding to $H^\taut(S^{1,1};\underline{\F}_2) \cong \F_2$, is the Thom class of the tautological bundle $u_\taut$.  As observed in \cite[Example 3.4]{Hazel_fund}, this Thom class does not exist in $\RO(C_2)$-graded Bredon cohomology with $\underline{\F}_2$-coefficients. The extended grading is required to capture both Thom classes and the Thom isomorphism for arbitrary vector bundles.  Moreover, it exhibits Thom isomorphisms for any Mackey functor coefficients.  From the explicit cochain complex calculation above, we observe that we can identify $u_\taut$ with $b_1^*$. 
\end{rem}

\begin{rem}
\label{rem:easiercellstructure}
Alternatively, to compute the cohomology of $S^{1,1}$ in these same degrees, one could have used a simpler $\CWg$-structure on $S^{1,1}$. 
 For example, the one from the end of \cref{ex:cellS11} consisting of a single $\pars{\gamma+0}$-cell:
 \[b_0\colon  C_2/C_2\rightarrow S^{1,1}\]
 and a single $\pars{\gamma+1}$-cell:
 \[b_1\colon C_2 \times_{C_2} D(\R^{1,1})\rightarrow S^{1,1}\] 
oriented as shown in \cref{fig:simplerCW}. 

\begin{figure}[h]
    \includegraphics[width=\textwidth]{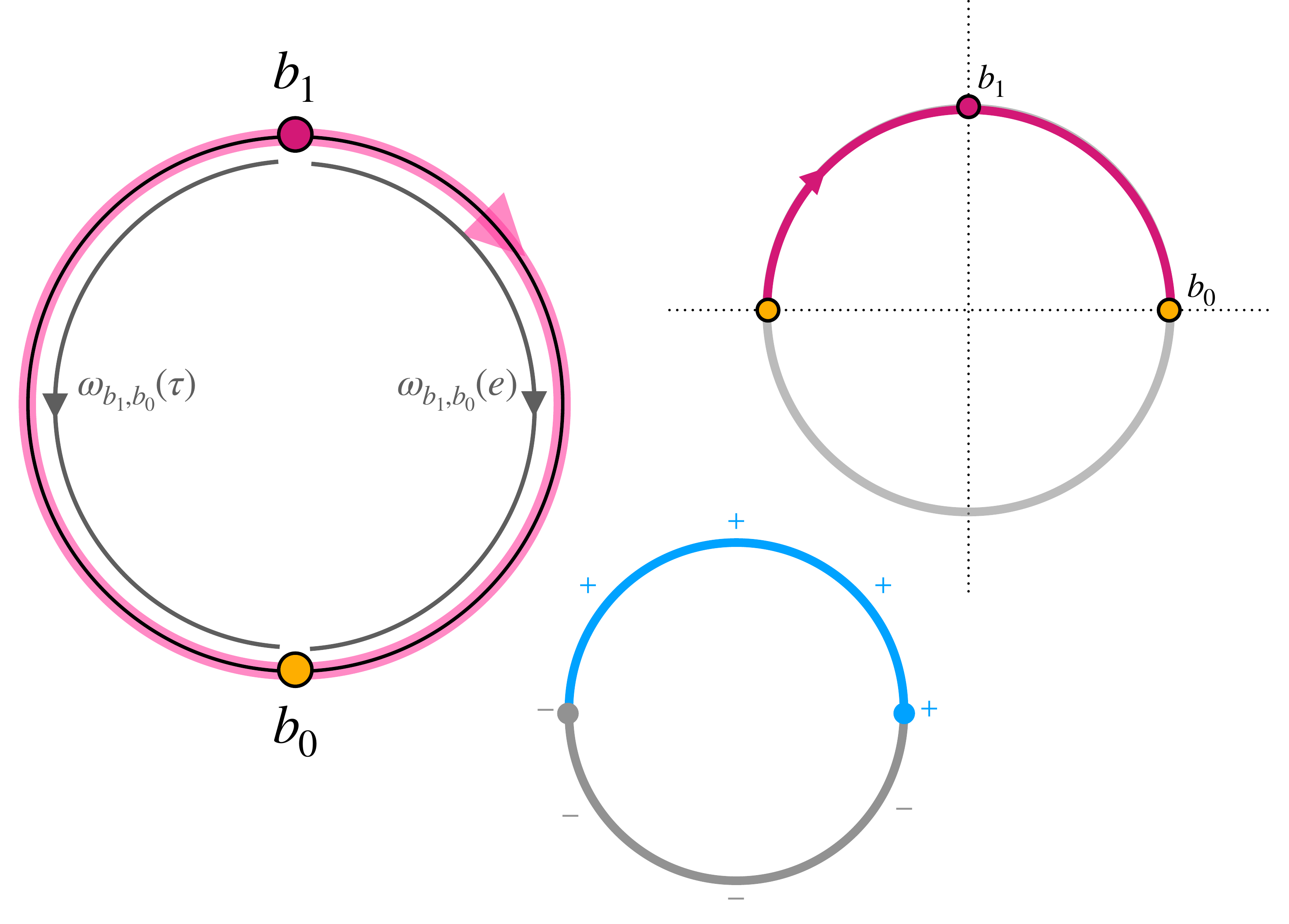}
    \caption{Simpler $\CWg$-structure}
    \label{fig:simplerCW}
\end{figure}
 
 Then the connecting homomorphism is 
\[\delta=\res_\gamma(\rho, \omega_{b_1,b_0})\circ\tr_\gamma(\rho,c),\]
where again $\rho\colon C_2/e\rightarrow C_2/C_2$ is the quotient map and $c$ is the constant path at $b_1$,
and we get the following cochain complexes with $\uZ$-coefficients:
\[0\rightarrow \Z\{b_0^*\}\xrightarrow{2}\Z\{b_1^*\}\rightarrow 0\]
and with $\underline{\F}_2$-coefficients:
\[0\rightarrow \F_2\{b_0^*\}\xrightarrow{0}\F_2\{b_1^*\}\rightarrow 0.\]
Taking cohomology agrees with the computations above.
\end{rem}

\bigskip

\subsubsection{$G=C_2$, $B=S^{1,1}$, and $\gamma=\chi\taut-1$}
For completeness, we also compute $H^{\gamma+*}(S^{1,1};\uZ)$ and $H^{\gamma+*}(S^{1,1};\underline{\F}_2)$ with $\gamma=\chi\taut-1=(0,1,0)\in \RO \Pi S^{1,1}$, where $\chi$ denotes $C_2$-action on $\RO(\Pi S^{1,1})$ induced by reflecting $S^{1,1}$ across a horizontal line through $b$.

We compute the cohomology using a simple $\CWg$-structure, analogous to the one in \cref{rem:easiercellstructure}.  Consider a cell structure consisting of a single $\pars{\gamma+0}$-cell:
 \[b_1\colon  C_2/C_2\rightarrow S^{1,1}\]
 and a single $\pars{\gamma+1}$-cell:
 \[b_0\colon C_2 \times_{C_2} D(\R^{1,1})\rightarrow S^{1,1}.\]
This yields the analogous cochain complex but with the roles of $b_0$ and $b_1$ swapped, and so we get
\[H^{\gamma+n}(S^{1,1};\uZ)=\begin{cases}
    \Z/2 & n=1\\
    0& \text{otherwise}
\end{cases}\]
and
\[H^{\gamma+n}(S^{1,1};\underline{\F}_2)=\begin{cases}
    \F_2  & n=0, 1\\
    0& \text{otherwise.}
\end{cases}\]
In this way, we can identify the Thom class $u_{\chi \taut}$ with $b_0^*$.

\bigskip

\subsection{A $C_2$-equivariant real projective space}
We give one more example for a $C_2$-space, the so-called \emph{$\R P^2$-twist}.
\subsubsection{$G=C_2$, $B=\RPtwist$, and $\gamma =\taut-1$}
Recall from \cref{lem:ROpiRP2} that  $RO(\Pi \RPtwist) \cong \Z^3 \times \Z/2$ and let
\[\taut=(1,0,1,0)\in RO(\Pi \RPtwist)\]
be the representation defined by the tautological bundle on $\RPtwist$. As usual, let
\[
\gamma=\taut-1 = (0,0,1,0) \in RO(\Pi \RPtwist)
\]
so that $\gamma$ has virtual dimension zero.
We compute  $H^{\gamma+*}(\RPtwist ;\uZ)$ and $H^{\gamma+*}(\RPtwist ;\underline{\F}_2)$ with $*\in \Z$.

There is a $\CWg$-structure on $\RPtwist$, described in Example 3.1.2(3) of \cite{CW_book}, consisting of the following:\\
one $\pars{\gamma+0}$-cell:
\[
b_0 \colon C_2/C_2\to \RPtwist,
\]
two $\pars{\gamma+1}$-cells:
\begin{align*}
    a&\colon C_2 \times_{C_2} D(\R^{1,0})\to \RPtwist\\
b_1&\colon C_2 \times_{C_2} D(\R^{1,1})\to \RPtwist,
\end{align*}
and one $\pars{\gamma+2}$-cell:
\[
b_2\colon C_2 \times_e D(\R^{2})\to \RPtwist,
\]
oriented as shown on the left in  \cref{fig:RP2TWCELL}.  In this depiction, the $C_2$-action is given by rotation about the center so that the outer circle and the isolated point $b_1$ are fixed.  In the upper right, we consider $\mathbb{R}P^2$ as $\RPtwist$.  Flattening this picture, the bottom of \cref{fig:RP2TWCELL} depicts our choices for the basis elements/orientations in $\taut$.  See \cref{comp:S11taut} for further explanation of the signs.

\begin{figure}
    \centering
    \includegraphics[width=\textwidth]{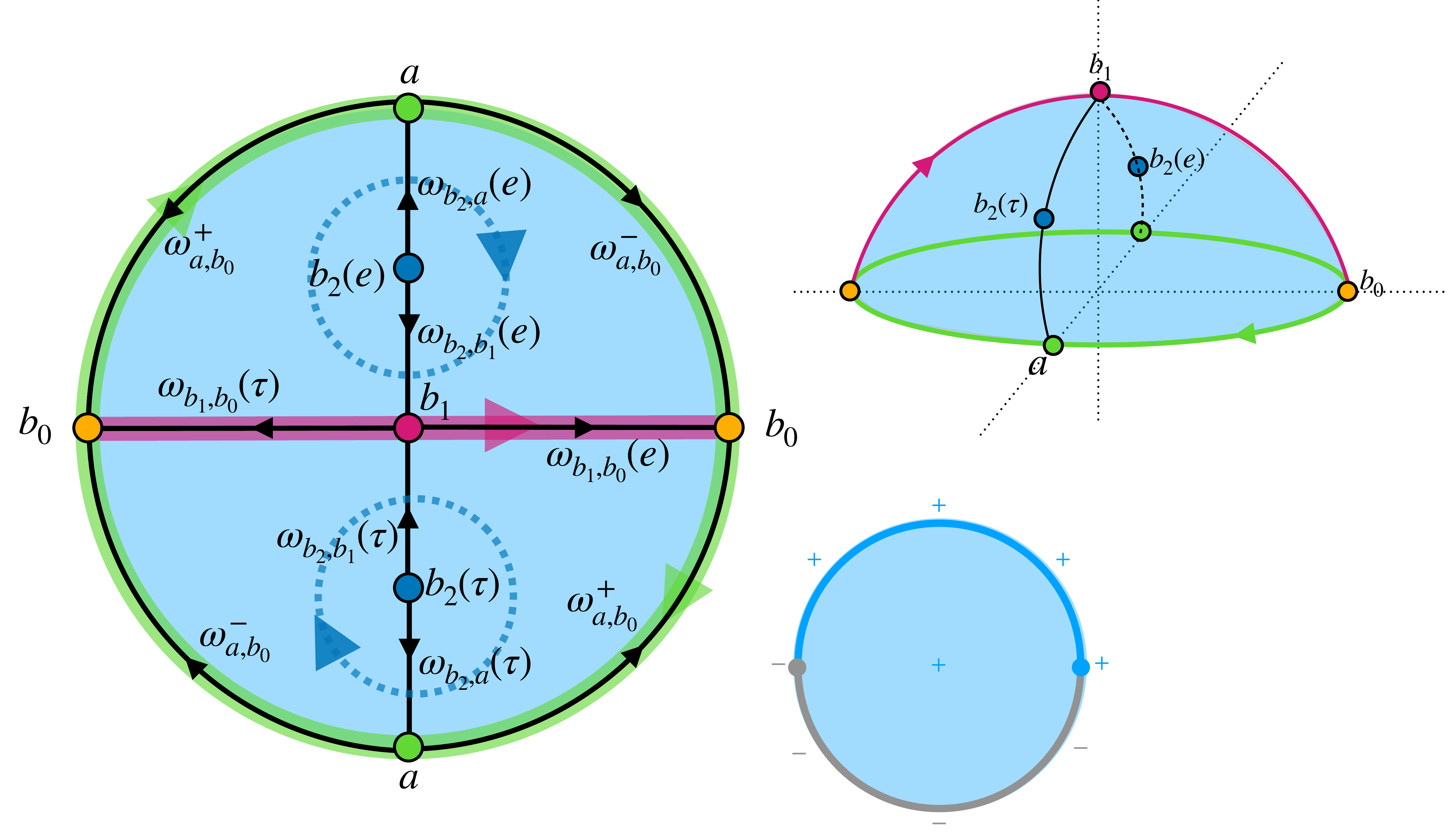}
    \caption{$\CWg$-structure for $\RPtwist$ with $\gamma=\taut-1$}
    \label{fig:RP2TWCELL}
\end{figure}

We have to compute the coboundary of the cochain complex
\[0\rightarrow\Z\{b_0^*\}\xrightarrow{\Z(\Gamma_\gamma^{-1}(\delta_{1+\gamma}))}\Z\{a^*,b_1^*\}\xrightarrow{\Z(\Gamma_\gamma^{-1}(\delta_{2+\gamma}))}\Z\{b_2^*\}\rightarrow 0.\]
The components of the connecting homomorphisms are given by 
\begin{itemize}
\item $\delta^{a,b_0}=\res_\gamma(\id,\omega_{a,b_0}^-) +\res_\gamma( \id,\omega_{a,b_0}^+)$, 
\item $\delta^{b_1,b_0}=\res_\gamma(\rho,\omega_{b_1,b_0})\circ \tr_\gamma(\rho,c)$, 
 
    \item $\delta^{b_2,a}=\res_\gamma(\rho, \omega_{b_2,a})$, and
    \item $\delta^{b_2,b_1}=-\res_\gamma (\rho,\omega_{b_2,b_1})$.
\end{itemize}

We explain our answer for $\delta^{a,b_0}$ since we found this one to be a possible source of confusion. We have
\begin{align*}
i^*_e\delta^{a,b_0} &= \begin{pmatrix} \omega_{a,b_0}^{-} - \omega_{a,b_0}^+
\end{pmatrix} \\
&= i^*_e \res_\gamma(\id, \omega_{a,b_0}^{-})+ i^*_e \res_\gamma(\id, \omega_{a,b_0}^{+}).
\end{align*}
To explain the signs, note that the path $\omega_{a,b_0}^{-}$ in the fixed-set has no monodromy: the lift which starts at a positive basis element travels entirely through the positive (blue) region. However, $\omega_{a,b_0}^{+}$ has monodromy, and this accounts for the change in sign from the first equality to the second.

Since this is the first two-cell we compute with, we also explain our answer for $\delta^{b_2,a}$ and $\delta^{b_2,b_1}$. We have
\begin{align*}
i^*_e\delta^{b_2,a} = \begin{pmatrix}
+\omega_{b_2,a}(e) & +\omega_{b_2,a}(\tau)
\end{pmatrix} = i^*_e\res_\gamma(\rho, \omega_{b_2,a}).
\end{align*}
The positive signs in the first matrix come from the fact that the orientation of the two (blue) 2-cells with center $b_2(e)$ and $b_2(\tau)$ agrees with the orientation of 1-cell with center $a$ (green). For the signs in the second part of the equation, we compute
\[\xymatrix@C=5pc{
\gamma(b_2)\ar@{=}[d] \ar[r]^{\gamma(\rho, \omega_{b_2,a})} & \gamma(a) \ar@{=}[d]  \\
C_2\times_e \R^1  \ar@{-->}[r]^-{\begin{pmatrix} +1 & +1 
\end{pmatrix}} &  \R^{1,0}}\]
For the $e$-fiber, using the positive basis as starting point, we use the fact that the path $\omega_{b_2,a}(e)$ lifts to a unique path in the two sphere that only travels within the positive (blue) region. So the degree of $\gamma(\rho, \omega_{b_2,a})$ on the $e$-fiber is $+1$. For the $\tau$-fiber, equivariance and the fact that $C_2$ acts trivially on $\R^{1,0}$ gives $+1$ as well.

For $\delta^{b_2,b_1}$, we have 
\begin{align*}
i^*_e\delta^{b_2,b_1} = \begin{pmatrix}
-\omega_{b_2,b_1}(e) & +\omega_{b_2,b_1}(\tau)
\end{pmatrix} = -i^*_e\res_\gamma(\rho, \omega_{b_2,b_1}),
\end{align*}
where now the signs in the first matrix come from the geometric attaching maps, and the choice of sign in the second equality
is determined by the following computation
\[\xymatrix@C=5pc{
\gamma(b_2)\ar@{=}[d] \ar[r]^{\gamma(\rho,\omega_{b_2,b_1})} & \gamma(b_1) \ar@{=}[d]  \\
C_2\times_e \R^1  \ar@{-->}[r]^-{\begin{pmatrix} +1 & -1 
\end{pmatrix}} &  \R^{1,1}}\]
The path $\omega_{b_2,b_1}(e)$ travels through the blue region and so has positive monodromy. Hence the degree of the map on the $e$-fiber is $+1$. However, the action of $C_2$ on $\R^{1,1}$ reverses orientation and so implies that the degree on the $\tau$-fiber is $-1$.

So we have verified
\begin{itemize}
\item $\delta^{a,b_0}=\res_\gamma(\id,\omega_{a,b_0}^-) +\res_\gamma( \id,\omega_{a,b_0}^+)$, 
\item $\delta^{b_1,b_0}=\res_\gamma(\rho,\omega_{b_1,b_0})\circ \tr_\gamma(\rho,c)$, 
    \item $\delta^{b_2,a}=\res_\gamma(\rho, \omega_{b_2,a})$, and
    \item $\delta^{b_2,b_1}=-\res_\gamma (\rho,\omega_{b_2,b_1})$.
\end{itemize}
Finally, we get 
\[0\rightarrow\Z\{b_0^*\}\xrightarrow{\begin{pmatrix}
    2\\2
\end{pmatrix}}\Z\{a^*,b_1^*\}\xrightarrow{\begin{pmatrix}
    1&-1
\end{pmatrix}}\Z\{b_2^*\}\rightarrow 0\]
which yields the following cohomology groups
\[H^{\gamma+n}(\mathbb{P}(\R^{3,1});\uZ)=\begin{cases}
    \Z/2& n=1\\
    0&\text{otherwise.}
\end{cases}\]

Taking $\underline{\F}_2$-coefficients, the cochain complex is given by
\[0\rightarrow\F_2\{b_0^*\}\xrightarrow{\begin{pmatrix}
    0\\0
\end{pmatrix}}\F_2\{a^*,b_1^*\}\xrightarrow{\begin{pmatrix}
    1&1
\end{pmatrix}}\F_2\{b_2^*\}\rightarrow 0\]
and
\[H^{\gamma+n}(\mathbb{P}(\R^{3,1});\underline{\F}_2)=\begin{cases}
    \F_2& n=0,1\\
    0&\text{otherwise.}
\end{cases}\]

\subsection{A $C_4$-equivariant real projective space}

We finish by looking at one example in the case when $G=C_4=\langle \tau\rangle$.

\subsubsection{$G=C_4$, $B=\mathbb{P}(\lambda+1)$, and $\gamma=\taut-1$}\label{sec:exC4}
Let $\lambda$ be the $2$-dimensional representation of $C_4$ where the generator $\tau$ acts by rotation by $90$ degrees on $\R^2$. 
Let $B=\mathbb{P}(\lambda+1)$ and $\taut$ be the tautological bundle over $B$. We compute the parametrized equivariant cohomology of $B$ in degrees $\gamma+n$ for $\gamma=\taut-1$.

There is a $\CWg$-structure on $B$ consisting of one  $\pars{\gamma+0}$-cell:
\[
b_0\colon C_4/C_4\rightarrow B,
\]
one $\pars{\gamma+1}$-cell:
\[
b_1\colon C_4\times_{C_2}D({\R^{1,1}})\rightarrow B,
\]
and $\pars{\gamma+2}$-cell:
\[
b_2\colon C_4\times_{C_2}D(\R^{2,1})\rightarrow B,
\]
as shown on the left in \cref{fig:C4example}, and recalling our notation for $C_2$-representations.

\begin{figure}
    \includegraphics[width=\textwidth]{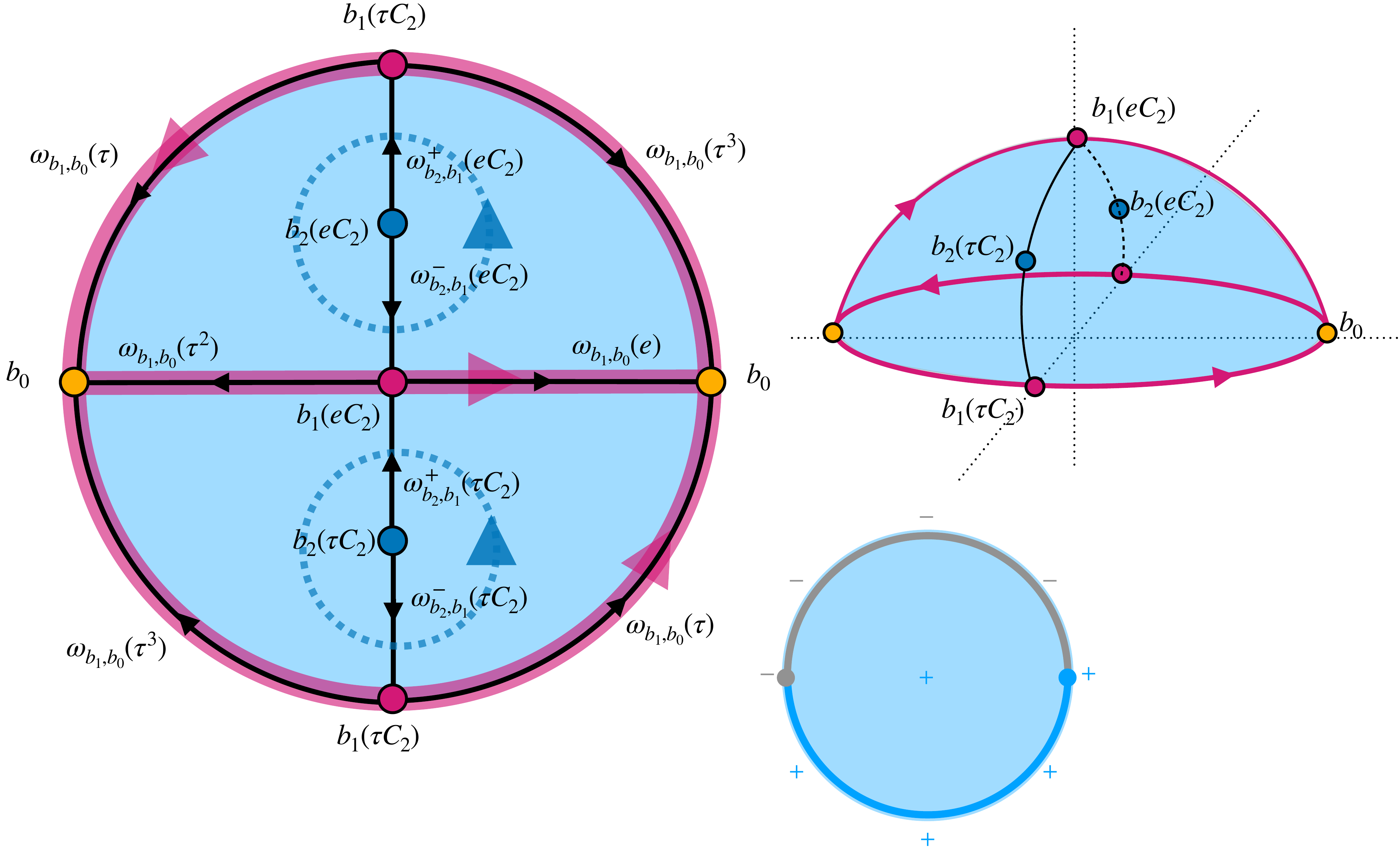}
    \caption{$\CWg$-structure for $\mathbb{P}(\lambda+1)$ with $\gamma=\taut-1$. In the upper right picture, the $C_4$-action rotates $\R^3$ around the axis containing $b_0$, rotating the north pole towards the front.}
    \label{fig:C4example}
\end{figure}

In order to be able to discuss degrees of maps coherently we need to fix coset representatives as well as a choice of basis for the fibers of $\gamma$.  For example, $\gamma(b_1)$ has two copies of $\R^{1,1}$.  To identify a basis for each copy, we need to choose coset representatives.  
If $|H|=h$, then we choose coset representatives for $C_4/H$ to be $e, \tau, \cdots, \tau^{h-1}$ and then identify
\[G/H \times \gamma_0(x) \xrightarrow{\cong} G\times_H \gamma_0(x)\]
via the map
\[(\tau^i H,v) \mapsto [\tau^i,v].\]
Then a choice of orientation for $\gamma_0(x)$ will orient each fiber of $G\times_H \gamma_0(x)$.
For example, if $H=C_2 =\langle \tau^2\rangle$ and $\gamma_0(x) = \R^{1,1}$, our coset representatives are $e,\tau$ and 
\begin{align*}
   ( eC_2,v) &\mapsto  [e,v] \\
    ( \tau C_2,v) &\mapsto  [\tau,v].
\end{align*}
If instead we had chosen $e,\tau^3$ as representatives, we would have had
\begin{align*}
   ( eC_2,v) &\mapsto  [e,v] \\
    ( \tau^3 C_2,v) &\mapsto  [\tau^3,v]=[\tau,\tau^2v]=[\tau,-v] = -[\tau,v].
\end{align*}
So we see that the different choices of coset representatives give different signs.

In the upper right, we consider $\mathbb{R}P^2$ as $\mathbb{P}(\lambda+1)$.  Flattening this picture, the bottom of \cref{fig:C4example} depicts our choices for the basis elements/orientations in the fibers of $\taut$.  See again \cref{comp:S11taut} for further explanation of the signs.

We are now ready to compute the connecting homomorphisms.  The first is given by
\[\delta^{b_1,b_0}= \res_\gamma(\rho, \omega_{b_1,b_0})\circ \tr_\gamma(\rho',c),\]
where $\rho'\colon C_4/e\rightarrow C_4/C_2$, $\rho \colon C_4/e \to C_4/C_4$ and $\omega_{b_1,b_0}$ is as in  \cref{fig:C4example}.
 We compute this as follows: 
Note that $i_e^*\delta^{b_1,b_0}$ consists of two morphisms in $\Pi i_e^*B$, namely one defined by a path from $b(eC_2)$ to $b_0$ and one from $b(\tau C_2)$ to  $b_0$. So we want to write $i_e^*\delta^{b_1,b_0}$ as a $1\times2$-matrix.
Observe that with this notation
\[i_e^*\delta^{b_1,b_0}=\begin{pmatrix}
    \omega_{b_1,b_0}(e)-\omega_{b_1,b_0}(\tau^2)& \omega_{b_1,b_0}(\tau)-\omega_{b_1,b_0}(\tau^3).
\end{pmatrix}\]
We claim that this equals
\[i_e^*\res_\gamma(\rho,\omega_{b_1,b_0})\circ i_e^* \tr_\gamma(\rho',c).\]
Indeed,
\[i_e^*\res_\gamma(\rho,\omega_{b_1,b_0})=\begin{pmatrix}
    \omega_{b_1,b_0}(e)&\omega_{b_1,b_0}(\tau)& \omega_{b_1,b_0}(\tau^2)&
\omega_{b_1,b_0}(\tau^3)
\end{pmatrix}\]
since
\[\xymatrix@C=8pc{
\gamma(b_1\rho')\ar@{=}[d] \ar[r]^{\gamma(\rho, \omega_{b_1,b_0})} & \gamma(b_0) \ar@{=}[d]  \\
C_4\times_e \R^1  \ar@{-->}[r]^-{\begin{pmatrix} +1 & +1 & +1 & +1 
\end{pmatrix}} &  \R^1}\]
Here, again the degree on the $e$-fiber is determined by the monodromy, it is $+1$ since the lift remains in the positive (blue) region. The degrees of the other fibers are determined by equivariance, but since $\tau$ acts trivially on $\gamma(b_0)=\R^1$, they are all equal to $+1$.

For the transfer, we are computing a map
\[ {C_4}_+ \wedge_{C_2} S^{\gamma_0(b_1),b_1} \to  {C_4}_+ \wedge_{e} S^{\gamma_0(b_1\rho),b_1\rho}  \]
which, using that $\gamma_0(b_1)=\R^{1,1}$ and $\gamma_0(b_1\rho')=\R^1$, comes from applying $C_4 \times_{C_2}(-)$ to the composite
\[ S^{1,1} \xrightarrow{\mathrm{collapse}} (C_2)_+ \wedge S^{1,1}  \xrightarrow{\mathrm{shear}} (C_2)_+ \wedge_e S^{1} 
\]
and gluing to $B$.
So we get
\[i_e^*\tr_\gamma(\rho',c)=\begin{pmatrix}
    c(e)& 0\\
    0& c(\tau)\\
    -c(\tau^2)& 0\\
    0&-c(\tau^3)
\end{pmatrix}.\]
Composing with $i_e^*\res_\gamma(\rho, \omega_{b_1,b_0})$ gives
\[\delta^{b_1,b_0}= \res_\gamma(\rho, \omega_{b_1,b_0})\circ \tr_\gamma(\rho',c),\]
as claimed above.

It remains to compute $\delta^{b_2,b_1}$.
For this we first compute $\gamma(\id,\omega_{b_2,b_1}^-)$. We use path lifting on the $e$-fiber and equivariance to get
\[\xymatrix@C=5pc{
\gamma(b_2)\ar@{=}[d] \ar[r]^{\gamma(\id,\omega_{b_2,b_1}^-)} & \gamma(b_1) \ar@{=}[d]  \\
C_4\times_{C_2} \R^{1,1} \ar@{-->}[r]^-{\begin{pmatrix} +1 & 0 \\
0&+1
\end{pmatrix}} &  C_4\times_{C_2}\R^{1,1}.}\]
Indeed, $\omega_{b_2,b_1}^-(e)$ travels only within the positive (blue) region so
\[ \gamma(e,\omega_{b_2,b_1}^-)([e,v]) = [e, v].\]
Then from equivariance we get that
\begin{align*}
\gamma(e,\omega_{b_2,b_1}^-)([\tau,v]) &= \tau \gamma(e,\omega_{b_2,b_1}^-)([e,v]) \\
&= \tau[e,v] \\
&= [\tau,v].
\end{align*}
For $\gamma(\tau,\omega_{b_2,b_1^+})$, we get
\[\xymatrix@C=5pc{
\gamma(b_2)\ar@{=}[d] \ar[r]^{\gamma(\tau,\omega_{b_2,b_1}^+)} & \gamma(b_1) \ar@{=}[d]  \\
C_4\times_{C_2} \R^{1,1}  \ar@{-->}[r]^-{\begin{pmatrix} 0 & +1 \\
-1&0
\end{pmatrix}} &  C_4\times_{C_2}\R^{1,1}
}\]
To see this, note that the $-1$ entry is the monodromy of the path $\omega_{b_2,b_1}^+(e)$, which travels from the positive (blue) region but ends on the negative (gray) region. It follows that 
\[ \gamma(\tau,\omega_{b_2,b_1}^+)([e,v]) = [\tau, -v].\]
Then from equivariance we get that
\begin{align*}
\gamma(\tau,\omega_{b_2,b_1}^+)([\tau,v]) &= \tau \gamma(\tau,\omega_{b_2,b_1}^+)([e,v]) \\
&= \tau[\tau,-v] \\
&= [\tau^2,-v] \\
&= [e,v] .
\end{align*}
This explains the $+1$ entry in the matrix.
 Now, inspecting the degrees of the attaching maps, it is easy to see that 
\begin{align*}i_e^*\delta^{b_2,b_1}&=\begin{pmatrix}
    \omega^-_{b_2,b_1}(eC_2)& 0\\
    0 &\omega^-_{b_2,b_1}(\tau C_2)
\end{pmatrix}+ \begin{pmatrix}
    0&-\omega^+_{b_2,b_1}(\tau C_2)\\
    \omega^+_{b_2,b_1}(e C_2)& 0
\end{pmatrix}\\
&=i_e^*\res_\gamma(\id,\omega_{b_2,b_1}^-)-i_e^*\res_\gamma(\tau, \omega_{b_2,b_1}^+),
\end{align*}
and thus 
\[\delta^{b_2,b_1}=\res_\gamma(\id,\omega_{b_2,b_1}^-)-\res_\gamma(\tau,\omega_{b_2,b_1}^+).\]

So we get the following cochain complex
\[0\rightarrow \Z\{b_0^*\}\xrightarrow{2}\Z\{b_1^*\}\xrightarrow{0} \Z\{b_2^*\}\rightarrow 0\]
which yields the following cohomology groups
\[H^{\gamma+n}(\mathbb{P}(\lambda+1);\uZ)=\begin{cases}
    0 & n=0\\
    \Z/2 & n=1\\
    \Z & n=2.
\end{cases}\]
Taking $\underline{\F}_2$-coefficients we get the cochain complex
\[0\rightarrow \F_2\{b_0^*\}\xrightarrow{0}\F_2\{b_1^*\}\xrightarrow{0} \F_2\{b_2^*\}\rightarrow 0\]
and 
\[H^{\gamma+n}(\mathbb{P}(\lambda+1);\underline{\F}_2)=\begin{cases}
    \F_2 & n=0,1,2\\
    0 & \text{otherwise.}
\end{cases}\]